\documentclass[11pt, reqno]{amsart}
\usepackage{amssymb,amsfonts}
\usepackage[margin=1in]{geometry}
\usepackage{amsthm, scrextend}
\usepackage{fancyhdr}
\usepackage{qtree} 
\usepackage{pgfplots}
\usepackage{amsmath}
\usepackage{eqnarray}
\usepackage[mathscr]{euscript}
 \let\mathscr\relax
\usepackage[scr]{rsfso}

\usepackage[all,arc]{xy}
\usepackage{enumerate}
\usepackage{mathrsfs}
\numberwithin{equation}{section}
\usepackage[mathscr]{euscript}
\usepackage{multicol}
\pgfplotsset{compat=1.14}

\theoremstyle{plain} 
\newtheorem{thm}{Theorem}[section]
\numberwithin{thm}{section}
\newtheorem{cor}[thm]{Corollary}
\newtheorem{prop}[thm]{Proposition}

\theoremstyle{definition}
\newtheorem{defn}[thm]{Definition}

\newtheorem{exmp}[thm]{Example}

\theoremstyle{remark}
\newtheorem{rem}[thm]{Remark}

\newcommand{\NN}{{\mathbb N}}
\newcommand{\ZZ}{{\mathbb Z}}

\newcommand{\ffD}{{\mathfrak D}}
\newcommand{\fF}{{\mathfrak{F}}}
\newcommand{\sA}{{\mathcal A}}

\newcommand{\sB}{{\mathcal B}}
\newcommand{\sC}{{\mathcal C}}

\newcommand{\sN}{{\mathcal N}}

\newcommand{\sP}{{\mathcal P}}

\newcommand{\sS}{{\mathcal S}}

\newcommand{\sW}{{\mathcal W}}
\newcommand{\sY}{{\mathcal Y}}

\newcommand{\msrA}{{\mathcal{A}}}

\newcommand{\bw}{{\bf w}}
\newcommand{\bx}{{\bf x}}

\newcommand{\by}{{\bf y}}
\newcommand{\bz}{{\bf z}}

\newcommand{\sX}{{X}}

\newcommand{\pr}{{\circledast}}
\newcommand{\prn}{{\circledast_n}}
\newcommand{\topp}{{\rm top}}
\newcommand{\lcm}{{\rm lcm}}

\newcommand{\IC}{{\mathcal{N}}}
\newcommand{\oX}{{\widetilde{X}}} 
\newcommand{\oz}{{\tilde{\bf{z}}}} 
\newcommand{\AP}{{\rm AP}} 
\newcommand{\self}{{\rm self}}

\usepackage{times}
\usepackage[T1]{fontenc}
\usepackage{mathrsfs}
\usepackage{epsfig}
\usepackage{color}
\makeatletter
\@namedef{subjclassname@2020}{%
  \textup{2020} Mathematics Subject Classification}
\makeatother

\numberwithin{equation}{section}

\title{ Decimation and Interleaving Operations in One-Sided Symbolic Dynamics}

\author{William C.  Abram}
\address{Allen Bailey \& Associates,   
Austin, TX 78731}
\email{abramwc@gmail.com}
\author{ Jeffrey C. Lagarias}
\thanks{The research of the second author was supported by NSF grant
DMS-1701229, and a 2018 Simons Fellowship in Mathematics .}
\address{Department of Mathematics, University of Michigan, 
Ann Arbor, MI 48109-1043,USA}
\email{lagarias@umich.edu}
\author{Daniel J. Slonim}
\address{Department of Mathematics, Purdue University, 
West Lafayette, IN 47906}
\email{dslonim@purdue.edu}
\date{November 9,  2020}
\begin{document}

\begin{abstract}

This paper studies subsets  of one-sided shift spaces 
on a  finite alphabet. 
Such subsets arise in symbolic dynamics, in fractal constructions, and in number theory.
We study a family of decimation operations, which extract
 subsequences of  symbol sequences in infinite arithmetic progressions, and show  
 they  are closed under composition. 
We also study a family of $n$-ary interleaving operations, one  for each $n \ge 1$.
Given subsets $X_0, X_1, ..., X_{n-1}$ of the shift space,  the $n$-ary interleaving operator
produces a set whose elements 
combine individual  elements $\bx_i$, one from each $X_i$, by interleaving their symbol sequences  cyclically
in arithmetic progressions $(\bmod\,n)$. 
We determine algebraic relations between decimation
and interleaving operators and the shift operator. 
We study set-theoretic $n$-fold closure operations $X \mapsto X^{[n]}$,
which interleave decimations of $X$ of modulus level  $n$.  
A set is $n$-factorizable if  $X=X^{[n]}$. The $n$-fold interleaving  operators are
 closed under composition and are idempotent. 
To each $X$ we assign the set $\sN(X)$
of all values $n \ge 1$ for which $X= X^{[n]}$. We characterize the
possible sets  $\sN(X)$ 
as  nonempty sets of positive integers that 
 form a distributive lattice under the divisibility partial order and are downward closed under divisibility.
 We show that all  sets of this type occur. 
 We introduce a class of weakly shift-stable sets and show that this class is closed under
all decimation,  interleaving, and shift operations. 
We study  two  notions of entropy for subsets of the full one-sided shift and show that they coincide for 
weakly shift-stable $X$, but can be different in general.
 We give  a formula  for  entropy of interleavings of  weakly shift-stable sets
in terms of individual entropies.  
\end{abstract}

\maketitle

\tableofcontents

\section{Introduction}\label{sec:1}

Let $\mathcal{A}$ be a finite alphabet of symbols, and suppose $|\sA| \ge 2$.  A basic object in  one-sided symbolic dynamics
is the    full one-sided shift  space $\mathcal{A}^\mathbb{N}$,  
which is  the space of 
all one-sided infinite strings of symbols drawn from $\sA$.
 Here $\NN=\{0, 1, 2,\ldots\}$ denotes the natural numbers, and  $\NN^{+}= \NN \smallsetminus \{0\}$ denotes
 the positive integers.  We view $\sA^{\NN}= \prod_{j \in \NN} \sA$ as a compact topological space carrying the product
topology, with each copy of $\sA$ carrying the discrete topology; we  call this topology of $\sA^{\NN}$ the {\em symbol topology}.
 The dynamics in one-sided symbolic dynamics is  the action of the 
 {\em (one-sided)  shift operator}  $S: \sA^{\NN} \to \sA^{\NN}$ on  individual symbol sequences $\bx= a_0a_1a_2a_3 \cdots$ by 
\begin{equation}
S(\bx) := a_1a_2a_3a_4 \cdots. 
\end{equation} 
In contrast, two-sided symbolic dynamics  (treated in Lind and Marcus \cite{LM95})  uses  the two-sided
shift operator $S: \sA^{\ZZ} \to \sA^{\ZZ}$ with $S((a_i)_{i \in \ZZ}) = (b_i)_{i \in \ZZ}$ with $b_i= a_{i+1}$.
It focuses on  sets $X \subseteq \sA^{\ZZ}$  that are  invariant under the  (two-sided) shift operator: $SX= X$.  
Such sets arise as discretizations of continuous
dynamical systems such as geodesic flow, 
and led to  the original formulation of symbolic dynamics
by  Morse and Hedlund \cite{MH:38}. 
 In one-sided symbolic dynamics on subsets of  $\sA^{\NN}$ the  spaces $X$ can encode initial conditions. 
 Initial conditions can break shift-invariance, so  it is  natural 
to consider  spaces that are stable under the shift: $SX \subseteq X$.

This paper studies the action of decimation and  interleaving operations acting on sets $X$ in the framework of symbolic dynamics and coding theory. 
Decimation operations are important in digital signal processing  and coding theory,
and interleaving operations form a kind of inverse operation to them, see \eqref{eqn:basic}. 

\begin{enumerate}
\item[(1)] 
 At the level of individual symbol sequences, the 
 {\em $j$th decimation operation at level $n$}, for $ i \ge 0$ and $n \ge 1$, denoted $\psi_{i,n}: \sA^{\NN} \to \sA^{\NN}$, for an individual 
symbol sequence  $\bx= a_0a_1a_2a_3 \cdots$ is
\begin{equation} 
 \psi_{i,n}(\bx) :=  a_i a_{i+n}a_{i + 2n}a_{i+3n}  \cdots.
\end{equation}
This operator extracts symbol subsequences  having  indices in an  arithmetic progression given by $i \, (\bmod \, n)$, starting at initial  index $j$. 
\item[(2)]
The  {\em $n$-fold interleaving operation} $\pr_n:  \sA^{\NN} \times \sA^{\NN} \times \cdots \times \sA^{\NN} \to \sA^{\NN}$ 
is an $n$-ary operation whose action on  $n$ individual symbol sequences 
$\bx_{i} = a_{i,0} a_{i,1}a_{i,2} \cdots $ for $0 \le i \le n-1$ 
is defined by  
\begin{equation}
(\bx_0, \bx_1, \cdots, \bx_{n-1}) \mapsto  \by :=
(\prn)_{i=0}^{n-1}  \bx_i = \bx_0 \pr \bx_1 \pr \cdots \pr \bx_{n-1}= b_0b_1b_2 \cdots
\end{equation} 
in which the output  sequence $\by := b_0b_1b_2\ldots$  interleaves the  symbols in arithmetic progressions of
symbol indices $(\bmod \, n)$, so that   
$$b_{i+jn}= a_{i, j}  \quad \mbox{for} \quad j \ge 0, \,\,  0 \le i \le n-1.$$
That is, the output $\by$  has  in its symbol positions  
 $i \, (\bmod\, n)$   the symbols of $\bx_i$  given in order. 
\end{enumerate}
Decimation and interleaving  operations defined pointwise extend  by set union to define set-valued operators acting on arbitrary 
subsets $X$ of $\sA^{\NN}$ (resp. of $(\sA^{\NN})^n$). For examples, see Sections \ref{subsec:2001} and \ref{subsec:2002}. 

All individual symbol sequences $\bx$ are constructible  as  $n$-fold interleavings
of suitable decimations:
\begin{equation}\label{eqn:basic}
\bx= (\prn)_{i=0}^{n-1} \psi_{i, n}(\bx) \quad \mbox{for} \quad \bx \in \sA^{\NN},
\end{equation}
see Section \ref{subsec:41a}. 

\subsection{Summary}\label{sec:12}

This paper treats  two topics.

\subsubsection{}The first  topic  studies   algebraic properties of the algebraic structure of decimation and interleaving operators under composition. The set of all decimation operations is closed under composition, and the decimation and shift actions are compatible in a sense we describe in Section 3.
Decimation operators are closed under composition.

We define the {$n$-fold interleaving closure} $X^{[n]}$ of a set $X$ in Section \ref{sec:main_results} 
as $X^{[n]} = (\pr_n)_{i=0}^{n-1} \psi_{i,n}(X),$ an operation that combines both decimations and interleavings.  
We show the  operation sending $X$ to $X^{[n]}$ is a set-valued closure operation in the Moore sense, in particular $X \subseteq X^{[n]}$. 
 A main result is  that  interleaving closure operators  under composition satisfy 
 \begin{equation}\label{eqn:compose_cc} 
(X^{[m]})^{[n]}= (X^{[n]})^{[m]} = X^{ [\lcm(m,n)]},
\end{equation} 
where $\lcm(m,n)$ denotes the least common multiple of $m$ and $n$. Thus
these operators are closed under composition,
commute under composition, and  are idempotent. 

We show  $X$ is closed under $n$-fold interleaving closure, 
meaning $X = X^{[n]}$,  if and only if $X$ factorizes as $X =(\pr_n)_{j=0}^{n-1} X_j$ under the $n$-fold interleaving operation for some $X_j$.
 We study   the allowable  sets $M$ of integers $n\ge 1$ for which there exists some
  set $X$ that has $X= X^{[n]} $ if and only if $n \in M$.  That is, letting  $\IC(X)= \{ n: \, X= X^{[n]}\},$ 
  we classify the sets $M \subset \NN^{+}$ such that $M = \IC(X)$ for some $X\subset \sA^{\NN}$. 
 We show that if finite, the set $\IC(X)$  consists of  the set of all divisors of an integer $n_0$,
 and all such $n_0$ may occur.  A new phenomenon is the existence of {\em infinitely  factorizable} $X$,  which necessarily have $X=X^{[n]}$ for all $n$ in an infinite distributive sublattice of $\mathbb{N}^+$ under the divisibility partial order, downward closed under divisibility. 
 We show all  such infinite sublattices may occur for non-closed $X$, but if $X$ is closed, we show  the only allowed infinite sublattice is $\mathbb{N}^{+}$. 

 There is   an additional algebraic structure consisting of  the collection of all operations obtained
 from combining  interleaving operations of different arities  under composition. These form a { nonsymmetric operad} in the category of sets,
 and we  term  it the {\em interleaving  nonsymmetric operad}. We
 give a series of universal shuffle identities under composition satisfied by this  operad. 
 We discuss  the operad formalism   in Section \ref{subsubsec:136} 
 and in Appendix \ref{sec:operadsec}.

\subsubsection{}The second topic studies symbolic dynamics aspects of  decimation and interleaving operations.
The shift operation acts compatibly with decimations and with $n$-fold interleaving, and we give commutation identities describing its action.
The class of shift-invariant sets (those with  $SX = X$)  and the class of shift-stable sets (those with $SX \subseteq X$) are
 not  preserved under interleaving. We  introduce an enlarged class of sets  better adapted to these operations. 

A  set $X \subseteq \sA^{\NN}$ is said to be {\em weakly shift-stable}  if there are integers  $k > j \ge 0$ such that $S^k X \subseteq S^j X$. 
 The set  $X$ need not  be  a closed set  in the symbol topology.
  We show the  class of all  weakly shift-stable sets, denoted $\sW(\sA)$,  is closed under the shift and  under all decimation and interleaving operations, as is the subclass $\overline{\sW}(\sA)$ of all closed  sets in $\sW(\sA)$.
%
%

The  complexity of a set  $X$ can be measured using various notions of the entropy of  $X$, which provide
 invariants that distinguish  dynamical systems.  The paper \cite{AL14a} studied  two concepts of  
of entropy for $X$,  the topological entropy  $H_{\topp}(X)$ and {\em path topological entropy} $H_p(X)$,
which we term here  {\em prefix topological entropy}.
For general $X$ 
one has $H_p(X) \le H_{\topp}(X)$, and 
strict inequality may occur.
We obtain an inequality  relating the prefix topological entropy of an $n$-fold interleaving $X= (\prn)_{i=0}^{n-1} X_i$
to that of its factors  $X_i$:  
\begin{equation}\label{eqn:interleave-prefix-entropy0}
H_{p}(X) \le \frac{1}{n} \sum_{i=0}^{n-1} H_{p} (X_i),
\end{equation}
and strict inequality may occur.
  A main result is that the class of weakly shift-stable sets $\sW(\sA)$ has good properties
for both entropies; the two entropies are equal and equality holds in the interleaving inequality  \eqref{eqn:interleave-prefix-entropy0} . 
In consequence, for weakly shift-stable sets we obtain  a formula for topological entropy under interleaving:   
$$
H_{\topp}(X) = \frac{1}{n} \sum_{i=0}^{n-1} H_{\topp} (X_i). 
$$

\subsubsection{} Most  results in this paper apply  to general sets $X$,  but for symbolic dynamics applications we are most interested in 
closed sets $X$ in the product  topology on $\sA^{\NN}$.  These satisfy:
 \begin{enumerate}
  \item
  If $X$ is a closed set, then all decimations  $\psi_{j, n}(X)$ are closed sets for $j \ge 0$ and  $n \ge 1$.
  \item
  If $X_0, X_1, \cdots, X_{n-1}$ are closed sets, then their $n$-fold interleaving 
$X = (\pr_n)_{i=0}^{n-1} X_i$ is a  closed set. 
\item
 Conversely,  if $X$ is a closed set and has an $n$-fold interleaving factorization
$X= (\pr_n)_{i=0}^{n-1} X_i$, then each $X_i$ is a closed set. 
\end{enumerate}
The  decimation,  interleaving, and shift operators all commute with topological closure operation $X \mapsto \overline{X}$.
In consequence all  $n$-fold interleaving closure operations   commute with topological closure.

Detailed statements of results are made in  Section \ref{sec:main_results}.
The main results concerning properties of $n$-fold interleaving closure operations of a set $X$ are Theorems \ref{thm:410},
 \ref{thm:lcm-factorization} 
 and  \ref{thm:infinite_factor}.
The main results concerning    weakly shift-stable sets $X$  are  Theorems \ref{thm:operation-closure} and
\ref{thm:semistable_prefix_entropy0}.

\subsection{Background}\label{sec:14}
This study  was motivated by work on    path sets   initiated in  \cite{AL14a}.
  Path sets are a class of closed sets in $\sA^{\NN}$  that form a generalization of shifts of finite type and of sofic shifts in symbolic dynamics, and
 which also  include sets not invariant under the   shift operator. Path sets  are described by
 finite automata, and have an  automata-theoretic characterization  as the  
 closed sets  in $\sA^{\NN}$ that are  the set of all infinite paths in some deterministic 
  B\"{u}chi automaton.    
 This class includes interesting sets 
  arising  in fractal constructions and in study of radix expansions in number theory
  (see \cite{AL14b}, \cite{AL14c}) arising from a number theory problem of Erd\H{o}s (\cite{Lagarias09})
   The paper \cite{AL14a} considered decimation operations on path sets
  and showed that decimations of path sets are also path sets.  
  The $p$-adic integers with the $p$-adic topology form a shift space with $p$-symbols, and 
 interleaving operations on path sets arose in this context in \cite{ABL17}.  
 
 The authors recently  studied the action of interleaving operations on path sets, in \cite{ALS20}.
  Interleaving operations already lead to the  breaking of shift-invariance
  even if all sets $X_i$ used  in  the interleaving are shift-invariant.
   The  paper  \cite{ALS20}  shows that  the class $\sC(\sA)$ of all
 path sets on a finite alphabet $\sA$  is closed under all  interleaving operators.

This paper obtains  results valid for general sets $X \subseteq \sA^{\NN}$, 
which provide perspective 
on results on path sets proved in \cite{ALS20}. 
The concept  of  weakly shift-stable
 closed sets $\overline{\sW}(\sA)$ supplies  a good generalization of the class $\sC(\sA)$ of path sets  to more general closed sets.
 The  paper  \cite{ALS20}  shows that all path sets are  weakly shift-invariant,
 which implies they are weakly shift-stable.
 In consequence, the entropy equalities of the present paper under weak shift-stability 
apply to interleaving of path sets.
The present  paper  includes  examples  showing  that various finiteness results given  in \cite{ALS20} for path sets 
are not valid for general closed  sets $X\subseteq \sA^{\NN}$, see 
Remark \ref{rem:73}. 

\subsection{Related work}\label{sec:15}
Decimation operations play an important role  in  sampling and interpolation operations in digital signal processing (``downsampling''), 
and  in multi-scale analysis and wavelets (e.g., \cite{Da92}, \cite{HBL17}). 
Interleaving constructions have been used in coding theory as a method for improving the burst error correction capability of a code (cf. \cite[Section 7.5]{VO89}). 
They are also considered in formal language theory; see  Krieger et al. (\cite{KMRRS:09}). 
The analogue of  $n$-fold  interleaving for finite codes is referred to by coding theorists as \emph{block interleaving of depth $n$}.
Decimation and interleaving operators together  have been considered both in cryptography and cryptanalysis,
cf.  Rueppel \cite{Rueppel86} and Cardell et al \cite{CFR18}. 
Since methods of encoding and decoding can be viewed as dynamical processes, it is of interest to view these operations in
a dynamical context.

\subsubsection{}
There has been prior work on interleaving  operations   in the automata theory literature,
typically for finite words.  In 1974  Eilenberg \cite[Chap. II.3, page 20]{Eilenberg74}
introduced a notion of {\em internal shuffle product} $A \coprod B$ of two recognizable sets (= regular language) which corresponds to
$2$-fold interleaving. A more general notion is {\em alphabetic shuffle}. 
The shuffle product has been characterized in the context of finite automata by Duchamp et al  \cite[Sect. 4]{DFLL:01}.
In this paper we are considering such operations  on infinite words, which differ in nature from the finite word case. 
For infinite words viewed in an automata-theoretic context, see Perrin and Pin \cite{PP04}.
 We are not aware of prior work studying  the  algebraic operator structure of   interleaving operations.

\subsubsection{} 
Regarding dynamics, one-sided shift-stable sets have their dynamics partially classified  by $C^{\star}$-algebra  invariants. The work of Cuntz and Krieger \cite{CuntzK80} and Cuntz \cite{Cuntz81} 
was seminal in attaching such invariants to topological Markov chains (= two-sided shifts of finite type). 
Carlsen (\cite{Carlsen08}, \cite{Carlsen13}) attached $C^{\ast}$-algebras to one-sided subshifts and studied their properties.
Shift-stable sets $X$ are studied in the context of partial isometry actions and $C^{\star}$ algebras attached to them by Dokuchaev and Exel \cite{DE17}.
See Exel \cite{Exel17} for related background. One may ask whether there are  generalizations of these constructions to
weakly shift-stable sets introduced in this paper.

\subsubsection{}
 In the ongoing development of operad theory and $n$-categories, interleaving operations have recently  played a role
 at a categorical level, see 
Leinster \cite{Leinster04} and Cottrell \cite{Cottrell18}. General references for operads are  Markl, Shnider and Stasheff \cite{MSS02},
and more recently   Loday and Vallette \cite{LV12} and Bremner and Dotsenko \cite{BremnerD:14} for algebraic operads.\smallskip

{\bf Acknowledgments.}
We  thank the reviewers for helpful comments and references. 
We  thank   I. Kriz for  a discussion   on operads.
Some  work of  W. Abram and D. Slonim was  facilitated by the Hillsdale College LAUREATES program, 
done by  D. Slonim under the supervision of W. Abram.
Work of J. Lagarias.was partially supported by NSF grants DMS-1401224 and DMS-1701229,
 and by   a Simons Fellowship in Mathematics  in  2019.

%
%
\section{Results}\label{sec:main_results}

We give formal definitions with examples, and then state  results. 

%
%

\subsection{Decimation  operations}\label{subsec:2001}

\begin{defn}  \label{decimationdef1}
{\rm (Decimation operations)} Let $\sA$ be a finite alphabet of symbols.
\begin{enumerate}

\item[(1)] 
For individual sequences $\bx \in \sA^{\NN}$ the {\em $i$-th decimation operation at level $n$}, denoted $\psi_{i,n}: \sA^{\NN} \to \sA^{\NN}$, for
 $i \ge 0$  is defined  for $\bx= a_0a_1a_2a_3 \cdots$  by 
$$
\psi_{i,n}(\bx) :=  a_i a_{i+n}a_{i + 2n}a_{i+3n}  \cdots.
$$
This operator extracts symbol subsequences  having  indices in an  arithmetic progression given by $ i \,\pmod n$, 
which starts at initial  index $i$. 
The {\em principal $n$-decimations} are those with $0 \le i \le n-1$. 

\item [(2)] For sets  $X \subseteq \sA^{\NN}$ the
 {\em $i$-th decimation  at level $n$}, denoted  $\psi_{i, n}(X)$, 
is the set union
\begin{equation}
\psi_{i, n}(X) := \left\{\psi_{i,n}(\bx)\, : \,\bx\in X\right\}.
\end{equation} 
\end{enumerate}
\end{defn}

\begin{exmp} \label{exmp:11}
For the alphabet
$\sA= \{0,1,2,3\}$ consider the  sets\footnote{Here $\bx_1= (01)^{\infty} = 010101...$}   
$$
X= \{ \bx_1= (01)^{\infty}, \bx_2= (10)^{\infty}\},\,\,  \mbox{and} \,\, Y= \{ \by_1=(323)^{\infty}, \by_2= (332)^{\infty}\}, 
$$
containing two periodic infinite words of period $2$ 
and  two  periodic infinite words of period $3$, respectively. 

 The principal  $2$-decimations of the elements of $X$ are 
$$
\psi_{0,2} ((01)^{\infty}) = 0^{\infty}, \psi_{1,2}((01)^{\infty}) = 1^{\infty},
\quad \mbox{and} \quad 
\psi_{0,2}((10)^{\infty}) = 1^{\infty},
 \psi_{1,2} ((10)^{\infty})= 0^{\infty}.
$$
Thus 
$
\psi_{0,2}(X) := 
 \{ 0^{\infty}, 1^{\infty} \}  \quad \mbox{and}\quad  \psi_{1,2}(X)= \{ 1^{\infty}, 0^{\infty}\}= \psi_{0,2}(X).
$

 The principal $2$-decimations of the elements of $Y$  are 
$$
 \psi_{0,2} ((323)^{\infty}) = (332)^{\infty}, 
 \psi_{1,2}((323)^{\infty}) =  (233)^{\infty}, 
 \,  \mbox{and} \, 
 \psi_{0,2}((332)^{\infty}) = (323)^{\infty},
 \psi_{1,2} ((332)^{\infty})= (332)^{\infty}. 
$$
We obtain 
$
\psi_{0,2}(Y) := \{   (332)^{\infty},  (323)^{\infty}\}=Y$ 
 and $ \psi_{1,2}(Y)= \{ (233)^{\infty}, (332)^{\infty} \} \ne Y.$
\end{exmp} 

In Section \ref{sec:decimation} we show:
\begin{enumerate}
\item
The set of all  decimation operations are closed under composition. For $X \subseteq \sA^{\NN}$, 
$$\psi_{j,m} \circ \psi_{i, n}(X)  = \psi_{i+ jn, mn}(X) .$$  
 This identity on subscripts matches an action of the $(ax+b)$-group on $\ZZ$. 
 \item
 The  shift action is compatible with the decimation action: For $X \subseteq \sA^{\NN}$, 
 $$
 \psi_{i,n}(SX)  = \psi_{i+1, n}(X). 
$$
and
$$
S \psi_{i,n}(X) = \psi_{i,n}(S^nX).
$$
\end{enumerate}

%
%

\subsection{Interleaving   operations}\label{subsec:2002}

 Interleaving operations  comprise an  infinite  collection of  $n$-ary operations $(n \ge 1)$,   defined for arbitrary subsets $X$ of the shift space  $\mathcal{A}^\mathbb{N}$.

\begin{defn}  \label{interleavingdef2}
{\rm (Interleaving operations)} 
Let $\sA$ be a finite alphabet of symbols.

\begin{enumerate}
\item[(1)] For individual sequences $\bx_i =a_{i,0} a_{i,1}a_{i, 2} \cdots  \in \sA^{\NN}$, 
 ($0 \le i \le n-1$) 
the   {\em $n$-fold interleaving operation} $\pr_n:  \sA^{\NN} \times \sA^{\NN} \times \cdots \times \sA^{\NN} \to \sA^{\NN}$, 
denoted either $(\prn)_{i=0}^{n-1}  \bx_i$ or $\bx_0 \pr \bx_1 \pr \cdots \pr \bx_{n-1}$,
combines these sequences  by 
$$ (\bx_0, \bx_1, \cdots, \bx_{n-1}) \mapsto        
   \bx_0 \pr \bx_1 \cdots \pr \bx_n = \by :=( a_{0,0}\,a_{1,0}   \cdots  a_{n-1,0}) \circ ( a_{0,1}\,a_{1,1}  \cdots a_{n-1,1}) \circ (a_{0,2} \cdots. 
 $$
where $\circ$ denotes concatenation of sequences. 
That is, $\by= b_0b_1b_2\cdots$ with 
$$
b_{i+kn}= a_{i, j}  \quad \mbox{for} \quad  \,\,  0 \le i \le n-1, \,\, \mbox{and} \,\, j \ge 0,
$$
so that the symbols of $\by$ in symbol positions  $i \, (\bmod\, n)$ are the  symbols of $\bx_i$, ($0 \le i \le n-1)$.
\item[(2)]
 For  sets  
 $X_j \subseteq \mathcal{A}^\mathbb{N}$, ($0 \le j \le n-1$),
 their \emph{$n$-fold interleaving}, 
 denoted $(\prn)_{i=0}^{n-1}X_i$ or \\
 $X_0 \prn X_1\prn X_2\prn \cdots \prn X_{n-1}$,
 is defined by the set union:
$$
 (\prn)_{i=0}^{n-1}X_i  = \{\bx_0\pr \bx_1 \pr \cdots \pr \bx_{n-1}:\bx_j\in X_j\text{ for all }0\leq j\leq n-1\}.
 $$
\end{enumerate}
\end{defn}

A set $X= (\pr_n)_{i=0}^{n-1} X_i$ is said to have an {\em $n$-fold interleaving factorization}. 
The sets $X_i$ are called {\em $n$-fold interleaving factors} of $X$, or just {\em interleaving factors}.
One can express $n$-fold interleavings in terms of principal decimations of level $n$ as:
 $(\prn)_{i=0}^{n-1}X_i=\{\bx\in\sA^{\NN}:\psi_{j,n}(\bx)\in X_j \text{ for all }0\leq j\leq n-1\}$.,
 see Proposition \ref{prop:41}. 
 
\begin{exmp} \label{exmp:12} 
Continuing Example \ref{exmp:11}, the $2$-fold interleaving of $X$ with itself is
\begin{eqnarray*}
X \pr X&=& \{ \bx_1 \pr \bx_1, \bx_1\pr \bx_2, \bx_2 \pr \bx_1, \bx_2\pr \bx_2\} \\
&=&\{ (0011)^{\infty}, (0110)^{\infty}, (1001)^{\infty}, (1100)^{\infty}\}. 
\end{eqnarray*} 
It contains four periodic words of period $4$. 

 The $2$-fold interleaving of $Y$ with itself is
\begin{eqnarray*}
Y \pr Y&=& \{ \by_1 \pr \by_1, \by_1\pr \by_2, \by_2 \pr \by_1, \by_2\pr by_2\} \\
&=&\{ (332233)^{\infty}, (332332)^{\infty}, (333223)^{\infty}, (333322)^{\infty}\}.
\end{eqnarray*} 
It contains four periodic words of period $6$.
The  $2$-fold interleavings of $X$ and $Y$ are
\begin{eqnarray*}
X \pr Y &:= &\{ \bx_1 \pr \by_1, \bx_1 \pr \by_2, \bx_2 \pr \by_1, \bx_2 \pr \by_2 \} \\
&= & \{ (031203130213)^{\infty}, (031302130312)^{\infty}, (130213031203)^{\infty},
(130312031302)^{\infty}.  \}
\end{eqnarray*}
\begin{eqnarray*}
Y \pr X &:= &\{ \by_1 \pr \bx_1, \by_1 \pr \bx_2, \by_2 \pr \bx_1, \by_2 \pr \bx_2 \} \\
&= & \{ (302130312031)^{\infty}, (303120313021)^{\infty}, (312031302130)^{\infty},
(313021303120)^{\infty}.  \},
\end{eqnarray*} 
Each of them contains four periodic words of  period $12$. We have $X \pr Y \ne Y \pr X$.
\end{exmp}

A  basic relation between interleaving and decimation is an identity, 
 valid at the pointwise level,  stating that ordered $n$-fold decimations post composed with $n$-fold interleaving
give  the identity map:
\begin{equation}\label{eq:principal-identity}
(\prn)_{i=0}^{n-1} \psi_{i, n}(\bx) = \bx \quad \mbox{for} \quad \bx \in \sA^{\NN}.
\end{equation} 
For this reason  we  call the decimations $\psi_{i,n}$ for $0 \le i \le n-1$, the
 {\em principal decimations}. 
The remaining decimations    $ i\ge n$ may be obtained by applying the one-sided shift operator 
to these decimation sets; see Proposition  \ref{prop:decimation_shift}.

%
%

\subsection{Interleaving closure operations}\label{subsubsec:201}

 The interleaving operations together with principal decimations define a family of set-theoretic closure operations on general
subsets $X \subseteq \sA^{\NN}$. These operators are a main focus of this paper.

\begin{defn} \label{def:interleaving-closure}
The  {\em $n$-fold interleaving closure} operation $X \mapsto X^{[n]}$ is
defined for each  $X \subseteq \sA^{\NN}$ by
\begin{equation}
X^{[n]} := (\prn)_{i=0}^{n-1} \psi_{i,n}(X).
\end{equation} 
\end{defn}


\begin{exmp}
For  $X= \{ (10)^{\infty}, (01)^{\infty}\}$,  the $2$-fold interleaving closure
$X^{[2]}  := \psi_{0,2} (X) \pr \psi_{1,2}(X)$ is 
$$
X^{[2]}= \{ 0^{\infty}, (01)^{\infty}, (10)^{\infty}, 1^{\infty} \}.
$$
We  have $X \subsetneq X^{[2]}$.
\end{exmp}


\begin{exmp} \label{basicexmp} (Interleaving and $n$-fold  interleaving closure) Let $\mathcal{A} = \{0,1\}$ and let 
$X_0 \subset \mathcal{A}^\mathbb{N}$ be the one-sided Fibonacci shift consisting of all words that do not contain the pattern  $11$ in two consecutive digits. 
Let $X_1=\sA^{\NN}$  be the full shift. Then:
\begin{enumerate}[(1)]
\item[(1)]  $X_0 \pr X_1 \subset \mathbb{A}^\mathbb{N}$ consists of all words that do not contain a $1$ in 
digits $i$ and $i+2$ for any $i$ even.
 That is, there can be no $1$'s in consecutive even digits, but there are no other restrictions on the word. Here 
  $X_0$ and $X_1$ are each invariant under the shift operator, i.e., $S(X_i)=X_i$, 
   but $X_0 \pr X_1$ is not shift-invariant.
\item[(2)]  Interleaving any number of copies of $X_1$ gives $X_1$. That is, $(\pr_n)_{j=0}^{n-1} X_1 = X_1$ for  $n \ge 1$.
\item[(3)] The $n$-fold interleaving closure of $X_0$ is $X_1$ for all $n \ge 2$, that is, $X_0^{[n]} = X_1= \mathcal{A}^\mathbb{N}$. This holds because $\psi_{j,n}(X_0) = X_1$ for all $j \ge 0$ when $n \ge 2$.
\item[(4)] Likewise, $X_1^{[n]} = X_1$ for $n \ge 1$. So $X^1$ has $n$-fold interleavings for all $n \ge 1$.
\end{enumerate}
\end{exmp}

In  Section \ref{subsec:41a} we show the existence of an $n$-fold interleaving factorization of a set $X$ corresponds to 
its invariance under $n$-fold interleaving closure, and that its 
interleaving factors are its principal decimations.

\begin{thm}\label{thm:DIF} 
{\rm (Decimations and interleaving factorizations)}

(1) A  subset $X$ of $\msrA^{\NN}$  
has an {$n$-fold interleaving factorization} $X = X_{0} \pr X_{1} \pr \cdots \pr X_{n-1}$
if and only if $X= X^{[n]}$.  

(2)  If $X= X_0 \pr X_1 \pr \cdots \pr X_{n-1}$  has an $n$-fold interleaving factorization, then its
ordered set of $n$-fold interleaving factors is  unique, given by its principal decimations   
$$
X_{i} = \psi_{i,n}(X) \quad \mbox{for} \quad 0 \le i \le n-1.
 $$
\end{thm} 

Regarding  $(2)$, there typically are many sets $Y$ such that
$\psi_{i,n} (Y) = \psi_{i, n}(X)$ for $ 0\le i \le n-1$, and we show $X$ contains every such set $Y$
in Theorem \ref{thm:inclusion}. 

In  Section \ref{subsec:21}  we justify the name $n$-fold interleaving closure by showing  
that $X \mapsto X^{[n]}$ is a set-theoretic closure operation,
as formalized in 
 Gr\"{a}tzer  \cite[Chap. I, Sect. 3.12, Defn.26]{Gratzer:11}, 
and  that $X^{[n]} $  is characterized  as the maximal set $Z$  having the 
property that 
$\psi_{i, n}(Z) = \psi_{i ,n}(X)$
for $0 \le  i\le n-1$.

In Section \ref{subsec:23} we establish  universal algebraic identities relating certain compositions
of  $n$-fold interleavings for different $n$.  

\begin{prop}\label{thm:shuffle} {\rm (Interleaving shuffle identities)}
For each $m,n \ge 2$ and arbitrary sets $\{ X_i : 0 \le i\le  mn-1\}$ contained in the one-sided shift $\msrA^{\NN}$,
one has the identity of sets
\begin{equation}\label{eqn:shuffle}
(\pr_n)_{i=0}^{n-1}\left((\pr_m)_{j=0}^{m-1} X_{i+ jn} \right) =  (\pr_{mn})_{k=0}^{mn-1}\, X_{k}. 
\end{equation}
\end{prop}

These identities  are termed  {\em shuffle identities} because the $n$-fold interleaving operation acts
like  a shuffling of $n$ decks of cards together, taking the top cards in a particular order from each of the $n$ decks,
where the cards correspond to positions of symbols in the expansion.

 In Section \ref{subsec:24} we establish a main result determining the action of composition of interleaving closure operators.
The shuffle identities play a crucial role in proving this  result.

\begin{thm}\label{thm:410}
{\rm (Composition of interleaving closures)}
 For all $m ,n \ge1$, and all $X \subseteq \sA^{\NN}$,
\begin{equation}\label{eqn:decimation_mn2a}
 (X^{[m]})^{[n]} = (X^{[n]})^{[m]} = X^{[ \lcm(m,n) ] },
 \end{equation}
 where $\lcm(m,n)$ denotes the least common multiple of $m$ and $n$. 
\end{thm} 

In Section \ref{sec:25}  we   show  interleaving commutes with intersection:
$$\bigcap_{j=0}^{m-1}(({\prn})_{i=0}^{n-1} X_{jn+i}) =({\prn})_{i=0}^{n-1}(\bigcap_{j=0}^{m-1}X_{jn+i}).$$
In Section \ref{sec:26a} we determine the action of the 
shift operator on  $n$-fold interleavings  and interleaving closures.  In particular we show that 
$$
S X^{[n]}= (SX)^{[n]}.
$$

In Section \ref{sec:26b} we  show that the  topological closure operation commutes with both decimations and interleaving
operations.
In particular it
 commutes under composition with  $n$-fold interleaving closure:
$$
\overline{X}^{[n]} = \overline{ X^{[n]}}. 
$$
Thus if   $X$ is a closed then its  $n$-fold interleaving closure   $X^{[n]}$  is  closed.

%
%

\subsection{Structure of interleaving factorizations}\label{subsubsec:133}

We study the possible structure of the set of all interleaving factorizations of a fixed set $X  \subseteq \sA^{\NN}$.


\begin{defn} 
Let  $X$ in $\sA^{\NN}$ be a fixed set, with $\sA$  a finite alphabet.

(1) The {\em interleaving closure set} $\IC(X) \subseteq \NN_{+}$ of $X$ is the set of integers
$$
\IC(X) := \{ n: n \ge 1 \quad \mbox{and} \quad X= X^{[n]} \}.
$$

(2) The {\em interleaving factor set} $\fF(X)$ consists of all $n$-ary interleaving factors, $X_{i,n}$, for all   $n \in \IC(X)$, i.e.
$$
\fF(X) = \{\psi_{i,n}(X):n\in\IC(X),0\leq i\leq n-1\}
$$

(3) The {\em (full) decimation set} $\ffD(X)$ consists of all decimations of $X$.
$$
\ffD(X) = \{ \psi_{i,n}(X): \, i \ge 0, \,\, n \ge 1\}.
$$
The {\em principal decimation set} $\ffD_{\text{\em prin}}(X)$ consists of all principal decimations
$$
\ffD_{\text{\em prin}}(X) := \{ \psi_{i,n}(X): \,  \,  n \ge 1, \,\, 0\le i \le n-1 \}.
$$
\end{defn}

The interleaving factor set is a subset of the set of all principal decimations: $\fF(X) \subseteq \ffD_{\text{\em prin}}(X) $
We always have   $X \in \fF(X)$ and $1 \in \sN(X)$.

 An important feature of factorizations  is that some  $X$
 are {\em  infinitely factorizable}  in the sense that they have $n$-fold interleaving factorizations for infinitely many  $n$,
 i.e.  $\IC(X)$  is infinite.
  The full one-sided shift $ X= \sA^{\NN}$ on the alphabet $\msrA$ is an example;
it has $m$-fold factorizations for all $m \ge 1$, and $\IC(\sA^{\NN}) = \NN^{+}$, while its  interleaving factor set  $\fF(\sA^{\NN})= \{ \sA^{\NN}\}$ contains one element.
 We term all the remaining ones {\em finitely factorizable}.   
There exist  closed sets $X$ 
having infinite $\IC(X)$ and   having an  infinite 
 interleaving factor set $\fF(X)$, see example \ref{exmp:infinite-factor-set}.

\begin{thm} \label{thm:lcm-factorization} 
{\rm (Structure of  interleaving closure sets) }
Let $\IC(X) = \{ n \ge 1: \, X=X^{[n]}\}$. Then $\IC(X)$ is nonempty and 
has the following properties.

(1) If $n \in \IC(X)$ and $d$ divides $n$, then $d \in \IC(X)$.

(2) If $m,n \in \IC(X)$ then their least common multiple $\lcm(m,n) \in \IC(X)$. 

 Conversely,  if a subset $N \subset \NN^{+}$ is nonempty 
and has properties (1) and (2), then there exists $X \subseteq \sA^{\NN}$ with $N = \IC(X)$. 
\end{thm}

This result is proved separately in the direct and converse  directions as Theorem \ref{thm:lcm-factorization2}
and Theorem \ref{thm:sublattices-are-attained2}, respectively. 
A nonempty structure $N$ having properties (1), (2) is abstractly characterized as
any nonempty subset of $\NN^{+}$ that is a sublattice under the divisibility partial order, 
which is also  downward closed
under divisibility, see \ref{thm:lcm-factorization2} (3). 
The notion of lattice  here is that of G. Birkhoff, see Gr\"{a}tzer \cite{Gratzer:11}.


 In Section \ref{sec:32a} we also  treat {\em self-interleaving factorizations}, which are interleaving  factorizations in which all
factors are identical.  
For a general set $X$ we define the {\em self-interleaving closure set} 
$$
\IC_{\self}(X):=\{n\in\NN:X=(\pr_n)_{i=0}^{n-1}Y\text{ for some }Y\subseteq\sA^{\NN}\}
$$
 as the set of $n$ such that $X$ has an $n$-fold self-interleaving factorization.
 Self-interleavings naturally occur in circumstances given in
 We show that  the structure of $\IC_{\self}(X)$ has
exactly  the same collection allowed structure as $\IC(X)$ in Theorem \ref{thm:lcm-factorization}; however 
for individual $X$ the set of values $ \IC_{\self}(X)$ can be strictly smaller than $\IC(X)$.

In Section \ref{sec:4}  we study infinitely factorizable sets $X$ in the special case that $X$
is a closed set.


\begin{thm}\label{thm:infinite_factor}
{\rm (Classification of infinitely factorizable closed $X$)} 
 For a closed set $X \subseteq \msrA^{\NN}$, where $\msrA$ is a finite alphabet,
the following properties are equivalent.
\begin{enumerate}
\item[(i)]
$X$ is infinitely factorizable; i.e., $\IC(X)$  is an infinite set.
\item[(ii)]
$X$ has an $n$-fold interleaving factorization for all $n \ge 1$, ie. $\IC(X) = \NN^{+}$. 
\item[(iii)]
For each $k \ge 0$  there are nonempty subsets $\msrA_k \subseteq \msrA$ such that $X = \prod_{k=0}^{\infty} \msrA_k$
is a countable  product of finite sets  with the product topology.
\end{enumerate}
\end{thm}

In view of Theorem \ref{thm:lcm-factorization}, the assumption that $X$ is closed is necessary for these three equivalences to hold. The important consequence for closed sets $X$ is that if they are infinitely factorizable then $\IC(X)= \NN^{+}$.

In Section \ref{sec:5} we study an iterated interleaving factorization process for a closed set $X$. If $X$ is infinitely factorizable, we ``freeze'' it. If it is finitely factorizable, we decompose it to its maximal factorization, and then repeat the process on each of these factors. We show by example that this factorization process can go to infinite depth.

 %
%
\subsection{Shift-stability and weak shift-stability}\label{subsubsec:132}

 We consider several classes of sets $X$ having different  transformation properties  under the shift action.

\begin{defn}  \label{defn:weakshift} (Shift-invariance, shift-stability, weak shift-stability)  

(1) A set $X\subseteq \sA^{\NN}$ is {\em shift-invariant } if $SX = X$.

(2)  A set $X$ is  {\em shift-stable}  if $SX \subseteq X$. 

(3) A  set $X$ is {\em weakly shift-invariant}   if there are integers  $k > j \ge 0$ such that $S^k X = S^j X$.

(4) A set $X$ is {\em weakly shift-stable}  if there are integers  $k > j \ge 0$ such that $S^k X \subseteq S^j X$.
\end{defn} 
These definitions do not require the set  $X$ to be closed in the symbol topology.

  In Section \ref{sec:30} we show  consequences of these properties. We show 
 that for shift-invariant sets, all interleaving factorizations are self-interleaving factorizations; 
 that is, if $X$ is shift-invariant, then $\IC(X)=\IC_{\self}(X)$. 
We  show that closed shift-stable sets have a forbidden blocks characterization paralleling the
two-sided shift case.  
 Example \ref{exmp:descending_chain} constructs a closed set $X$ having
  an  infinite, strictly descending chain of sets under the shift operation.

An  important property 
 introduced  here  is weak shift-stability.
The usefulness of this property is  that   the class $\sW(\sA)$ of all weakly shift-stable sets on  a finite alphabet $\sA$ is closed under all decimation,  interleaving and shift operations. This is not the case for  properties (1)-(3) above.

\begin{thm}\label{thm:operation-closure}
 Let $\sA $ be finite alphabet and let $X \subseteq \sA^{\NN}$ be a general set.

(1) If $X$ is weakly shift-stable, then all decimations $\psi_{j,n}(X)$ for $j \ge 0$, $n \ge 1$ 
are weakly shift-stable.

(2) If $X_0, X_1, \cdots , X_{n-1}$ are weakly shift-stable, then their $n$-fold interleaving 
$Y := (\pr_n)_{i=0}^{n-1} X_i$ is also weakly shift-stable.

(3) If $X$ is weakly shift-stable, then 
its $n$-fold interleaving closure $X^{[n]}$ is weakly shift-stable for each $n \ge 1$. 
\end{thm}

 A parallel result holds for the class $\overline{\sW}(\sA)$
 of all closed weakly shift-stable sets on the finite alphabet $\sA$.  This
 latter class  of sets includes the path sets studied in \cite{AL14a},
 as shown in \cite{ALS20}. 

 %
%
\subsection{Entropy of interleavings}\label{subsubsec:134}

In Section \ref{sec:entropy} we study two notions of entropy for general sets $X$.
\begin{defn}\label{defn:217} 
{\rm (Topological entropy)} 
The topological entropy $H_{\topp}(X)$ is given by
\[ 
H_{\topp}(X) := \limsup_{k \to \infty} \frac{1}{k} \log N_k(X)
\]
where $N_k(X)$ counts the number of distinct blocks of length $k$ to be found across all words $\bx \in X$.
\end{defn} 
The topological entropy  is  defined here as a limsup, however the limit always exists, as a consequence of a submultiplicativity
property of block counting functions $N_{k_1+ k_2}(X) \le N_{k-1}(X) N_{k_2}(X)$
of  $N_k(X)$, see \cite[Property 8]{AKM65}.  Here $\log$ denotes the natural logarithm;  in information theory
$\log_2$ is used instead.  

\begin{defn}\label{defn:218}
{\rm (Prefix entropy and stable prefix entropy)} 

 (1) The   \emph{prefix entropy}  (or {\em path topological entropy} of $H_p(X)$ of a general set $X$ ,
  is defined by
\begin{equation}\label{eqn:pathentropy1} 
H_p(X) := \limsup_{k \rightarrow \infty} \frac{1}{k} \log N_k^I(X),
\end{equation}
where $N_k^I(X)$ counts the number of distinct prefix blocks.
 $b_0b_1\cdots b_{k-1}$ of length $k$ found 
across  all words $\bx \in X$.

 (2) The limit in \eqref{eqn:pathentropy1} does not always exist, and
we say that $X$ has {\em stable prefix entropy} if the
limit does exist:
\begin{equation}\label{eqn:pathentropy2} 
H_p(X) := \lim_{k \rightarrow \infty} \frac{1}{k} \log N_k^I(X),
\end{equation}
\end{defn} 

The prefix entropy  was introduced in \cite{AL14a} under the name   {\em path topological entropy}  for a
class of sets called {\em path sets}.
In that paper symbol sequences  were labels attached to paths of edges in a directed labeled graph.
Prefix blocks were termed there {\em initial blocks} (for path sets) i
because they represented the initial steps along a path in a directed labeled graph defining the path set.
Since $N_k^{I}(X) \le N_k(X)$ we always have  $H_p(X) \le H_{\topp}(X)$, and strict inequality may hold.

In Section \ref{subsec:entropy2} we show the shift operator preserves both entropies.
Decimation operations need not preserve entropy
and  Section \ref{subsec:entropy-decimation} gives inequalities 
such entropies  must satisfy.
In Section \ref{subsec:prefix_entropy_bounds} we establish an inequality for prefix entropy 
of interleavings of general sets,

\begin{thm}\label{thm:entthm1} {\rm (Prefix entropy bound under interleaving)}
Let $X_0, X_1, \cdots, X_{n-1}$ be arbitrary subsets of $\sA^{\NN}$. 
The prefix
 entropy of the  set $X = X_0 \pr X_1 \pr \cdots \pr X_{n-1}$  is 
bounded above by the arithmetic mean of the prefix 
entropies of $X_0, \ldots, X_{n-1}$. That is:
\begin{equation}\label{eqn:prefix-ent-ineq1}
H_p(X_0 \pr \cdots \pr  X_{n-1})
\le \frac{1}{n} \sum_{i=0}^{n-1} H_p(X_i).
\end{equation} 
\end{thm}

Example \ref{exmp:910} shows that strict inequality  in \eqref{eqn:prefix-ent-ineq1} may occur.

In  Section \ref{subsec:stable_prefix_entropy_bounds} we
 show that  the assumption of   stable prefix entropy for each of the sets $X_0, X_1, ..., X_{n-1}$ implies equality in this formula.
and that the $n$-fold interleaving $X = X_0 \pr \cdots \pr X_{n-1}$ itself has stable prefix entropy.


\begin{thm}\label{thm:entthm-gen} {\rm (Stable prefix entropy interleaving formula)}
If  each of the sets $X_0, X_1, ... , X_{n-1}$ has  stable prefix entropy,
then the $n$-fold interleaving  $X = X_0 \pr X_1 \pr \cdots \pr X_{n-1}$ also  has stable prefix 
 entropy. In addition  
\begin{equation}\label{eqn:prefix-ent-eq2}
H_p(X)
= \frac{1}{n} \sum_{i=0}^{n-1} H_p(X_i).
\end{equation}
\end{thm}

In contrast  to this result for interleaving, decimations  of  a set   $X$ having stable prefix entropy need not
 have stable prefix entropy, see Remark \ref{rem:99}.

We also  deduce  in Section \ref{subsec:stable_prefix_entropy_bounds}
that  all weakly shift-stable sets $X$  have 
good entropy properties.

\begin{thm}\label{thm:semistable_prefix_entropy0}
{\rm (Weak shift-stability implies stable prefix entropy)}  
If $X$ is weakly shift-stable, then $X$ has stable prefix  entropy.
 and  in addition $H_p(X) = H_{\topp}(X)$. 
 Consequently $n$-fold
 interleaving of weakly-shift stable sets $X_i$ has
 \begin{equation}\label{eqn:prefix-ent-eq3}
H_{\topp}(X)
= \frac{1}{n} \sum_{i=0}^{n-1} H_{\topp}(X_i).
 \end{equation} 
\end{thm}

Finally we observe that since  all  decimations of weakly shift-stable sets are weakly shift-stable, 
they will have   stable prefix entropy.

\subsection{Composition of interleavings  and operad structure}\label{subsubsec:136} 

In Section \ref{sec:5} 
we consider factorizations of a set $X$ under iterated composition. 
We give examples of sets $X$ having  iterated factorizations going to
infinite depth. This behavior differs from interleaving restricted to  the class of all path sets on the 
finite alphabet $\sA$, as we show in \cite{ALS20} that the iterated factorization
of any path set terminates at some finite depth.

Abstractly, the   family of operations obtained under iterated composition
 from interleaving operations of all arities  determines a
 \emph{non-symmetric operad} (also called a  \emph{non-$\Sigma$ operad}) in the sense of May \cite{May97},
see also Markl et al. \cite[Part I, Sect. 1.3]{MSS02} and Markl \cite[Sect. 1]{Markl08}. 
Non-symmetric operads arise in many combinatorial constructions, see work  of Giraudo \cite{Giraudo17},  \cite{Giraudo18}.  
Iterated interleaving operations satisfy nontrivial universal identities under composition, 
examples being the shuffle identities given in Theorem \ref{thm:shuffle}. 
These identities show that certain  nested compositions of interleaving operations give equivalent operations.
 However most nestings of compositions yield  distinct operations.  In particular,
 interleaving operations  do not satisfy the associative law when acting on collections of subsets $X$ of 
$\mathcal{A}^\mathbb{N}$.
For instance, each of the  $3$-ary operations
$X_0 \pr X_1 \pr X_2$ and $X_0 \pr (X_1 \pr X_2)$ and $(X_0 \pr X_1) \pr X_2$ are distinct.

Operads in general are characterized as (universal) algebraic objects  satisfying a given set of universal identities. 
We shall consider the {\em interleaving non-symmetric operad} to be the  non-symmetric operad whose universal identities are all 
the identities satisfied on the collection of all sets $X \subseteq \sA^{\NN}$ with alphabet size $|\sA| =2$. 
These universal identities include the shuffle identities in Theorem \ref{thm:shuffle}.
This  set of identities  may be  a generating set for all universal  identities for this operad; we leave it 
as an open question to determine a generating set.

In  Appendix \ref{sec:operadsec} we provide  details checking the   operad structure associated to interleaving.

%
\subsection{Contents of paper}\label{sec:124}

The contents of the remainder of the paper are 
as follows:

Section ~\ref{sec:decimation} relates  decimation operations and shows
these operations are closed under composition and under the shift operator.

 Section ~\ref{sec:factorization2A} studies interleaving operations and the interleaving closure operation $X\to X^{[n]}$ for general sets $X\subseteq\sA^{\NN}$, proving Theorem \ref{thm:DIF} and the shuffle identities.

Section \ref{sec:3} establishes  divisibility properties of $n$-fold factorizations of a closed set $X$.

Section ~\ref{sec:4} classifies infinitely factorizable closed sets $X$.
These sets have more restricted factorizations than for non-closed sets.

 Section \ref{sec:5} studies iterated interleaving factorizations of closed sets $X$.
 It shows by example that such iterated factorizations can continue to infinite depth.

Section \ref{sec:30} studies  shift-stability and weak shift-stability of sets $X \subseteq \sA^{\NN}$.
It gives a forbidden-blocks characterization of shift-stable closed sets. 
It shows that  the class of weakly shift-stable sets is closed under all decimation and interleaving operations.

Section ~\ref{sec:entropy} defines and discusses topological entropy and prefix (topological) entropy, proving Theorems \ref{thm:entthm1}, 
through \ref{thm:semistable_prefix_entropy0}.

 Section \ref{sec:concluding}  discusses further directions for research.

 Appendix \ref{sec:operadsec} studies an operad structure generated by interleaving operations.

\medskip

%
%
\section{Decimations of arbitrary subsets of $\sA^{\NN}$} \label{sec:decimation}

This section studies  decimations and interleaving for  subsets $X \subseteq \sA^{\NN}$.
All results in this section apply to arbitrary subsets $X$ of $\sA^{\NN}$.

%
%
\subsection{Compositions of decimations  } \label{subsec:31a}

The set of all decimation operators is closed under composition of operators.
This composition   action is a representation 
of the  discrete $ax+b$  semigroup given by the nonnegative
 integer matrices $\left[  \begin{smallmatrix} a & b \\ 0 & 1\end{smallmatrix} \right]$ with $ a \ge 1$ and  $b \ge 0$.

\begin{prop}\label{prop:decimation_shuffle}
{\rm (Composition  of decimations)} 
Let $X \subseteq \sA^{\NN}$ be an arbitrary set.
 For all $j, k \ge 0$ and $m, n \ge 1$ we have 
\begin{equation}\label{eqn:compose-decimations} 
\psi_{j, m} \circ \psi_{k, n} (X) :=  \psi_{j, m} (\psi_{k, n} (X)) = \psi_{jn+k, mn}(X) 
\end{equation}
\end{prop}

\begin{proof} 
The result is verified separately  for each  element  $\bx= x_0x_1x_2 \cdots \in X$
We set $\by:=\psi_{k, n} (\bx) = x_k x_{k+n}x_{k+2n}x_{k+3n}  \cdots$ where $\by= y_0y_1y_2 \cdots$ has $y_j= x_{k+jn}$. Now
\begin{eqnarray*} 
\psi_{j, m} \circ \psi_{k, n}(\bx) & = & \psi_{j, m} (\by) = y_j y_{j+m} y_{j+2m} y_{j+3m} \cdots \\
& = & x_{k+jn} x_{k+ jn + mn} x_{k+jn + 2mn} \ldots,\\
&=&  \psi_{jn+k, mn}(\bx), 
\end{eqnarray*}
as asserted.
\end{proof}

%
%
\subsection{Decimations and the shift } \label{subsec:32a}

The  decimation operations also transform nicely under the one-sided shift $S(a_0a_1a_2...) = a_1a_2a_3 \cdots$.

\begin{prop}\label{prop:decimation_shift}
{\rm (Shift of decimations)} 
Let $X \subseteq \sA^{\NN}$ be an arbitrary set.

(1) For all $j \ge 0$ and $m \ge 1$ the one-sided shift $S$ acts as
\begin{equation}\label{eqn:shift_decimation1}
\psi_{j, m}(SX) =\psi_{j+1, m} (X).
\end{equation}

(2) In addition 
\begin{equation} \label{eqn:shift_decimation2}
 S(\psi_{j,m} (X))= \psi_{j+m,m}(X) =  \psi_{j, m} (S^m X)  .
\end{equation} 
\end{prop}

\begin{proof} 
For a single element $\bx \in X$,  \eqref{eqn:shift_decimation1} is equivalent to the assertion
$$\
\psi_{j, m} (S\bx) = \psi_{j, m} (S(x_0 x_1 x_2 \cdots))=\psi_{j, m}(x_1x_2 x_3 \cdots) = x_{j+1} x_{m+(j+1)} x_{2m+(j+1)}\cdots= \psi_{j+1, m}(\bx) .
$$
In this case \eqref{eqn:shift_decimation2} for $\bx \in X$ is equivalent to
$\quad S(x_j x_{m+j} x_{2m+ j} \cdots) = x_{m+j} x_{2m+j} x_{3m+j} \cdots .$

(2) For $\bx \in X$ we have 
$$
S\psi_{j, m}(\bx) = S(x_{j}x_{j+m}x_{j+ 2m} \cdots)= x_{j+m}x_{j+2m} \cdots= \psi_{j+m,m}(\bx) = \psi_{j,m}(S^m\bx),
$$
where the last equality used (1) iterated $m$ times.
\end{proof}

%
%
\section{Interleaving for arbitrary subsets of $\sA^{\NN}$} \label{sec:factorization2A}

%
%
\subsection{Interleaving and decimation } \label{subsec:41a}

Interleaving operations can be characterized in terms of the principal decimations of their output. The criterion  (2) below
could be used as an alternate  definition of $n$-fold interleaving of sets.

\begin{prop}\label{prop:41} {\rm (Decimation characterization of interleavings) }

(1) Every $\bx \in \sA^{\NN}$ has an $n$-fold interleaving factorization $\bx= (\pr_n)_{i=0}^{n-1} \bx_i$ for all $n \ge 1$. This factorization
is unique, with $\bx_i= \psi_{i,n}(\bx)$  \,($0 \le i \le n-1$), so that
\begin{equation}\label{eqn:single-factorization} 
\bx = \psi_{0,n}(\bx) \pr \psi_{1,n}(\bx) \pr \cdots \pr \psi_{n-1,n}(\bx)= (\prn)_{i=0}^{n-1} \psi_{i, n} (\bx). 
\end{equation} 

(2) If $X \subseteq \sA^{\NN}$ has an $n$-fold interleaving factorization $X= (\pr_n)_{i=0}^{n-1} X_i$, then
\begin{equation}\label{eqn:interleave-def2} 
X =\{ \bx\in\sA^{\NN}:\psi_{i,n}(\bx) \in X_i \quad \text{ for all } \quad 0\leq i\leq n-1\}.
\end{equation}
This factorization is unique 
 with  $X_i = \psi_{i,n}(X)$ \, $(0 \le i \le n-1)$, so that
 \begin{equation}\label{eqn:interleave-def1} 
X = \psi_{0,n}(X) \pr \psi_{1,n}(X) \pr \cdots \pr \psi_{n-1,n}(X) = (\pr_n)_{i=0}^{n-1} \psi_{i, n} (X).
\end{equation} 
\end{prop}

\begin{proof} 
The identity 
\eqref{eqn:single-factorization} is immediate from the definition of interleaving product,
checking it symbol by symbol. This $n$-fold interleaving factorization of $\bx$ is unique because if 
$\bx=(\pr_n)_{i=0}^{n-1}\bx_i$, then the $(i+kn)$th symbol of $\bx$ is by definition the $k$th symbol of $\bx_i$, so that each symbol of $\bx_i$ is determined by a symbol of $\bx$.

(2) Let $X= X_0 \pr X_ \pr \cdots \pr X_{n-1}.$ By definition
 \begin{align*}
X   &= \{\bx \in\sA^{\NN}:\,\, \bz=(\pr_n)_{i=0}^{n-1} \bx_i \, :\, \bx_i\in X_i, \,\, \text{ for all \,\, }0\leq i\leq n-1\}
\nonumber    \\
    &=\{\bx\in\sA^{\NN}\, :\,\,\psi_{i,n}(\bx) =\bx_i,\,\, \text{with} \,\, \bx_i\in X_i, \,\, \text{ for all} \,\, 0\leq i\leq n-1\}
\nonumber    \\
    &=\{\bx\in\sA^{\NN}:\,\, \psi_{i,n}(\bx) \in X_i \,\, \text{ for all } \,\, 0\leq i\leq n-1\},
\end{align*}
which is  \eqref{eqn:interleave-def2};  we used (1) to deduce the second equality.  

To show  \eqref{eqn:interleave-def1}, it suffices to show $\psi_{i,n}(X) =X_i$. 
We have $\psi_{i,n}(X) \subseteq X_i$ by  \eqref{eqn:interleave-def2}.
To show the map is onto, for any  $\bx_i$ we can pick arbitrary  $\bx_j \in X_j$ for $j \ne i$
and then  (1) gives   $\bx := (\pr_n)_{j=0}^{n-1} \bx_j \in Z$ has $\psi_{i,n}(\bx) = \bx_i$,
as required. 
\end{proof} 

We deduce Theorem \ref{thm:DIF} from the proposition.   

%
%
\begin{proof}[Proof of Theorem \ref{thm:DIF}]
We are to show $X$ has an interleaving factorization if and only if $X=  X^{[n]}.$  

(1)  Suppose $X = X^{[n]}$. By definition  $X^{[n]}= \psi_{0,n}(X) \pr \cdots \pr \psi_{n-1,n}(X)$ has  an interleaving factorization,
so $X$ does too. 
Conversely if  $X= X_0 \pr X_1 \pr \cdots \pr X_{n-1}$  is an interleaving factorization then by Proposition \ref{prop:41} (2) 
$X_i= \psi_{i, n}(X)$ whence
$X^{[n]} = X_0 \pr X_1 \pr \cdots \pr X_{n-1} = X.$ 

(2) This is Proposition \ref{prop:41} (2).
\end{proof}

%
%
\subsection{$n$-fold  interleaving closure operations} \label{subsec:21}

We show that the family of closure operations $X\to X^{[n]}$ on sets $X\subseteq\sA^{\NN}$ commutes with topological closure, and the equality $X= X^{[n]} $ corresponds to $X$ having an $n$-fold interleaving factorization.

The following result shows this operation is a closure operation in the set-theoretic  sense.

\begin{thm}\label{thm:inclusion} {\rm ($n$-fold interleaving closure)} \,
The $n$-fold interleaving closure  operation $X^{[n]}$ of sets $X \subseteq \sA^{\NN}$ has the following   properties: 
\begin{enumerate}
\item[(1)]
{\rm (Projection property)} The  $n$-fold interleaving closure $X^{[n]}$ is  characterized by the  property 
that  it is the maximal set $Z$ such
that its principal decimations at level $n$ satisfy 
\begin{equation}\label{eqn:projection} 
\psi_{i,n}(Z) =  \psi_{i, n}(X) \quad \mbox{for} \quad 0\le i \le n-1.
\end{equation} 
\item[(2)] {\rm (Extension property)} Any  set $X \subseteq \sA^{\NN}$ satisfies 
\begin{equation}\label{eqn:interleave-inclusion}
X \subseteq X^{[n]}.
\end{equation}
\item[(3)]  {\rm (Idempotent property)} The operation $X \mapsto X^{[n]}$ is idempotent; i.e., 
$(X^{[n]})^{[n]} = X^{[n]}$ for all $X$.

\item[(4)] {\rm(Isotone property)} If $X \subseteq Y$ then $X^{[n]} \subseteq Y^{[n]}$. 
\end{enumerate}
\end{thm}

\begin{rem}\label{rem:43} 
(Set theory  closure property) 
Properties (2) , (3), and (4) comprise the axioms of a  {\em Moore closure} property 
(See   Schechter \cite[Sec. 4.1-4.12]{Schechter96}). 
These axioms are known to be equivalent to the property of being closed under arbitrary intersections.
The  $n$-fold interleaving closure operation
does not  satisfy all of Kuratowski's axioms  defining the closed sets of a topology; it does not  
satisfy the set union property $(X \cup Y)^{[n]} = X^{[n]} \cup Y^{[n]}$. It does satisfy the inclusion  
 \begin{equation}\label{eqn:interleaving_closure_intersection}
 X^{[n]} \cup Y^{[n]} \subseteq  \big(X \cup Y\big)^{[n]}.
 \end{equation} 
As an  example showing the inclusion can be strict,  take $X = X^{[2]} = \{ 0^{\infty}\}$, $Y=Y^{[2]}=  \{ 1^{\infty}\}$.
Then $X^{[2]} \cup Y^{[2]}     \subsetneq \big(X \cup Y\big)^{[2]} = \{ 0^{\infty}, 1^{\infty}, (01)^{\infty}, (10)^{\infty}\}$. 
Relations between the interleaving closure  operations and topological closure in $\sA^{\NN}$ are
given in Section \ref{sec:26b}.
\end{rem} 

\begin{proof} 
 (1) If a collection of sets each have property \eqref{eqn:projection} then so does 
 their union, and $X$ has property \eqref{eqn:projection}, so 
  there exists a maximal set $Z$ with property \eqref{eqn:projection}. 
 By definition 
 $$
   X^{[n]} :=\psi_{0,n}(X)\pr\psi_{1,n}(X)\pr\cdots\pr\psi_{n-1,n}(X) \
 $$
 Then  by  Proposition \ref{prop:41}(2), 
 \begin{align}\label{eqn:above} 
     X^{[n]}
     &=\{\bz\in\sA^{\NN}:\psi_{i,n}(\bz)\in\psi_{i,n}(X)\text{ for all }0\leq i \leq n-1\}
 \end{align}
 The statement $\psi_{i,n}(Z)=\psi_{i,n}(X)$ means that $\psi_{i,n}(\bz)in \psi_{i,n}(X)$ for all $\bz\in Z$. From \eqref{eqn:above}, 
 one  sees that $Z= X^{[n]}$ is precisely the maximal set  such that \eqref{eqn:projection}   holds  for all $0\leq i\leq n-1$.

 (2) It follows from (1). Alterntively, 
 by Proposition \ref{prop:41}(1) 
given $\bx \in X$ we have
$$
\bx = \psi_{0, n}(\bx) \pr \psi_{1,n}(\bx) \pr \cdots \pr \psi_{n-1,n}(\bx) \in \psi_{0, n}(X) \pr \psi_{1,n}(X) \pr \cdots \pr \psi_{n-1, n}(X),
$$
which  certifies  $\bx \in X^{[n]}$, proving  \eqref{eqn:interleave-inclusion}.

(3) Idempotence follows from (1) and  (2): By (1) $X^{[n]}$ is the maximal set having   $\psi_{i,n}(X^{[n]}) \subseteq \psi_{i,n}(X)$ holds for $0 \le i \le n-1$. Now by (2) $(X^{[n]})^{[n]}$ contains $X^{[n]}$.  But  $\psi_{i,n}((X^{[n]})^{[n]}) \subseteq  \psi_{i,n}((X^{[n]})  \subseteq \psi_{i,n}(X)$
for $0\le i\le n-1$, 
so it is also maximal, so $(X^{[n]})^{[n]} = X^{[n]}$.

(4) Suppose that $ X \subseteq Y$. Using the projection property (2) for $X$ and $Y$ separately shows
 $$
 \psi_{j,n} (Y^{[n]} \cup X^{[n]})=\psi_{i,n} (Y^{[n]}) \cup \psi_{i,n} (X^{[n]}) =\psi_{ji n}(X) \cup \psi_{i,n}(Y) = \psi_{i,n}(Y) = \psi_{i,  n}(Y^{[n]}) \quad 0 \le i \le n-1.
 $$  
 The projection property now gives $Y^{[n]} \cup X^{[n]}\subseteq Y^{[n]}$, whence $X^{[n]}\subseteq Y^{[n]}$.  
\end{proof}

We now  prove for general sets $X \subseteq \sA^{\NN}$  the characterization of interleaving factorizations
given in Theorem \ref{thm:DIF}. 

%
%
\begin{proof}[Proof of Theorem \ref{thm:DIF}] 
(1) Suppose $X = X^{[n]}$. Then by definition 
$X =  X^{[n]}= \psi_{0,n}(X) \pr \cdots \pr \psi_{n-1,n}(X)$ is an interleaving factorization. 
Conversely if  $X= X_0 \pr X_1 \pr \cdots \pr X_{n-1}$  is an interleaving factorization then by Proposition \ref{prop:41} (2) 
$X_i= \psi_{i, n}(X) = X_i$ whence
$X^{[n]} = X_0 \pr X_1 \pr \cdots \pr X_{n-1} = X.$ This proves  both (1) and the uniqueness of factorization (2). 
\end{proof}

%
%
\subsection{Shuffle  identities for interleaving operators}\label{subsec:23}

The  family of interleaving operations  satisfy 
universal algebraic identities under particular compositions of operations, 
acting on general subsets of $\msrA^{\NN}$.
We now prove Proposition \ref{thm:shuffle}, which asserts
\begin{equation}\label{eqn:shuffle2}
(\pr_n)_{i=0}^{n-1}\left((\pr_m)_{j=0}^{m-1} X_{i+jn} \right) =  (\pr_{mn})_{k=0}^{mn-1}\, X_{k}. 
\end{equation}

One reads the interleaving of interleavings  on the left side  of  \eqref{eqn:shuffle2} as 
\begin{eqnarray*}
&&(X_{0} \pr X_{n} \pr \cdots \pr X_{(m-1)n}) \pr (X_{1} \pr X_{n+1} \pr \cdots  \pr X_{(m-1)n+1})\pr \cdots  \\
&& \quad\quad\quad \quad\quad\quad\quad  \cdots \pr (X_{n-1} \pr X_{2n-1} \pr \cdots  \pr X_{mn-1}),
\end{eqnarray*} 
with parentheses indicating   composition of $m$-fold interleavings  given as input to an $n$-fold interleaving. The right side
of  \eqref{eqn:shuffle2}
is  an $mn$-fold interleaving, \begin{equation}
X_{0} \pr X_{1} \pr X_{2} \pr \cdots \pr X_{n-1} \pr X_{n} \pr X_{n+1} \pr  \cdots
 \pr \cdots \pr X_{(m-1)n + n-2}  \pr X_{(m-1)n + n-1},
\end{equation}
with  factors taken in  linear order.


\begin{proof}[Proof of Proposition \ref{thm:shuffle} ]
Using Proposition \ref{prop:41}(2) we obtain 
\begin{align}
    (\pr_n)_{i=0}^{n-1}\left((\pr_m)_{j=0}^{m-1} X_{i+jn} \right)
    &=\{\bx\in\sA^{\NN}:\psi_{i,n}(\bx)\in (\pr_m)_{j=0}^{m-1} X_{i+jn}\text{ for all }0\leq i\leq n-1\}
 \nonumber   \\
    &=\{\bx \in\sA^{\NN}:\psi_{j,m}\left(\psi_{i,n}(\bx)\right)\in X_{i+jn}\text{ for all }0\leq j\leq n-1,0\leq i\leq n-1\}
\nonumber    \\
    &=\{\bx \in \sA^{\NN}:\psi_{i+jn,mn}(\bx)\in X_{i+jn}\text{ for all }0\leq j\leq n-1,0\leq i\leq n-1\}
\nonumber    \\
    &=(\pr_{mn})_{k=0}^{mn-1}\, X_{k}. \nonumber
\end{align}
The interleaving definition gives the second and fourth  equality and Proposition \ref{prop:decimation_shuffle}   the third equality. 
\end{proof}

Shuffle identities are useful in studying self-interleavings of sets $X$. 
\begin{defn}\label{defn:self-interleaving} 
 Give $X \subseteq \sA^{\NN}$ let $X^{(\pr n)}$ denote the {\em $n$-fold self-interleaving} defined by
$$
X^{(\pr n)} := (\pr_n)_{i=0}^{n-1} X = X \pr X \pr \cdots \pr X \quad \mbox{($n$ factors in product)}. 
$$
\end{defn}

The  special case of  self-interleaving under composition satisfies identities similar to that of  exponentiation,
a consequence of  the shuffle identities.

\begin{prop}\label{expprop}  {\rm (Composition of self-interleavings)}  For any natural numbers $m, n \ge 1$, and 
any subset  $X$ of $\msrA^{\NN}$, the following set-theoretic identity holds
for $n$-fold, $m$-fold and $mn$-fold self-interleaving.
\begin{equation}\label{eqn:self_interleaving0} 
(X^{(\pr n)})^{(\pr m)} = (X^{(\pr m)})^{(\pr n)} = {X}^{(\pr mn)}.
\end{equation} 
\end{prop}

\begin{proof} 
In Theorem \ref{thm:shuffle} choose all $X_k=X$ for  $0 \le k \le mn-1$ and obtain $(X^{(\pr m)})^{(\pr n)} = {X}^{(\pr mn)}$.
Then interchange $m$ and $n$.
\end{proof}

%
%
\subsection{Composition identities for interleaving closure operations}\label{subsec:24}

We prove  Theorem \ref{thm:410}  determining  the composition  of  self-interleaving closure operations:
$ (X^{[m]})^{[n]} = (X^{[n]})^{[m]} = X^{[ \lcm(m,n) ] }$.

We first establish a preliminary result  giving  formulas  and inclusions for   compositions of interleaving closure operations.

\begin{prop}\label{prop:cdc}
{\rm (Composition formulas)} 
(1) For all $m ,n \ge1$, and all $X \subseteq \sA^{\NN}$,  
\begin{equation}\label{eqn:decimation_mn2}
 (X^{[m]})^{[n]} = (\prn)_{i=0}^{n-1} \psi_{i,n}(X^{[m]}) . 
\end{equation}

(2) For all $m ,n \ge1$, and all $X \subseteq \sA^{\NN}$,
\begin{equation}\label{eqn:decimation_mn1}
 X^{[mn]}= (\prn)_{i=0}^{n-1} (\psi_{i,n}(X)^{[m]}) . 
\end{equation}

(3) For  all $m,n \ge 1$ 
\begin{equation}\label{eqn:divisibility-inclusion}
X^{[m]} \subseteq X^{[mn]} . 
\end{equation}

(4) If $\gcd(m,n)=1$ then
$$
(X^{[m]})^{[n]} = (X^{[n]})^{[m]} =X^{[mn]} 
$$
\end{prop} 

\begin{proof}
(1) This assertion is the definition of the $n$-fold interleaving closure of $X^{[m]}$. 

(2) We set  $X_{k} := \psi_{k, mn}(X)$ for $0 \le k \le mn-1$ in the shuffle identity \eqref{eqn:shuffle}, obtaining
$$
X^{[mn]} = (\pr_n)_{i=0}^{n-1} \big((\pr_m)_{j=0}^{m-1} \psi_{i+ jn, mn}(X) \big)
$$
 The right side of this equation contains  terms $Z_i := (\pr_m)_{j=0}^{m-1} \psi_{jn+i, mn}(X)$, and we  must show 
$Z_i =\psi_{i,n}(X)^{[m]}.$
We have
\begin{eqnarray*} 
 \psi_{i,n}(X)^{[m]} & := &\Big(\psi_{0,m}\circ \psi_{i,n}(X)\Big) \hspace{2pt} \pr \hspace{2pt} \Big(\psi_{1, m} \circ \psi_{i,n}(X)\Big)
  \hspace{2pt} \pr \hspace{2pt} \cdots \hspace{2pt}\pr\hspace{2pt} \Big(\psi_{m-1, m} \circ \psi_{i, n}(X)\Big)\\
&=& \psi_{i, mn}(X)  \hspace{2pt}\pr\hspace{2pt} \psi_{i+n, mn} (X)\hspace{2pt}\pr\hspace{2pt} \cdots \hspace{2pt}\pr 
\hspace{2pt}\psi_{i+ (m-1)n, mn}(X)\hspace{2pt}\hspace{2pt} =\hspace{2pt}\hspace{2pt} Z_i.
\end{eqnarray*}
as required.

(3) We have
 $X^{[m]} \subseteq (X^{[m]})^{[n]}$
 by the extension property of $n$-fold interleaving.
 We claim that 
 \begin{equation}\label{eqn:m-n-mn}
 (X^{[m]})^{[n]}\subseteq X^{[mn]}.
 \end{equation} 
To prove the claim,  comparing  the now proved \eqref{eqn:decimation_mn2} and \eqref{eqn:decimation_mn1},  it suffices to show 
\begin{equation}\label{eqn:inclusion1}  
\psi_{i,n}  (X^{[m]}) \subseteq \psi_{i,n}(X)^{[m]} \quad \mbox{for} \quad 0 \le i \le n-1. 
\end{equation} 
For fixed $i$, the right side of this inclusion is an $m$-fold interleaving 
$$
 \psi_{i,n}(X)^{[m]} = \big(\psi_{0,m} \circ \psi_{i,n}(X) \big) \hspace{2pt}\pr\hspace{2pt} \big(\psi_{1,m} \circ \psi_{i,n}(X)\big)  
 \hspace{2pt}\pr\hspace{2pt} \cdots \hspace{2pt}\pr\hspace{2pt} \big(\psi_{m-1,m} \circ \psi_{i,n}(X)\big)
$$ 
The composition rule for  decimations (Proposition \ref{prop:decimation_shuffle})   shows that 
\begin{equation}\label{eqn:leftside}
 \psi_{i,n}(X)^{[m]} = \psi_{ i, mn}(X) \pr \psi_{i+n,mn}(X) \pr \cdots \pr \psi_{i+(m-1)n ,mn}(X). 
\end{equation}

To evaluate  the left side   of the inclusion \eqref{eqn:inclusion1},    
suppose  $\bx = \psi_{i, n}(\bz) \in \psi_{i,n}(X^{[m]})$  with $\bz \in X^{[m]}$. 
Now by Proposition \ref{prop:41} (1),  $\bx$  has an $m$-fold interleaving factorization
$$
\bx=(\pr_m)_{j=0}^{m-1} \bw_j = \bw_0 \pr \bw_1 \pr \cdots \pr \bw_{m-1},
$$
where
$$
\bw_j =\psi_{j,m}(\bx) = \psi_{j,m} (\psi_{i,n} (\bz)) = \psi_{i+jn, mn} (\bz).
$$
Therefore
\begin{equation} \label{eqn:double1}
\bx= \psi_{i,n}(\bz) = (\pr_m)_{j=0}^{m-1} \psi_{i+ jn, mn}(\bz) = \psi_{i,mn}(\bz) \pr \psi_{i+ n, mn}(\bz) \pr \cdots \pr \psi_{i+(m-1)n, mn}(\bz).
\end{equation}

We are to show $\bx \in \psi_{i,n}(X)^{[m]}$ .
It suffices to show 
\begin{equation}\label{eqn:double2}
\psi_{i+jn,mn}(\bz) \in \psi_{i+jn, mn}(X)  \quad \mbox{for} \quad  0 \le j \le m-1,
\end{equation} 
since \eqref{eqn:double1} then asserts $\bx \in (\pr_m)_{j=0}^{m-1} \psi_{i+jn, mn }(X)$
whence   \eqref{eqn:leftside} shows $\bx \in \psi_{i, n}(X)^{ [m]}$.

To show \eqref{eqn:double2},  any $\bz \in X^{[m]}$ has, for $0 \le k \le m-1$, 
$$
\psi_{k,m}(\bz) \in \psi_{k,m}(X^{[m]})= \psi_{k,m}(X),
$$
where the equality of sets holds  by definition of $m$-fold interleaving.
Thus there exists some $\oz_k\in X$ with $\psi_{k,m}(\bz) = \psi_{k,m}(\oz_k).$
Now for $0 \le i \le n-1, 0\le j \le m-1$,
there exist unique $(k , \ell)$ satisfying
\begin{equation}\label{eqn:ijkl}
i + jn = k + \ell m,
\end{equation}
with $0 \le k \le m-1, 0\le \ell\le n-1$.
Here $k = k(i,j)$ is determined by $k \equiv i+jn \, (\bmod\, m)$. We have
\begin{eqnarray*}
\psi_{i+jn, mn} (\bz) &=& \psi_{k + \ell m, mn}(\bz) = \psi_{\ell, n} (\psi_{k,m}(\bz))\\
&=& \psi_{\ell, n}(\psi_{k,m}(\oz_k)) = \psi_{k+\ell m, mn}(\oz_k)\\
&=& \psi_{i +jn, mn} (\oz_k) \in \psi_{i+jn, mn}(X),
\end{eqnarray*}
showing \eqref{eqn:double2}.

  (4) It suffices to  show $(X^{[m]})^{[n]} = X^{[mn]}$ if $\gcd(m,n)=1$; interchanging $m$ and $n$ then
  gives the other case.  The  proof of (3) showed that   $(X^{[m]})^{[n]}\subseteq X^{[mn]}$
  holds (with no $\gcd$ restriction),  so it suffices to show the reverse inclusion
  $ X^{[mn]} \subseteq (X^{[m]})^{[n]}$.
  By the already proved \eqref{eqn:decimation_mn2}  and  \eqref{eqn:decimation_mn1}  this assertion is
 \begin{equation} \label{eqn:equality0} 
  X^{[mn]}=    (\prn)_{i=0}^{n-1} (\psi_{i,n}(X)^{[m]})         \subseteq (\prn)_{i=0}^{n-1} (\psi_{i,n}(X^{[m]})) = (X^{[m]})^{[n]}.
  \end{equation}
   It therefore suffices to prove the individual set equalities 
  \begin{equation}\label{eqn:equality3}  
\psi_{i,n}  (X^{[m]}) = \psi_{i,n}(X)^{[m]} \quad \mbox{for} \quad 0 \le i \le n-1. 
\end{equation}
hold when $\gcd(m,n)=1$.

  Now suppose we are given an arbitrary $\bx\in \psi_{i,n}(X)^{[m]}$. 
  We wish to show $\bx \in \psi_{i,n}(X^{[m]})$. 
  To begin, $\bx$  has 
  an $m$-fold interleaving factorization
  $$
  \bx = (\pr_m)_{j=0}^{m-1} \psi_{j,m}(\bx_j) , 
  $$
  in which  each $\bx_j = \psi_{i,n}(\bz_j) \in \psi_{i,n}(X)$ with  $\bz_j \in X$. Thus we have   
   \begin{equation}\label{eqn:psi-x-to-z} 
   \psi_{j,m}(\bx_j) = \psi_{j,m}(\psi_{i,n}(\bz_j)) = \psi_{i+jn, mn}(\bz_j). 
  \end{equation} 
  As in (3) there are $(k, \ell)$ with
   $$
  \psi_{i+jn, mn}(\bz_j) = \psi_{k + \ell m, mn}(\bz_j) = \psi_{\ell, n}(\psi_{k,m}(\bz_j)).
  $$

    Here, for fixed $i$, the value $k= k(i,j) $ is given by  $k \equiv i+jn\, (\bmod \, m)$. 
    The values $k(i,j)$ are all distinct as $j$ ranges from 0 to $m-1$ with $i$ fixed, because  $\gcd(m,n)=1$.
  It follows that the inverse map $j =j(i,k)$ is well defined.
 By definition of  $m$-fold interleaving closure, there will exist a value $\bz \in X^{[m]}$  having
  \begin{equation} \label{eqn:make-z} 
  \psi_{k,m}(\bz) = \psi_{k,m}(\bz_{j}) \quad \mbox{for} \quad 0 \le k \le m-1,
 \end{equation} 
 with $\bz_j \in X$ and  $j=j(i,k)$ runs over all $0 \le j \le m-1$ as $k$ varies.
 
    We claim that $\psi_{i,n} (\bz) = \bx$. We have
  \begin{eqnarray*}   
   \psi_{i,n}(\bz) &=&      (\pr_m)_{j=0}^{m-1} \psi_{i+jn, mn}(\bz) \quad\quad\quad\quad \quad \text{by \eqref{eqn:double1}}  \\
  &=& (\pr_m)_{j=0}^{m-1} \psi_{k(i, j) + \ell(i,j)m,mn}(\bz)     \quad\,\,\, \,\text{by \eqref{eqn:ijkl}} \\
   &=& (\pr_m)_{j=0}^{m-1} \psi_{\ell(i,j),n}(\psi_{k(i,j), m}(\bz)) \\
    &=& (\pr_m)_{j=0}^{m-1} \psi_{\ell(i,j),n}(\psi_{k(i,j), m}(\bz_{j})), \quad\,\,\, \,\text{by \eqref{eqn:make-z}}
    \end{eqnarray*}
  where the last line used the defining property of $\bz \in X^{[m]}$.  Now we simplify
  \begin{eqnarray*}   
  \psi_{i,n}(\bz)     &=& (\pr_m)_{j=0}^{m-1} \psi_{k(i,j) + \ell(i,j)m, mn}(\bz_j)) \\
    &=& (\pr_m)_{j=0}^{m-1} \psi_{i + jn, mn}(\bz_j) 
    = (\pr_m)_{j=0}^{m-1} \psi_{j,m}(\bx_j)  \quad\quad  \text{by \eqref{eqn:psi-x-to-z}}  \\
    &=& \bx,
  \end{eqnarray*}
  so $\bx \in \psi_{i,n}(X^{[m]})$. 
 \end{proof}

We now  prove the   Theorem \ref{thm:410} formulas for composition of interleaving closures.
 .

%
%
\begin{proof}[Proof of Theorem \ref{thm:410}] 
It suffices to prove $(X^{[m]})^{[n]}  = X^{[ \lcm(m,n)]}$,  because its  right
side is symmetric in $m$ and $n$; we  may then exchange $m$ and $n$  to establish $(X^{[n]})^{[m]}  = X^{[ \lcm(m,n)] }$.
We have already proven $(X^{[m]})^{[n]}  = X^{[mn]}$
for the case $\gcd(m,n)=1$ in Proposition \ref{prop:cdc}.
 
 For general $n,m$ we let $d= \gcd(m,n)$, the greatest common divisor. One can always find
$e,f$ with $d=ef$ such that $e|m$ and $f|n$ and $\gcd(\frac{m}{e}, \frac{n}{f})=1.$
To see  this, let  $d=\prod_{p} p^{e(p,d)}$ denote the prime factorization of $d$; then the choice $e = \prod_{p^{e(p,d)} || m} p^{e(p,d)}$
and $f= \prod_{p^{e(p,d) +1} | m} p^{e(p,d)}$ will work. Note that if $p^{e(p,d) +1} | m$, then necessarily $p^{e(p,d)} || n$, so that $f|n$. By construction, $e|m$, $ef=d$, and $\gcd(\frac{m}{e},\frac{n}{f})=1$. 

We then have 
$$
X^{[\lcm(m,n)]} = X^{ [mn/ef]} = (X^{[m/e]})^{[n/f]} \subseteq (X^{[m/e]})^{[n]} \subseteq (X^{[m]})^{[n]}.
$$ 
Reading from  left to right  the second equality  comes from Proposition \ref{prop:cdc} (4), the first inclusion  follows  from Proposition \ref{prop:cdc}(3), and the final inclusion follows  from the isotone property (4) in Theorem \ref{thm:inclusion}. 

It remains to show  that 
$$
(X^{[m]})^{[n]}\subseteq X^{[\lcm(m,n)]}.
$$
Now let  $d = \gcd(m,n)$, so that $\ell=\lcm(m,n) =  \frac{mn}{d}$. 
By Proposition \ref{prop:cdc}(1) we  have
$(X^{[m]})^{[n]} = (\prn)_{i=0}^{n-1} \psi_{i,n}(X^{[m]})$
(without any $\gcd$ restriction).

Now consider $\bx= \bx_0 \pr \bx_1 \pr \cdots \pr \bx_{n-1} \in (X^{[m]})^{[n]}$,
and write $\bx= b_0 b_1 b_2 \cdots$. 
Here for $0 \le i \le n-1$,  
$$
\bx_i  := \psi_{i,n}(\bx)
= b_{i}b_{i+n}b_{i+ 2n} b_{i+3n} \cdots. 
$$

We are to show that $\bx \in X^{[\lcm(m,n)]}$. 
To begin, we have 
$$
\bx_i  = \psi_{i,n}(\bz_{i,0} \pr \bz_{i,1} \pr \cdots \pr \bz_{i, m-1}) \in \psi_{i,n}(X^{[m]}).
$$
where  each $\bz_{i,j} \in \psi_{j,m} (X)$ for $ 0 \le j \le m-1$, so that
$$
\bz_{i,j} = \psi_{j,m} (\bw_{i,j}) \quad \mbox{with} \quad \bw_{i,j} \in X.
$$

Because $\gcd(m,n)=d$, the  application of $\psi_{i,n} (\cdot)$ to $\bz_i= (\pr_m)_{j=0}^{m-1} \bz_{i,j} \in X^{[m]}$ 
   only hits those words   $\bz_{i,j}$  having  subscripts $j$ falling in $\frac{m}{d}$ different
 residue classes $(\bmod \, m)$, and it visits each such class exactly $d$ times, as $j$ varies over  $0 \le j \le m-1$. 
 These $\frac{m}{d}$ classes  $(\bmod \, m)$ comprise distinct  residue classes 
 $(\bmod \,\frac{m}{d}),$  again because $\gcd(m,n)=d$. 
 These classes   are exactly   $i + jn \,(\bmod \, \frac{m}{d})$ for $0 \le j \le \frac{m}{d} -1$. 
  We can therefore rewrite  
  $\bx_i= \by_{i,0}  \pr \by_{i,1} \pr \cdots \pr \by_{i, \frac{m}{d}-1}$ 
  with 
  $\by_{i,j }= b_{i+jn}b_{i + jn+ mn/d} b_{i+ jn + 2mn/d} \cdots$ for $0 \le j \le \frac{m}{d}-1$.
  We have $\lcm(m,n) = \frac{mn}{d}$ different elements $\by_{i,j} \in \psi_{k, \frac{m}{d}}(X)$.
 The key point is that  for $k= i+jn$  we have
 $$
 \by_{i,j} =  \psi_{i+jn \, (\bmod \, m/d), m/d} (\bw_{i+jn }) \in \psi_{k, \frac{m}{d}}(X) \quad \mbox{for} \quad 0\le i\le n-1, \,\, 0 \le j \le \frac{m}{d}-1. 
 $$
Here $k= i+jn$ varies over the interval $0\le k \le \frac{mn}{d} -1$.  Consequently, 
$$
\bx = (\pr_{n})_{i=0}^{n-1} \bx_i  =(\pr_{n})_{i=0}^{n-1}  \left((\pr_{m/d})_{j=0}^{m/d-1} \by_{i,j} \right) 
= (\pr_{\frac{mn}{d}})_{k=0}^{mn/d -1} \psi_{k, \frac{mn}{d}}(\bw_{k})\big). 
$$
where  the last equality uses the shuffle identity  \eqref{eqn:shuffle}. We also find  that  $k=i+jn$ runs through the residue classes 
$(\bmod \, mn/d)$ in the correct order. 
 We conclude that $\bx \in X^{[mn/d]}= X^{[\lcm(m,n)]}$, 
  establishing the desired inclusion.
 \end{proof}

%
%
\subsection{Interleaving commutes with set intersection}\label{sec:25}

Interleaving also behaves well with respect to intersection.

\begin{prop}\label{intintprop}
{\rm (Interleaving commutes with intersection)}
For $m, n \ge 2$ and  subsets $X_0, X_1, \cdots, X_{mn-1}$ of $\msrA^{\NN}$, the following set-theoretic identity holds:
\begin{equation}\label{eqn:interleave_intersect} 
\bigcap_{j=0}^{m-1}(({\prn})_{i=0}^{n-1} X_{jn+i}) =({\prn})_{i=0}^{n-1}(\bigcap_{j=0}^{m-1}X_{jn+i}).
\end{equation}
\end{prop}

\begin{proof}
By Proposition \ref{prop:41} (2),  we have $\bx \in Z_j:= ({\prn})_{i=0}^{n-1} X_{jn+i}$ if and only if $\psi_{i}(\bx)\in X_{jn+i}$ for $0 \le i \le n-1$.
Consequently:
\begin{eqnarray*}
\bx \in \bigcap_{j=0}^{m-1}(({\prn})_{i=0}^{n-1} X_{jn+i}) & \Leftrightarrow& 
\psi_i(\bx) \in X_{jn+i} \quad \mbox{for} \quad 0\le i \le n-1, \, 0 \le j \le m-1\\
& \Leftrightarrow & \psi_i(\bx) \in \bigcap_{j=0}^{m-1}X_{jn+i} \quad \mbox{for} \quad 0\le i \le n-1 \\
& \Leftrightarrow& \bx \in (\prn)_{i=0}^{n-1}(\bigcap_{j=0}^{m-1}X_{jn+i}),
\end{eqnarray*} 
verifying  \eqref{eqn:interleave_intersect}.
\end{proof} 

\begin{cor}\label{cor:dec_closure_intersection}
Let $X, Y \subseteq \sA^{\NN}$. Then their $n$-fold interleaving closures satisfy
\begin{equation}\label{eqn:dec_closure_intersection}
X^{[n]} \cap Y^{[n]} = Z^{[n]} 
\end{equation}
where $Z := (\pr_n)_{i=0}^{n-1} \big(\psi_{i,n}(X) \cap \psi_{i,n}(Y)\big) =Z^{[n]}.$  
\end{cor}
\begin{proof} In Proposition \ref{intintprop}  take  $m=2$  and $n \ge 2$  and choose $X_{i} = \psi_{i,n}(X)$ and $X_{n+i} = \psi_{i,n}(Y)$ for $0 \le i \le n-1$. 
The left side of \eqref{eqn:interleave_intersect} is $X^{[n]} \cap Y^{[n]}$ and the right side is $Z$. Here $Z= Z^{[n]}$ holds
because $Z$ is defined as an $n$-fold interleaving of 
 $\psi_{i,n}(Z) :=  \psi_{i,n}(X) \cap \psi_{i,n}(Y)$  by construction.
(In general   $\psi_{i,n}(X \cap Y) \subseteq \psi_{i, n}(X) \cap \psi_{i,n}(Y) $, 
so that $ (X \cap Y)^{[n]} \subseteq Z^{[n]}$, and strict inequality can hold.)
\end{proof} 
\begin{rem}\label{rem:non-intersect} (Intersection of  general interleaving closures) 
For intersection of two interleaving closures  of different arities of  a single set $X$ we have, for all $m,n \ge 1$, 
\begin{equation}
X^{[\gcd(m,n)]}  \subseteq X^{[m]} \cap X^{[n]}.
\end{equation} 
Equality always holds trivially when $m=n$,  but need not hold when $m \ne n$. 
 As an example,  for  $m=2, n=3$  take $X=\{ \bx_1, \bx_2, \bx_3\} = \{ (010100)^{\infty}, (111111)^{\infty}, (110111)^{\infty} \}$. 
 Then $(01)^{\infty}$ is contained in both $X^{[2]}$  via the   $2$-fold interleaving $\psi_{0,2}(\bx_1) \pr \psi_{1,2}(\bx_2)$, and  $X^{[3]}$  via the $3$-fold interleaving 
$\psi_{0,3}(\bx_1) \pr \psi_{1,3}(\bx_1) \pr \psi_{2,3}(\bx_3)$. 
 We conclude  $X \subsetneq X^{[2]} \cap X^{[3]}$. 
Note this  example  is closed and weakly shift-stable, having  $S^6X = X$. 
\end{rem} 

%
%
\subsection{Shift action on interleavings}\label{sec:26a}

The one-sided shift map  acts as 
$$
S(a_0a_1a_2 a_3 \cdots)= a_1a_2a_3a_4\cdots.
$$
We show the one-sided shift $S$ action preserves the property of having an $n$-fold interleaving factorization.

\begin{prop}\label{prop:interleaving-shift}
{\rm (Interleaving and the shift map)} 
Suppose that $X $ has an $n$-fold interleaving factorization $X= X_0 \pr X_1  \pr \cdots \pr X_{n-2}\pr X_{n-1}$.

(1) The  one-sided shift map $S$ acts by 
\begin{equation}\label{eqn:shift_factorization}
S(X) = X_1 \pr X_2 \pr \cdots \pr X_{n-1} \pr S(X_0).
\end{equation}
Consequently 
\begin{equation}\label{eqn:iterated-shift-interleave}
S^{n}(X) = S(X_0)  \pr S(X_1)\pr \cdots \pr S(X_{n-2}) \pr S(X_{n-1}).
\end{equation}

(2) All  iterates  $S^{k}(X)$ possess $n$-fold interleaving factorizations
$$
S^{k}(X) = \psi_{k,n}(X) \pr \psi_{k+1, n}(X) \pr \cdots \pr \psi_{k+n-1,n}(X).
$$
\end{prop}

\begin{proof}
(1), (2). It suffices to prove \eqref{eqn:shift_factorization}. 
The other assertions in (1) and assertion (2) 
then follow easily by induction on $k \ge 1$. 

To begin, for all infinite words $\bx \in \msrA^{\NN}$ we have 
\begin{equation}\label{eq:one_step}
 \psi_{j, n}(S\bx)  = \psi_{j+1,n}(\bx) \quad \mbox{for all} \quad j \ge 0.
\end{equation}
By Theorem \ref{thm:DIF} 
we have $X_i = \psi_{i,n}(X)$
for $0 \le i \le n-1$. We set  $X_n= \psi_{n,n}(X)$.
By  Proposition \ref{prop:decimation_shuffle} (2), 
$$\psi_{n,n}(X) = \psi_{0,n}(SX) = S^n \psi_{0,n}(X) = S(X_0).$$
We assert $S(X) = X_1 \pr X_2 \pr \cdots \pr X_{n-1} \pr S(X_0)$. 
We have  the inclusion
$$S(X) \subseteq (\pr_n)_{i=0}^{n-1} \psi_{i, n}(S(X)) = (\pr_n)_{i=0}^{n-1} \psi_{i+1,n}(X) = (\pr_n)_{i=1}^{n} X_i.$$

To show the opposite inclusion
$$X_1 \pr X_2 \pr \cdots \pr X_{n-1} \pr S(X_0)\subseteq SX,$$
let $\by= \by_1 \pr \by_2 \pr \cdots\pr  \by_n  \in X_1 \pr X_2 \pr \cdots \pr X_n$; then $\by_i \in X_i$ for $1 \le i \le n$.
For $\by_n \in X_n$ by definition there exists $\bx \in X$ such that $\psi_{0, n}(\bx) = x_0 \circ \by_n \in X_0$, for some $x_0$,
where $x_0 \circ \by_n$ denotes the concatenation of the letter $x_0$ and the infinite word $\by_n$. 
By the $n$-fold factorization hypothesis on $X$  one may choose this $\bx$ so that also $\psi_{i,n}(\bx) =\by_i$ holds
for $1 \le i \le n-1$. Now one checks using \eqref{eq:one_step} that
$$
\psi_{i, n}(S(\bx)) = \by_{i+1} \quad \mbox{for} \quad 0 \le i \le n-1.
$$
\end{proof}

\begin{prop}\label{prop:int-closure-shift}
{\rm (Shift map and $n$-fold interleaving closure) }
The shift map commutes with $n$-fold interleaving closure.
 For each $n \ge 1$, and a general set $X \subseteq \sA^{\NN}$, there holds 
\begin{equation}\label{eqn:shift-n-closure}
S(X^{[n]}) = (SX)^{[n]}.
\end{equation} 
\end{prop}

\begin{proof}
By definition the $n$-fold interleaving closure  $X^{[n]}$ has an $n$-fold interleaving factorization. Applying parts  (1) and (2) 
of Proposition \ref{prop:interleaving-shift} 
we have
\begin{eqnarray*}
S(X^{[n]}) & =&  S\left(\psi_{0, n}(X) \pr \psi_{1, n}(X) \pr \cdots \pr \psi_{n-2, n}(X) \pr \psi_{n-1,n}(X) \right) \\
&=& \psi_{1,n} (X) \pr \psi_{2,n}(X) \pr \cdots \pr \cdots \pr \psi_{n-1,n}(X) \pr S\psi_{0,n}(X) \\
&=& \psi_{1,n} (X) \pr \psi_{2,n}(X) \pr \cdots \pr \cdots \pr \psi_{n-1,n}(X) \pr \psi_{n,n}(X) = (SX)^{[n]} . 
\end{eqnarray*} 
 \end{proof}

%
%
\subsection{Topological closure} \label{sec:26b}

Decimation and interleaving operations  and the shift operation all commute with  topological closure  in
$\sA^{\NN}$.

\begin{thm}\label{thm:closure 2} {\rm ($n$-fold interleaving closure)} \,
Given a subset $X$ of $\msrA^{\NN}$, let $\overline{X}$ denote its topological closure
in the shift topology (product topology) in $\sA^{\NN}$.  

(1) For each $n \ge 1$ and $j \ge 1$, 
$$
\psi_{j,n}(\overline{X}) = \overline{\psi_{j,n}(X)}. 
$$
In particular if $X$ is a closed set in $\msrA^{\NN}$
then each decimation  $X_{j,n}= \psi_{j,n}(X)$ is a closed set.

(2) 
For $X_0, X_1, \ldots, X_{n-1} \subseteq \sA^{\NN}$, there holds
$$
 (\pr_n)_{j=0}^{n-1} \overline{X}_j = \overline{(\pr_n)_{j=0}^{n-1} X_j }.
$$
In particular the $n$-fold interleaving of closed sets is a closed set.

(3) The $n$-fold interleaving closure operation  commutes with the closure operation on 
the product topology on $\sA^{\NN}$,  
$$
(\overline{X})^{[n]}  = \overline{X^{[n]}}. 
$$
(4) The shift operator commutes with topological closure, 
$$
S \overline{X} = \overline{SX}.
$$
\end{thm}

\begin{proof}
(1) Given a sequence $\{ \psi_{j,n}(\bx_k): k \ge 1 \}$ in $\psi_{j,n}(X)$, with each $\bx_k=x_{0,k} x_{1,k} x_{2,k} \cdots  \in X$, 
because  closed sets $X$ are compact in $\msrA^{\NN}$, there
exists  a convergent subsequence of the $\bx_k \in X$, having limit $\bx= x_0x_1 x_2 \cdots  \in X$, say. (Convergence is defined by
 eventual stability of each symbol $x_{\ell,k}$
as $k \to \infty$, having $x_{\ell, k}  = x_{\ell}$ for all sufficiently large $k$.)
 It is easy to see that if $\bx_k  \to \bx$ in $X$  then necessarily $\psi_{j,n}(\bx_k) \to \psi_{j,n}(\bx)$
 is a convergent subsequence in $\psi_{j,n}(X)$, establishing that  $\psi_{j,n}(X)$ is closed.
 
(2) For interleavings of closed sets  the convergence of symbols in a given position $\ell=n\ell'+j' $ in $X_0 \pr X_1 \pr \cdots \pr X_{n-1}$ 
depends only on $X_{j'}$, hence the closure property of $X$  is inherited from that of the individual factors $X_{j'}$.

(3) The closure equality  follows from (1), since  $X^{[n]}= \psi_{0,n}(X) \pr \psi_{1,n} \pr \cdots \pr \psi_{n-1,n}$
and   both sides of the equality add in all sequence limits taken  in each  symbol position separately.

(4) We have $\bx \in \overline{X}$ if there is  a sequence $\{ \bx_k: k \ge 1\}$ in $X$ converging to $\bx$. Then  the sequence $\by_k: = S\bx_k \in SX$ converges to $S\bx$ in $S \overline{X}$
so $S\overline{X} \subseteq \overline{SX}$. Take now $\by_k \in SX$ converging to $\bx$. By definition of $\sX$  there exists $\bx_k \in S$ with $S\bx_k=\by_k$. Since the alphabel $\sA$
is finite, infinitely many of the $\bx_k$  have a fixed letter $a_0$ as initial symbol. These  define a subsequence $\bx_{k_i}$ that converges in $X$ to a limit word $\bx$ and necessarily $S\bx= \by$.
Thus $\overline{SX} \subseteq S\overline{X}$.
\end{proof}

%
%
\section{Interleaving factorizations and divisibility} \label{sec:3}

We classify the possible values  of $n$ in $n$-fold interleaving factorizations for different $n$ of arbitrary  subsets $X \subseteq \sA^{\NN}$.

%
\subsection{Divisibility for interleaving factorizations }\label{sec:31a}

\begin{thm} \label{thm:lcm-factorization2} 
{\rm (Divisibility structure for interleaving factorizations) }
Let $\IC(X) = \{ n: \, X=X^{[n]} \}$. 

(1) If $n \in \IC(X)$ and $d$ divides $n$, then $d \in \IC(X)$. 

(2) If $m,n \in \IC(X)$ then their least common multiple $\lcm(m,n) \in \IC(X)$. 

(3) The  interleaving closure set $\IC(X)$ of $X$ has the structure of a  distributive lattice with respect to the divisibility partial order, being
closed under  the  join operation (least common multiple $\lcm$), and the meet operation (greatest common divisor $(\gcd)$).
It is  downward closed under divisibility, and  contains the  minimal element $1$.
\end{thm}

%
%

\begin{proof}[Proof of Theorem \ref{thm:lcm-factorization2}]
(1) If $n \in \IC(X)$ then  $X=X^{[n]}.$ Suppose  $d$ divides $n$, so $n=de$.
Now $X \subseteq X^{[d]}$ by the extension property  of Theorem \ref{thm:inclusion}. 
However  $X^{[d]}  \subseteq X^{[de]}= X^{[n]}$
by Proposition \ref{prop:cdc} (2). Since $X^{[n]} =X$ we conclude $X^{[d]}=X$, so $d \in \IC(X)$.

(2)  Suppose $m,n \in \IC(X)$ so that  $X=X^{[m]}$ and $X= X^{[n]}$.
Then 
$$
X = X^{[n]} = (X^{[m]})^{ [n]} = X^{[\lcm(m,n)]} 
$$
where, reading from the left, the second equality substituted $X^{[m]}$ for $X$ and the last equality is Theorem \ref{thm:410}.
Thus $\lcm(m,n) \in \IC(X)$. 

 (3) The set $\IC(X)$ is downward closed under divisibility by (1).  If $m,n \in \IC(X)$  then  $\gcd(m,n) \in \IC(X)$ 
siince it divides $m$.
It is closed under the
join operation $\lcm$ by (2).  Thus $\IC(X)$   is  a  sublattice of the distributive lattice of integers $\NN_{+}$ under divisibility.
It always has minimal element $1$.
\end{proof}

A corollary of part (2) says that interleaving factors of infinitely factorizable sets are infinitely factorizable. 

\begin{cor}\label{cor:infinitely-factorizable}
Let $X$ be infinitely factorizable. Then every interleaving factor of $X$ is also infinitely factorizable. 
\end{cor}

\begin{proof}
Suppose  $X$ is infinitely factorizable, and $X=(\pr_n)_{i=1}^{n-1}X_i$. We  show that $X_i$ is infinitely factorizable
 for each $0\leq i\leq n-1$. Since $X$ has an $m$-fold interleaving factorization for infinitely many $m$.
 Theorem \ref{thm:lcm-factorization} implies that $X$ has an $\lcm(m,n)$-fold interleaving factorization for infinitely many $m$. Thus, $X$ has an $ne$-fold interleaving factorization for infinitely many $e$. Moreover, for each such $e$, if $X=(\pr_{ne})_{k=0}^{ne-1}$, then the shuffle identities of Theorem \ref{thm:shuffle}, combined with uniqueness of $ne$-fold interleaving factorizations, imply that each $X_i$ has the $e$-fold interleaving factorization $X_i=(\pr_e)_{j=0}^{e-1}X_{i+jn}$.
\end{proof}

%
%
\subsection{Structure of interleaving factorizations } \label{sec:52a}


\begin{thm}\label{thm:sublattices-are-attained2}
{\rm (Converse divisibility structure  for interleaving factorizations)} 

Let $N\subseteq{\NN^{+}}$ be a nonempty set with the following properties:
\begin{enumerate}
    \item[(1)] If $n\in N$ and $d$ divides $n$, then $d\in N$. 
    \item[(2)]  If $m,n\in N$, then $\lcm(m,n)\in N$.
\end{enumerate}

If  the alphabet $\sA$ has at least two letters,  then  $N=\IC(X)$ for some $X\subseteq\sA^{\NN}$.
\end{thm}

%
%

\begin{proof}
Given a set $N$ satisfying (1), (2) we construct a set $\oX$ on   $\sA=\{0,1\}$ with $\sN(\oX) = N$. Enumerate the elements of $N$ as $n_1,n_2,\ldots,$. Let $\ell_1=n_1$, and for $i>1$, let $\ell_i=\lcm(n_1,\ldots,n_i)$. 
Notice for $i \le j$, that $\ell_j=\lcm(\ell_i,\ell_j)$, hence for any set $X\subseteq\sA^{\NN}$ we have
$X^{[\ell_i]}\subseteq \left(X^{[\ell_i]}\right)^{[\ell_j]}= X^{[\lcm(\ell_i, \ell_j)]} =X^{[\ell_j]}$, using  Theorem \ref{thm:410}.
 Thus, $\left(X^{[\ell_j]}\right)_j$ is an increasing sequence of sets.

Choose  $X=\{0^{\infty},1^{\infty}\}$. Notice that for any $n\in\NN$, $X^{[n]}$ is precisely the set of all sequences in $\sA$ that are periodic with period dividing $n$. 
Now set 
$\oX :=\lim_{j\to\infty}X^{[\ell_j]}=\bigcup_{j=1}^{\infty}X^{[\ell_j]}, $ 
so that $\oX$ is the set of sequences in $\sA$ that are periodic and have a period $p\in N$ (since $N$ is precisely the set $\{n:n|\ell_j\text{ for some }j\geq1\}$). 

{\bf Claim.} $N=\IC(\oX)$. 

  (1) We show that if $n \in N$, then  $\oX^{[n]}=\oX$. We already know $\oX\subseteq\oX^{[n]}$. Let $\bx\in\oX^{[n]}$. Then $\bx=(\pr_n)_{i=0}^{n-1}\bx_i$ for $\bx_1,\ldots,\bx_n\in\oX$. Since there are finitely many of these $\bx_i$, there is an $\ell_j$ large enough that $\bx_1,\ldots,\bx_n\in X^{[\ell_j]}$. Choose $\ell_j$ with $j$ large enough that $n|\ell_j$. Then $\lcm(n,\ell_j)=\ell_j$, so $X^{[\ell_j]}$ is closed under $n$-fold interleaving, and thus $\bx\in X^{[\ell_j]}\subseteq\oX$. Hence 
$\oX^{[n]}=\oX$, and so $n\in\IC(\oX)$.

   (2) We show that if $n\notin N$, then  $\oX^{[n]}\neq\oX$. Since $X\subseteq\oX$, we have $X^{[n]}\subseteq\oX^{[n]}$ by the extension property in Theorem \ref{thm:inclusion}. Let $\bx$ be any sequence in $\sA$ that is periodic with period $n$. Then $\bx\in X^{[n]}$, and so $\bx\in \oX^n$. However, for any $\ell_j$
we have $\ell_j\in N$ by the structure of $N$, and since $M$ is closed under divisibility, $n\notin N$ implies $n$ does not divide $\ell_j$; hence $\bx\notin X^{[\ell_j]}$. Since this is the case for all $\ell_j$, $\bx\notin\oX$, and so 
$n\notin\IC(\oX)$.
\end{proof}

%
%

\begin{rem}\label{rem:53}\label{rem:Nsh} 
The sets $\oX$ constructed in the proof  of Theorem \ref{thm:sublattices-are-attained2} are all  shift-invariant: $S\oX= \oX$. 
To show this, we note that  a  word $\bx$ on alphabet $\sA= \{0,1\}$ is in $\oX$ if and only if it is fully periodic with a minimal period $p$ 
belonging to $N\subseteq \NN^{+}$, since
$N$ is downward closed under divisibility. The word $S\bx$ is also periodic with the same period,
so $S\bx \in \oX$,  hence $S\oX \subseteq \oX$. Since $S^p\bx =\bx$, we have $\by= S^{p-1}\bx$ is periodic with the  same period , 
so $\by \in \oX$, and   $S\by= S^p \bx= \bx \in S\oX$. It follows that $S\oX= \oX$.
\end{rem}

%
%
\subsection{Divisibility for self-interleaving factorizations } \label{sec:32a}

\begin{defn}\label{def:self-interleaving}
An  $n$-fold interleaving  factorization $X = (\pr_n)_{i=0}^{n-1} X_{i,n} $ 
 is  {\em self-interleaving} (or {\em $n$-fold self-interleaving}),  if all factors 
are identical, i.e. $X_{i,n}= X_{0,n}$ holds for $1 \le i \le n-1$.  
We sometimes write $Z_n := X_{0,n}$ for the unique  factor in this case. 
\end{defn}

There exist many sets $X$  for which every interleaving factorization is  a  self-interleaving. 
We will  show later,  in   Proposition \ref{thm:shift-self-factorization}, that if $X$ is shift-invariant, then 
 $X=X^{[n]}$ implies that  $X_{0,n} = X_{i.n}$ holds for all $n \ge 1$.
In addition there  the exist  examples with   $X$  having an $n$-fold self-interleaving, so that 
$X_{0, n} = X_{in}$ for $0\le i \le n-1$, 
but with $X_{0,n} \ne X_{i, n}$ for all  $i \ge n$, see Example \ref{exmp:bad-self-interleave}. 
The latter sets  $X$ can  have a mixture of self-interleaving factorizations 
and non-self interleaving factorizations.

We show  set of values of $n$ for which a given $X$ has an  $n$-self-interleaving has 
 divisibility properties parallel to those described in Theorem \ref{thm:lcm-factorization}.


\begin{thm} \label{thm:lcm-self-factorization} 
{\rm (Structure of  self-interleaving closure sets) }\\
Let $\IC_{\self}(X) = \{ n \ge 1: \, X=(\pr_n)_{i=0}^{n-1}Z_n \text{ for some } Z_n \subseteq\sA^{\NN}\}$. Then $\IC_{\self}(X)$ is nonempty and 
has the following properties.

(1) If $n \in \IC_{\self}(X)$ and $d$ divides $n$, then $d \in \IC_{\self}(X)$.

(2) If $m,n \in \IC_{\self}(X)$ then their least common multiple $\lcm(m,n) \in \IC_{\self}(X)$.

 Conversely,  if a subset $N \subseteq \NN^{+}$ is nonempty 
and has properties (1) and (2), then there exists $X \subseteq \sA^{|nn}$ with $N = \IC(X)$. 
\end{thm}

\begin{proof}
(1) 
If $d$ divides $n$ we  have $n=de$ and now
$X= (\pr_n)_{k=0}^{de-1} Z_n$ and $Z_n=  \psi_{k, de}(X)$ for $0 \le k \le de-1$.
By the shuffle product identities in Theorem \ref{thm:shuffle}, 
$$
X=  (\pr_d)_{i=0}^{d-1} ((\pr_e)_{j=0}^{e-1}(X_{jd+i}))= (\pr_d)_{i=0}^{d-1} \left((\pr_e)_{j=0}^{e-1} Z_n \right).
$$
We deduce $X = (\pr_d)_{i=0}^{d-1} Z_d$ where $Z_d = (\pr_e)_{j=0}^{e-1} Z_n$,
so $X$ has a  $d$-fold self-interleaving.

(2) Suppose that $X$ has both an $n$-fold and an $m$-fold self-interleaving factorization.
We wish to show it has an $\lcm(m,n)$-fold  self-interleaving  factorization.
Let $d = \gcd(m,n)$, and recall that there exist $e,f$ with $e|m, f|n$ having $d=ef$ and $\gcd(\frac{m}{e}, \frac{n}{f}) =1$
(shown in the proof of Theorem \ref{thm:410}). 
By (1) the set is of self-interleaving factorizations is downward closed under divisibility, so that it has an $\frac{m}{e}$-fold self-interleaving
factorization and an $\frac{n}{f}$-fold self-interleaving factorization,
and now $\lcm(\frac{m}{e}, \frac{n}{f}) =\frac{mn}{ef} = \frac{mn}{d} = \lcm(m,n)$.
We have therefore reduced  proving  (2) 
 to  proving it in the special  case where $\gcd(m,n)=1$,
with  $\lcm(m,n) =mn$.

In this case we are given that $X$ has an $m$-fold and an $n$-fold self-interleaving factorization. 
We now have $\gcd(m,n)=1$   so by Theorem  \ref{thm:lcm-factorization} we have an $mn$-fold interleaving factorization $X= (\pr_{mn})_{k=0}^{mn-1} X_{k, mn}$. 
We wish to show it is self-interleaving, i.e. that 
\begin{equation} \label{eqn:self-int-mn} 
X_{k_1, mn} = X_{k_2, mn} \quad \mbox{ for}  \quad 0 \le k_1< k_2 \le mn-1. 
\end{equation} 
We assert  that  for each  $0 \le i \le n-1$, 
\begin{equation}\label{eqn:self-int-m}
X_{j_1+im, mn} = X_{j_2+ im, mn} \quad \mbox{for} \quad 0 \le j_1 < j_2 \le m-1.
\end{equation}
To see this, note that by the shuffle identities in Theorem \ref{thm:shuffle},
$$
 (\pr_{mn})_{k=0}^{mn-1} X_{k, mn}=(\pr_m)_{j=0}^{m-1}\Big((\pr_n)_{i=0}^{n-1}X_{j+im,mn}\Big).
$$
Since $m$-fold factorizations are unique, the right-hand side is a self-interleaving factorization, so for all $0\leq j_1<j_2\leq m-1$, $(\pr_n)_{i=0}^{n-1}X_{j_1+im,mn}=(\pr_n)_{i=0}^{n-1}X_{j_2+im,mn}$. This implies, again by uniqueness, that $X_{j_1+im,mn}=X_{j_2+im,mn}$ for all $0\leq i\leq n-1$.

Similarly, using the shuffle identity with the $n$-fold interleaving on the outside and the $m$-fold interleaving on the inside, we obtain for $0 \le i_1 < i_2 \le n-1$ that
for each $0 \le j \le m-1$, 
\begin{equation}\label{eqn:self-int-n}
X_{i_1+jn, mn} = X_{i_2+ jn, mn} \quad \mbox{for} \quad 0 \le i_1 < i_2 \le n-1.
\end{equation}

Now we assert that when $\gcd(m,n)=1$ that \eqref{eqn:self-int-m} and \eqref{eqn:self-int-n}  imply \eqref{eqn:self-int-mn}.
Now  \eqref{eqn:self-int-m} implies the equalities for all members in  consecutive values of $k$ in groups  of length $m$, 
$B_i =\{k= j + in: 0 \le j \le m-1\}$ (for fixed $i$)  but with no connection between between  blocks for different $i$. Now \eqref{eqn:self-int-n} implies
the same for blocks of consecutive values of $k$ in groups of length $n$, $C_j = \{ k= i+ jn: 0\le i \le n-1\}$ (for fixed $j$). Now the condition $\gcd(m,n)=1$
implies that for each $0 \le i \le m-2$ some group  $C_j$ includes members of  both $B_i$ and $B_{i+1}$, because the broken connection between
blocks $B_i$  is between $k\equiv m-1 (\bmod \, m)$ and  $k \equiv 0 (\bmod \, m)$, and all  multiples $jn \not\equiv 0 (\bmod \, m)$ for $1 \le j \le m-1$.   
Hence we get enough pairwise equalities to force \eqref{eqn:self-int-mn} to hold.
(This argument for equality of different $X_{k, mn}$ stops at position $mn$, where the groups $B_i$ and $C_j$ both line up to create a barrier going to higher $k$.)  This proves property (2) when $\gcd(m,n)=1$.

For the converse, it remains to show that if a subset $N \subseteq \NN^{+}$ is nonempty 
and has properties (1) and (2), then there exists $X \subseteq \sA^{\NN}$ with $N = \IC_{self}(X)$. For this, 
we use the fact that the sequences $\oX$ constructed in Theorem \ref{thm:sublattices-are-attained2}. 
that achieve $N=\IC(\oX) $ are shift-invariant, see Remark \ref{rem:53}. 
Now Proposition  \ref{thm:shift-self-factorization} (which will be proved in Section \ref{sec:30}) asserts that any shift- invariant $X$ has the 
property that  all of its interleaving factorizations will be self-interleaving factorizations. Thus, $\IC(\oX)=\IC_{self}(\oX)$. 
We have already  shown in the proof of Theorem \ref{thm:sublattices-are-attained2} that $\IC(\oX)=N$.
\end{proof}

\begin{exmp}\label{exmp:bad-self-interleave} 
For a general set $X$ the  set of $n$ giving self-interleaving factorization
 can be a strict subset of all interleaving factorizations of $X$. 
 For the binary alphabet $\sA = \{ 0,1\}$, 
take $X= \{ 00 \{0,1\}^{\NN}\} $, i.e. all infinite words beginning with $00$.  

We first show  $X$ has a  $n$-fold interleaving factorization for all $n \ge 1$, so $\IC(X)=\NN_{+}$. For $n=1$ and  $n=2$ the factorization is  self-interleaving
with $X_{0,2} = X_{1,2} = \{ 0 \{ 0,1\}^{\NN} \}$.
(Note that for $j \ge 2$ one has  $X_{j, 2} =  \{ 0, 1\}^{\NN}$.) 
%
%
In contrast we show $\IC_{\self} (X) = \{1, 2\}$ is finite.
For  $n \ge 3$ its  interleaving factorization has  $X_{0,n} = X_{1,n} = \{ 0 \{ 0,1\}^{\NN} \}$,  
while  $X_{j, n} = \{ 0, 1\}^{\NN}$ is the full shift, for all $j \ge 2$, so it is not self-interleaving.   
\end{exmp}

 %
%
\section{Infinitely factorizable closed  subsets of  $\msrA^{\NN}$
} \label{sec:4}

\begin{defn} \label{def:infinitely-factorizable}
A subset $X \subseteq \msrA^{\NN}$ is {\em infinitely factorizable}
(under interleaving) if it
has an $n$-fold interleaving factorization 
$$
X= X^{[n]} = \psi_{0,n}(X) \pr \psi_{1,n}(X) \pr \cdots \pr \psi_{n-1, n}(X)
$$
for infinitely many  $n \ge 1$. 
\end{defn}

%
\subsection{Characterization  of infinitely factorizable closed sets} \label{subsec:41}

We now characterize  infinitely factorizable closed sets $X$ by the properties given  in  Theorem \ref{thm:infinite_factor}.
Property (iii) shows there are uncountably many different infinitely factorizable closed sets when the
alphabet size  $|\sA| \ge 2$.


\begin{proof}[Proof of Theorem \ref{thm:infinite_factor}] 
We prove $(iii) \Rightarrow (ii) \Rightarrow (i) \Rightarrow (iii)$.

$(iii) \Rightarrow (ii)$. Suppose property $(iii)$ holds. Write an element of $X$ as 
$\bx =x_0 x_1x_2 \cdots$. 
Property (iii) says that the allowable symbols  in location $x_k$ of 
$\bx$ may be chosen arbitrarily while  holding all other $x_j,$ $j \ne k$ fixed.
In consequence all elements of $X_k=\psi_{k, n}(X)$ may be chosen arbitrarily from $\msrA_k$ while holding  all the other $X_{j}= \psi_{j,n}(X)$ constant.
This may be done for each value of $k$,  which
implies property $(ii)$ holds.

$(iii) \Rightarrow (ii)$. Suppose property $(iii)$ holds, and let $n\geq1$. Then, using Proposition \ref{prop:41}(2), we have,
\begin{eqnarray*}
X &=& \{\bx\in\sA^{\NN}:x_k\in\sA_k\text{ for all }k\geq0\}
\\
&=&\{\bx\in \sA^{\NN}:x_{j+kn}\in\sA_{j+kn}\text{ for all }k\geq0,~0\leq j\leq n-1\}
\\
&=&\{\bx\in\sA^{\NN}:\psi_{j,n}(\bx)\in\prod_{k=0}^{\infty}\sA_{j+kn}\text{ for all }0\leq j\leq n-1\}
\\
&=& (\pr_n)_{j=0}^{n-1}\prod_{k=0}^{\infty}\sA_{j+kn},
\end{eqnarray*}
which is an $n$-fold interleaving factorization.

$(ii) \Rightarrow (i)$. Immediate.  

$(i) \Rightarrow (iii)$.   We prove the contrapositive. Suppose
Property $(iii)$ does not hold for $X$, we are to show property $(i)$ does not hold.
Let $\msrA_k$ denote the  letters  that occur in the $k$th position of some word in $X$;
it is a finite nonempty subset of the (finite) alphabet $\msrA$. For each $k \ge0$  all letter patterns $\sA_k \times \sA_{k+1} \cdots \times \sA_{k +\ell}$
may occur for each $\ell \ge 1$, then by the assumption $\sA$ is closed, we would have  $X= \prod_{k=0}^{\infty} \msrA_k$ which has Property (iii), contradicting
our assumption.  Therefore there must exist some finite $k, \ell \ge 1$ and a
finite set of consecutive $\sA_k, \sA_{k+1}, \sA_{k+2}, ... \sA_{k+\ell}$  such that there is a block
$a_k a_{k+1} \cdots a_{k+\ell}$ with each $a_{k+i} \in \msrA_{k+i}$ for $0\le i \le \ell$  that does not occur in any element of $X$.
We call this situation  a  \emph{ $(k , \ell)$-missing-configuration}. 

To show Property (i) does not hold for this  $X$ we argue by contradiction.
If  Property (i) held for $X$, then there would exist some $n \ge k+\ell+1$ such that it had an $n$-fold interleaving factorization
$$
X= \psi_{0,n}(X) \pr \psi_{1,n}(X)\pr \cdots \pr \psi_{n-1, n}(X).
$$
Each word $\bx \in X$ has  symbol $x_{k+i}$ in position $k+i$ lying as the first symbol in a word in the $n$-decimation  
set $ \psi_{k+i,n}(X)$. We can find an infinite word, call it  $\bw (k+i)  \in X$ such that  in position
$j := k+i$ it has the symbol in position $k+i$ being  $a_i \in \msrA_{k+i}$, for $0 \le i \le \ell$
(by definition of $\msrA_{k+i}$).  For all remaining positions, $0 \le j \le n-1$,  with $j \not\in \{ k , k+1, \cdots , k+\ell\}$ we pick a word $\bw (j) \in X$ arbitrarily.  

Now the symbol sequence  
$\bw :=\pr_{j=0}^{n-1}  \psi_{j,n} (w(j))   \in  \psi_{0,n}(X) \pr \psi_{1,n}(X)\pr \cdots \pr \psi_{n-1, n}(X) $
contains the forbidden 
block $a_k a_{k+1} \cdots a_{k+\ell}$ in positions
$k$ through $k+\ell$, showing that $\bw \not\in X$,  the desired contradiction.  
\end{proof}

\begin{rem}\label{rem:64a} 
An important finiteness feature of the
 proof of Theorem \ref{thm:infinite_factor} is that it shows that  that existence of a $(k, \ell)$-missing-configuration certifies that $X$ has no 
$n$-fold interleaving factorization with $n \ge k+ \ell +1$ when $X$ is closed. 
\end{rem}

The following example shows the hypothesis of $X$ being closed set is necessary in the  statement of Theorem \ref{thm:infinite_factor}.
\begin{exmp}\label{exmp:odometer} 
{\rm (Non-closed infinitely factorizable  sets)} 
Let $X$ be the countable subset of $\sA^{\NN}$ consisting of all infinite sequences having a finite number of $1$'s. 
Then $X$ is infinitely factorizable, and all decimations $\psi_{j, n}(X)= X$  are copies of itself. It is not a closed set; its closure in $\sA^{\NN}$  is the full one-sided shift. 
It satisfies properties (i) and (ii) of Theorem \ref{thm:infinite_factor}  but  fails to satisfy  property (iii). 
(The set $X$ can be viewed as the  set of teminating binary expansions of all nonnegative  dyadic rationals $\frac{k}{2^m}$.)

The construction of   Theorem \ref{thm:sublattices-are-attained2} produces infinitely factorizable $X$ having  $\sN(X)  \subsetneq \NN^{+}$.
Such sets satisfy property (i), and do not satisfy properties (ii), (iii) of Theorem \ref{thm:infinite_factor}.
\end{exmp}

%
\subsection{Consequences of infinite factorizability}\label{subsec:42}


\begin{cor}\label{cor:42}
Let $X$ be an infinitely factorizable closed subset of $\mathcal{A}^{\NN}$. Then its  
factor set $\fF(X)$ consists of  all decimations
$\psi_{j, n}(X)$ for  $n \ge 1$ and $0 \le j \le n-1$. 
Each decimated set $\psi_{j, n}(X)$ is also infinitely factorizable. 
\end{cor}

\begin{proof} 
By property (ii) of Theorem \ref{thm:infinite_factor}  $X$ is factorizable for each $n \ge 1$,
and its  $n$-fold factors are $\psi_{j, n}(X)$ for $0 \le j \le n-1$.
Now the property (iii) is preserved under decimations of all orders, hence all $\psi_{j, n}(X)$
must be infinitely factorizable.
\end{proof}

\begin{exmp}\label{exmp:infinite-factor-set} 
(Infinitely factorizable closed subsets $X$ of $\mathcal{A}^{\NN}$
having all  decimations $\psi_{j,n}(X)$  distinct) 
For  $\sA = \{0,1\}$  define $\sA_k \subset \sA$ for $ 0 \le k <\infty$ as follows. Let 
$\sA_k=\{0\}$ for all  indices $k \in A$ with
$$
A := \{  k \ge 0:  \, 0 \le  \{ k \sqrt{2} \} < \frac{1}{2} \} \quad \mbox{where} \quad \left\{  x \right\} = x - \lfloor x \rfloor. 
$$
(No special properties other than irrationality of $\sqrt{2}$  are used).
This set of indices is aperiodic (and has natural density $\frac{1}{2}$, using Weyl's equidistribution theorem.)
 Set $\sA_k = \{ 0, 1\}$ for all  other 
  integers $k \not\in A$, which is also an aperiodic  set (of natural density $\frac{1}{2}$).  
   
   Set $X= \prod_{k=0}^{\infty} \sA_k$. By Theorem \ref{thm:infinite_factor} property (iii) it is a closed set and is infinitely factorizable, i.e.  $\IC(X)= \NN_{+}$.
   Each decimation  $\psi_{j,n}(X)$ is also an infinite product space
 of the same kind
  whose set of indices $k$ that have reduced alphabet $\{0\}$  is exactly
 $$ A(j,n)  := \{ k \ge 0: 0 \le  \left\{ (nk+j) \sqrt{2} \right\} < \frac{1}{2}\}.$$
 Each $\psi_{j,n}(X)$ is closed and infinitely factorizable.
 Consider now two distinct decimations $\psi_{j, n}(X)$ and   $\psi_{\ell,m} (X)$,
 where we may suppose $1 \le n \le m$ and $0 \le j, \ell < \infty$, with $j \ne \ell$ if $n =m$.     To show distinctness we must show $A(j,n)  \ne A(\ell, m)$. 
We  use the well known fact that for each $n \ge 1$  the sequence of fractional parts $x_k= \left\{ k (n \sqrt{2})\right\}$  ($k \ge 1$) is dense  modulo $1$. 
(In fact, since $n \sqrt{2}$ is irrational, Weyl's theorem implies that
the  sequence $x_k$ is uniformly distributed modulo $1$.)  The argument has two cases. 

{\em Case 1. $n =m$.} We write $x_k:=\{(kn+j)\sqrt2\}$, and $y_k:=\{(kn+\ell)\sqrt2\}$, where $j\neq\ell$. Now $x_k=\{a_k+\theta\}$ for all k, where $\theta=\{(\ell-j)\sqrt2\}$.
Because $\sqrt2$ is irrational, $\theta\in(0,1)$; hence there must be an open interval $(a,b)\subset[0,\frac12)$ such that $(a+\theta,b+\theta)\subset(\frac12,1]$. Since $x_k$ takes values dense in $(0,1)$, we will have infinitely many $k$ with $x_k\in(a,b)$, and thus with $y_k\in(a+\theta, b+\theta)$. Therefore, there are infinitely many $k$ with $x_k\in[0,\frac12)$ and $y_k\in(\frac12,1]$.  For these $k$, $\sA_{kn+j}=\{0\}$, while $\sA_{kn+\ell}=\{0,1\}$, so all sequences in $\psi_{j,n}(X)$ must have $k$th symbol $0$,  while $\psi_{n, \ell}(X)$ has sequences with $k$th symbol taking both values $0$ or $1$
 (i.e., $k \in A(j, n)$ but $k \not\in A(\ell, n)$ for these $k$). 
 Thus $\psi_{j, n}(X) \ne \psi_{\ell, n}(X).$

{\em Case 2. $n<m$.} We write $x_k:=\{(nk+j)\sqrt2\}$ and $y_k:=\{(mk+\ell)\sqrt2\}$.  A calculation shows that 
$y_k=\{\frac{m}{n}x_k+\theta\}$, where $\theta=\{(\ell-\frac{jm}{n})\sqrt2\}$. Again, $\theta\in(0,1)$. There is an open interval $(c,d)\subset(\frac{1}{2},1]$ such that $(c-\theta,d-\theta)\subset(0,\frac{1}{2})$. Letting $(a,b)=\frac{n}{m}(c-\theta,d-\theta)$, we see that if $x_k\in(a,b)$, then $y_k\in(c,d)$. Again, by 
positive density of $x_k$, this happens infinitely often, and so there are infinitely many $k$ with $x_k\in[0,\frac{1}{2})$ and $y_k\in(\frac{1}{2},1]$. We conclude as in Case 1 that $\psi_{j,n}(X)\neq\psi_{\ell,m}(X)$.
 
We conclude that the  interleaving factor set $\fF(X)$ consists of all principal decimations, and they are all distinct. Therefore $\fF(X)$ is  infinite.
\end{exmp} 


\begin{exmp} \label{exmp:66}
(A closed $X$ with  an infinite factor set $\fF(X)$ ) 
The set $X$ constructed in Example   \ref{exmp:infinite-factor-set}  
has infinitely many distinct decimations so its decimation set $\ffD(X)$ 
and its principal decimation set $\ffD_{prin}(X)$ are infinite. In  addition
 all principal decimations are interleaving factors, so that its factor set $\fF(X)$ is also infinite.
\end{exmp} 




\begin{cor}\label{cor:63}
The set  $\sY(\mathcal{A})$
of all infinitely factorizable closed subsets   $X \subseteq \mathcal{A}^{\NN}$ is closed under  
 $n$-fold  interleaving operations of all $n \ge 1$. That is, if $X_0, X_1, \cdots , X_{n-1} \in \sY(\mathcal{A})$, then
$$
(\pr_n)_{i=0}^{n-1} X_i = X_0 \pr X_1 \pr \cdots \pr X_{n-1} \in \sY(\mathcal{A}).
$$
\end{cor}

\begin{proof} 
The corollary  follows using the characterization of membership in $\sY(\mathcal{A})$ by property (iii) of  Theorem \ref{thm:infinite_factor}.
Property (iii)  is inherited under $n$-fold interleaving of sets $X_i$ that have it.
 \end{proof}
 
 %
%
\section{Iterated interleaving factorizations of general closed  subsets of  $\msrA^{\NN}$} \label{sec:5}

We consider iterated interleaving factorizations
for general sets $X\subseteq\sA^{\NN}$.
If a  set $X$ factors as 
$X=X_0\pr\cdots\pr X_{n-1}$, it is possible that one or more of the factors $X_{j}$ can itself be factored. However, unlike with factorizations of positive integers, for example, the further factors that appear at lower levels may not be interleaving factors of the original set $X$. We therefore name them  {\em iterated interleaving factors}. We  define an {\em iterated interleaving factorization} as follows:  
The iterated interleaving factorization of depth $0$ of a set $X$ is the equation $X=X$ (or the right hand side of such an equation).
An iterated interleaving factorization of depth $1$ is a single $n$-fold factorization $X = (Y_0 \pr Y_1 \pr \cdots \pr Y_{n-1})$ (with parentheses). 
The $Y_i$ are iterated interleaving factors of depth $1$.
An {\em iterated interleaving factorization of depth $k$} is obtained recursively from an iterated interleaving factorization of depth $k-1$, with one or more finitely factorizable sets $Y$ on the right hand side of depth $k$ being replaced by  interleaving factorizations $Y=(Y_0\pr Y_1\pr\cdots\pr Y_{n-1})$ (with parentheses), for $n \ge 2$ (allowing different $n$ for different $Y$). The  new added internal factors on the right are assigned  depth $k+1$; they are inside a nested set of $k+1$ parentheses.
 
 %
\subsection{Iterated interleaving factorization trees} \label{subsec:51}

 An iterated interleaving factorization can be visually represented by a rooted tree, as pictured in Figure \ref{fig71}.
 It has root node $X$, leaf nodes corresponding to the factors in the iterated interleaving factorization and internal
 nodes corresponding to intermediate factors.

\begin{figure}[ht]\label{fig71}
	\centering
	\Tree[.$X$ 
$X_{0,4}$
[.$X_{1,4}$
[.$Y_{0,2}$
$Z_{0,3}$
$Z_{1,3}$
$Z_{2,3}$
]
$Y_{1,2}$
]
$X_{2,4}$
$X_{3,4}$
]
\newline

\vskip 0.2in 
 \hskip 0.2in {\rm FIGURE 7.1.} Iterated interleaving tree for
$X=(X_{0,4}\pr((Z_{0,3} \pr Z_{1,3} \pr Z_{2,3})\pr Y_{1,2})\pr X_{2,4}\pr X_{3,4})$, an iterated interleaving factorization of depth 3.
\newline
\end{figure}

In our definition of iterated interleaving factorizations, each step is a finite factorization. 
If an iterated interleaving factor $Y$ at level $k$ has $n$-fold interleaving factorizations for multiple values of $n$, it is natural to choose the $n$-fold factorization with the largest $n$ because 
this factorization refines all the other possible factorizations of $Y$, 
by the divisibility properties of  $\IC(X)$ from Theorem \ref{thm:lcm-factorization}.

How should one treat infinitely factorizable factors? 
We will adopt the convention in this factorization
process that we  ``freeze" any infinitely factorizable factors encountered, and do not further factorize them. 
We do this for two reasons. 
First, for infinitely factorizable $Y$,  no  natural choice of $n$ exists
for a $n$-factorization at the next level. Secondly, all  interleaving factors of infinitely factorizable sets 
 are also infinitely factorizable by Corollary \ref{cor:infinitely-factorizable}, 
 so the factorization process would necessarily proceed forever if we did not freeze any infinitely factorizable factors.

This raises the question:  If one factorizes only finitely factorizable sets,
will  the   iterated interleaving factorization process always terminate at a finite depth? 
We show below that the answer is: there are closed $X$  where the iteration process can
go on forever.  
 %
%
\subsection{Arbitrary depth factorizations } \label{subsec:52}

We show, by  construction,   that there exist closed sets $X$ 
having iterated interleaving factorizations of all depths $k \ge 1$,
 with all factors at all depths being finitely factorizable. 
 (Thus so the ``freezing" property is never needed).

\begin{thm} \label{thm:71}
{\rm (Infinite depth interleaving  factorizations)} 
There exist uncountably many closed sets $Z_I\subseteq \sA^{\NN}$ with $\sA= \{0,1\}$,
indexed by $I \in \sA^{\NN}$,  that  possess iterated interleaving factorizations of every depth $k \ge 1$.
They each have a unique  iterated
interleaving factorization of depth $k$, for all $k \ge 1$. 
Each $Z_{I}$ has an interleaving factor set $\fF(Z_{I})$ containing at most three  elements.
There exist such $I$ for which the principal decimation set $\ffD_{\text{\em prin}}(Z_I)$ is infinite. 
\end{thm} 

\begin{proof}

Let $X_0$ and $X_1$ be two distinct closed sets in $\sA^{\NN}$ 
having trivial interleaving set $\IC(X_0) =\IC(X_1) = \{1\}.$ For definiteness consider $X_0=X_F$ the {\em Fibonacci shift} , consisting
of all words which do not have two consecutive $1$'s,  and $X_2=X_{\text{AF}}$ the {\em anti-Fibonacci shift}, 
which consists  of all one-sided infinite words 
which do not contain two consecutive $0$'s. Example \ref{basicexmp} showed $X_F$ has no $n$-fold interleaving
factorizations for $n \ge 2$, and the proof applies to $X_{\text{AF}}$. 
Given  an index set $I=i_0i_1i_2 \cdots \in \sA^{\NN}$, we define a set 
\begin{equation}
Z_I = \{ \bz \in \sA^{\NN}: \, \psi_{2^r -1, 2^{r+1}} (\bz) \in X_{i_r} \quad \mbox{for} \quad r \ge 0 \} .
\end{equation}
Let $\bz=z_0z_1z_2\cdots$.
The decimations determine the values of $z_i$ for subscripts  in arithmetic progressions.
We represent an arithmetic progression as  $\AP(a; d) = \{ n \ge 0: \, n \equiv a\, (\bmod \, d)\}$. 
 Then the values $\bz_i$ for $i \in \AP(2^r-1; 2^{r+1})$ are restricted by
$\psi_{2^r-1, 2^{r+1}}(\bz)\in X_{i_r}$. 
We first show   that  $Z_{I}$ is well-defined.  


{\bf Claim 1.}  {\em The set of arithmetic progressions $\AP(2^r -1; 2^{r+1}) $ for $r \ge 0$ form a partition of $\NN$.}\medskip

We show by induction on $r \ge 0$ that $N_m:= \sqcup_{r=0}^m \AP(2^{r}-1; 2^{r+1}) = \NN \smallsetminus \AP(2^{r+1} -1; 2^{r+1}),$
a disjoint union. 
The base case $r=0$ asserts $AP[0; 2) = \NN \smallsetminus AP(1; 2).$ The induction step uses
$AP(2^{m+1}-1,;2^{m+1}) = AP(2^{m+1}-1, 2^{m+2}) \sqcup \AP(2^{m+2}-1; 2^{m+2}).$ Finally,  the
set $N_m$ contains the interval $[0, 2^m-2]$, so the infinite set union covers $\NN$, proving  claim 1.  \medskip

{\bf Claim 2.} {\em If $I \ne J$ then $Z_{I} \ne Z_J$.} \medskip

If $I \ne J$ then some $i_r \ne j_r$. Then $\psi_{2^{r} -1, 2^{r+1}}(Z_{I}) = X_{i_r}$ and
$\psi_{2^{r} -1, 2^{r+1}} (Z_{J}) = X_{j_r}$ which are distinct since $X_1 \ne X_2$.
thus  $Z_{I} \ne Z_J$, proving claim 2.
\medskip

{\bf Claim 3.} {\em Each $Z_{I}$ is a closed set in $\sA^{\NN}$.} \medskip

It suffices to show each convergent subsequence
of elements of $Z_{I}$ has  a limit in $Z_{I}$. Convergence is $\sA^{\NN}$ is pointwise on each index separately. Suppose $ \bx_k \to \by$ in
 $\sA^{\NN}$ $(k \in \NN)$  as $k \to \infty$
with each $\bx_k \in Z_{I}$. We then have $\psi_{2^{r}-1, 2^{r+1}}(\bx_k) \to \psi_{2^{r}-1, 2^{r+1}}(\by)$ in $\sA^{\NN}$. 
 For each $r \ge 0$ we have
$ \psi_{2^{r}-1, 2^{r+1}}(\bx_k) \in X_{i_r}, $ hence  $\psi_{2^{r}-1, 2^{r+1}}(\bx_k) \to \psi_{2^{r}-1, 2^{r+1}}(\by) \in X_{i_r}$,
since $X_{i_r}$ is  a closed set. The property $\psi_{2^{r}-1, 2^{r+1}}(\by) \in X_{i_r}$ for all $r \ge 0$
certifies that $\by \in Z_{I}$, proving claim 3. \medskip

{\bf Claim 4.} {\em Each $Z_{I}$ has a  $2$-fold interleaving factorization
$$Z_I =  X_{i_0} \pr Z_{SI},$$
 where $SI = i_1i_2i_3 \cdots$ denotes  the one-sided shift of $I \in \sA^{\NN}$.} \medskip

Using Proposition \ref{prop:decimation_shuffle}  we find
$$
\psi_{2^{r}-1, 2^{r+1}}(\bz) = \psi_{0,2} \circ \underbrace{\psi_{1,2}\circ\cdots\circ\psi_{1,2}}_{\text{$r$ times}}(\bz),
$$
and one proves it  by induction on $r\ge0$. 
Letting $\bw = \psi_{1,2}(\bz)$, we have for $r \ge 1$
\begin{equation}\label{eqn:repeat}
\psi_{2^{r}-1, 2^{r+1}}(\bz) =  \psi_{0,2} \circ \underbrace{\psi_{1,2}\circ\cdots\circ\psi_{1,2}}_{\text{$r-1$ times}}(\bw)
= \psi_{2^{r-1}-1, 2^r}(\bw). 
\end{equation}
By definition 
$$
Z_{SI} =\{ \bw \in\sA^{\NN}:\, \psi_{2^r-1, 2^{r+1}}(\bw) \in X_{i_{r+1}} \,\, \mbox{for} \,\, r \ge 0\}. 
$$
Now we have, using \eqref{eqn:repeat}, 
\begin{eqnarray*}
Z_I &=& \{ \bz \in \sA^{\NN} ; \, \psi_{0,2}(\bz) \in X_{i_0} \,\, \mbox{and} \,\, \bw= \psi_{1,2}(\bz) \, \mbox{has} \,\, \psi_{2^r-1, 2^{r+1}}(\bw) \in X_{i_{r+1}}\,\, \mbox{for} \,\, r \ge 1\}\\
&=& \{ \bz \in \sA^{\NN}:\, \psi_{0,2}(\bz) \in X_{i_0} \,\, \mbox{and} \,\, \psi_{1,2}(\bz) \in Z_{SI} \} = X_{i_0} \pr Z_{SI}, 
\end{eqnarray*} 
proving  Claim 4.\medskip

At this point we  obtain an iterated interleaving factorization for $Z_I$ to arbitrary depth $k \ge 1$, by iterating the factorization given in Claim 4.
This  can be done since one factor is again of the form $Z_I$ (with a different $I$).  
Given $I$, using the notation  $Z_{0} := Z_{I}$ and $Z_{k} := Z_{S^k I}$ we have the depth $k$ factorization
$$
Z_k =  X_{i_0} \pr \left(X_{i_1} \pr \left(\cdots \left(X_{i_{k-2}} \pr \left(X_{i_{k-1}} \pr Z_{k} \right) \right) \cdots \right) \right).
$$
Figure \ref{fig72} shows a tree corresponding to such an iterated factorization after the fourth level of factoring.

 \begin{figure}[ht]\label{fig72}
	\centering
	\Tree[.$Z_0$ 
$X_{i_0}$
[.$Z_1$
$X_{i_1}$
[.$Z_2$
$X_{i_2}$
[.$Z_3$
$X_{i_3}$
[.$Z_4$
]
]
]
]
]
\newline

\vskip 0.2in 
 \hskip 0.2in {\rm FIGURE 7.2.} Iterated interleaving tree for
$Z_0=(X_{i,0} \pr(X_{i_1} \pr (X_{i_2} \pr (X_{i_3} \pr Z_4))))$.
\newline
\end{figure}

The remaining part of the proof will show this factorization tree is unique at every level $k$.
Finally a suitable choice of $I$ will lead to $Z_I$ having  infinitely many different principal decimations.\medskip

{\bf Claim 5.} {\em The interleaving closure set $\IC(Z_{I}) = \{ 1, 2\}$ with associated factor set $\fF(Z_{I}) = \{Z_I, X_{i_0}, Z_{SI}\}.$ }\medskip 

It suffices to show that $Z_{I}$ has no $n$-fold interleavings with $n \ge 3$, in view of Claim 4.  We argue
by contradiction.  Given an $n$-fold interleaving for $n \ge 3$, by 
Theorem  \ref{thm:lcm-factorization}(2),  it  would also have an  $\lcm(2, n)$-fold interleaving, and we set
$2m := \lcm(2,n)$ with $m \ge 2$. A shuffle identity from Proposition \ref{thm:shuffle} gives
$$
Z_I = (\pr_{2m})_{j=0}^{2m-1} X_{i, 2m} = \bigg((\pr_{m})_{i=0}^{m-1} X_{2i, 2m} \bigg) \hspace{2pt} \pr \hspace{2pt} \bigg((\pr_{m})_{i =0}^{m-1} X_{2i+1, 2m} \bigg). 
$$
Since $2$-fold interleaving factorizations are unique, and $Z_{I}= X_{i_0} \pr Z_{SI}$, we must have
$$
X_{i_0} = (\pr_{m})_{i=0}^{m-1} X_{2i, 2m}. 
$$
This contradicts the fact that $X_0$ and $X_1$ have no nontrivial interleaving factorizations, proving claim 5.\medskip

{\bf Claim 6.} {\em For $k \ge 1$, each $Z_{I}$ has a unique iterated interleaving factorization of depth $k$,
whose iterated interleaving factors are $X_{I_r}$ for $0 \le r \le k-1$ and $Z_{S^k}(I).$}  \medskip

This claim follows by induction on $k \ge 1$, the base case being the factorization in Claim 4. For the induction
step from $k$ to $k+1$, all but one of the leaves of the tree (iterated interleaving factors) are of form  $X_i$, which have no
non-trivial interleaving factors, and the remaining factor $Z_{J}$, with $J= S^k I$, which has only a  $2$-fold interleaving
factorization $Z_{S^k I} = X_{i_k} \pr Z_{S^{k+1}(I)}$. By updating the list of iterated interleaving factors we
complete the induction step.
This proves claim 6. \medskip

{\bf Claim 7.} {\em If $I$ is strongly aperiodic, meaning that all its shifts $S^k I $ for $k \ge 0$ are distinct,
then all the  decimations of $Z_I$ of form $\psi_{2^{r} -1, 2^{r+1}}(Z_{I})$ for $r \ge 0$  are distinct. In particular, the principal decimation set
$\ffD_{\text{\em prin}}(Z_{I})$ of $Z_{I}$ is an infinite set.}  \medskip

We have $\psi_{2^{r} -1, 2^{r+1}}(Z_{I})= Z_{S^r I}$. By Claim 2 distinct $S^I$ give distinct $Z_{S^{r}I}$. 
The strongly aperiodic assumption then makes all  $\psi_{2^{r} -1, 2^{r+1}}(Z_{I})$ distinct. They
are principal decimations, so $\ffD_{\text{\em prin}}(Z_{I})$ is infinite. This proves Claim 7. 
\end{proof} 

\begin{exmp} \label{exmp:72}
(A closed set with an infinite principal decimation set but a  finite factor set) 
Theorem \ref{thm:71}  exhibited $Z_I$ that have   infinitely many  distinct principal decimations; $\ffD_{\text{\em prin}}(Z_I)\subseteq \ffD(Z_I)$.
However Claim 5 showed the factor set $\fF(Z_I)$ is always  finite.
\end{exmp}

\begin{rem} \label{rem:73}
The sets $Z_I$ in Example \ref{exmp:72} exhibit the failure of  two finiteness properties possessed by 
all path sets studied in \cite{ALS20}.
First, interleaving factorizations of path sets $\sP$ always halt at finite depth  (under the freezing convention), while $Z_{I}$ never
does. 
Second, path sets $\sP$ always have finitely many different decimations, i.e. $\ffD(\sP)$ is finite,
while this example does not.   Example \ref{exmp:66} gave another example having infinitely  many different decimations.
\end{rem}

%
\section{Shift-stable and weakly shift-stable sets }\label{sec:30}

Classical symbolic dynamics is concerned with  properties of sets $X \subseteq \sA^{\NN}$ invariant under the shift operator. 
The  class of  such sets is not preserved  under decimation or interleaving operations.
We  study two weaker notions of  sets $X$ compatible with the shift operation---{\em shift-stable sets} and {\em weakly shift-stable sets}---
with better properties.
Shift-stable sets naturally arise in one-sided dynamics that encode initial conditions, and we show they are closed under all
decimations,  but not closed under interleaving operations. 
The  wider class of weakly shift-stable sets is  closed under all decimation and interleaving operations.

%
%
\subsection{Shift-stable sets}\label{subsec:31}

Recall from Definition \ref{defn:weakshift} that a  general set $X \subseteq \sA^{\NN}$ is {\em shift-stable} if $SX \subseteq X$,
and it  is {\em shift-invariant} if $SX=X$. These definitions allows non-closed sets.
Shift-stability is a strictly weaker condition than shift-invariance; see Example \ref{exmp:descending_chain} below.

Shift-stable and shift-invariant sets satisfy the following closure properties under decimation and interleaving closure operations: 
\begin{thm}\label{thm:312}
 Let $\sA $ be finite alphabet and let $X \subseteq \sA^{\NN}$ be a general set (not necessarily closed).

(1) If $X$ is shift-stable (resp. shift-invariant), then all decimations $\psi_{j,n}(X)$ for $j \ge 0$, $n \ge 1$ 
are shift stable (resp. shift-invariant).  

(2) If $X$ is shift-stable (resp. shift-invariant) then all $n$-fold interleaving closures $X^{[n]}$ with $n \ge 1$ are shift-stable (resp. shift-invariant). 
\end{thm} 
\begin{proof}
(1) Shift-stability of $X$ implies $S^mX \subseteq S^{m-1}X$ whence $S^m X \subseteq X$ for all $m\ge 0$.
Now  Proposition  \ref{prop:decimation_shift} gives
$$
  S \psi_{j,n}(X) = \psi_{j,n} (S^n X) \subseteq \psi_{j,n}(X).
$$
If $X$ is shift invariant,  then $S^mX=X$ for all $m \ge 0$ and equality holds. 

(2) If $X$ is shift stable, then we  have, by Proposition \ref{prop:interleaving-shift}, Proposition \ref{prop:decimation_shift}, and (1):
\begin{eqnarray*}
S X^{[n]} & = & S \big(\psi_{0,n}(X) \pr \psi_{1,n}(X)  \pr \cdots \pr \psi_{n-1,n}(X) \big)\\
&= & \psi_{1, n}(X)  \pr \psi_{2, n}(X) \pr \cdots \pr  \psi_{ n,n} (X) \\
&= & \psi_{0, n}(S X)  \pr \psi_{1, n}(S X) \pr \cdots \pr \psi_{ n-1,n} (S X) \\
&\subseteq&  \psi_{0,n}(X) \pr \psi_{1,n}(X)  \pr \cdots \pr \psi_{n-1,n}(X)  = X^{[n]}.
\end{eqnarray*}
If $X$ is shift invariant, then all steps hold with equality, as required.
\end{proof} 

The shift-invariant property restricts the form of interleaving factorizations.

\begin{prop} \label{thm:shift-self-factorization} 
{\rm (Shift invariance implies  self-interleaving) }
If a general set $X \subseteq \sA^{\NN}$ is shift-invariant, then  all  of its   interleaving factorizations will be   
self-interleaving factorizations.
\end{prop}

\begin{proof}
We have for each $n \ge1$, that for $j \ge 0$ 
$$
\psi_{j+1, n}(X) = \psi_{j, n}(SX) = \psi_{j, n}(X)
$$
with the leftmost  equality generally true by  Proposition \ref{prop:decimation_shuffle} (2) 
and the second equality from shift invariance. 
We now have
$$ 
\psi_{j, n}(X) = \psi_{0,n}(X) \quad \mbox{for} \quad  j \ge 0.
$$
But by Theorem \ref{thm:DIF} any  $n$-fold interleaving  $X = (\pr_n)_{i=0}^{n-1} X_{i,n}$ 
has $X_{i,n} = \psi_{i,n}(X)$, hence it is a self-interleaving with $Z_n= \psi_{0,n}(X)$. 
\end{proof}

%
%
\subsection{Closed shift-stable sets}\label{subsec:27a}

An important feature of 
closed shift-stable sets is that they are characterized by forbidden blocks, paralleling the definition of  
two-sided shift spaces in \cite[Sec. 1.2]{LM95}. 
Let $\sA^{\ast}$ denote  the set of all finite  words in the alphabet $\sA$, including the empty word.
A {\em block} in an infinite word $\bx= a_0a_1a_2 \cdots$ is a finite sequence of consecutive symbols
$a_k a_{k+1} \cdots a_{k+\ell}$.


\begin{prop}\label{prop:forbidden_block}
{\rm (Forbidden block characterization of shift-stability)}  
The following statements about a set $ X \subseteq \sA^{\NN}$ are equivalent.

(1) $X$ is closed and shift-stable, i.e. $X$ is closed and $S X \subseteq X$.

(2) $X$ is the set of all infinite words avoiding a (finite or infinite) set $\sB^{\perp}   \subseteq \sA^{\ast}$ of forbidden blocks. 
\end{prop}


\begin{rem}\label{rem:85}
An analogous result holds in two-sided symbolic dynamics for subsets of $\mathcal{A}^\mathbb{Z}$, 
(\cite[Theorem 6.1.21]{LM95}), 
where shift-stability is replaced by  shift invariance,  proved with a similar argument. 
The difference  between  shift-stablity  and shift-invariance is discussed in Example \ref{exmp:219a}.
\end{rem} 

\begin{proof}
(2) $\Rightarrow$ (1). The set $X$ is closed, since any limit word in the sequence topology will not contain  any forbidden block.
Now $SX$ is a closed set of infinite words, which do not contain any of the forbidden blocks. It follows that $SX \subseteq X$.

(1) $\Rightarrow$ (2). The hypothesis $S X \subseteq X$ implies $S^k X \subseteq S^{k-1} X\subseteq X$ for all $k \ge 1$ by induction on $k$.
We let $\sB^{\perp}(X) \subseteq \sA^{\ast}$  denote all the finite words that do not appear anywhere in any  word in $X$.
Let $Y$ denote the set of all infinite words that avoid any block in $\sB^{\perp}(X)$. By definition $X \subseteq Y$. To complete the
proof we show 
the reverse inclusion $Y \subseteq X$. Let $\by= b_0b_1b_2 \cdots  \in Y$. By hypothesis the  initial word $b_0 b_1 \cdots b_k \in Y$ does
not contain any element of $\sB^{\perp}(X)$ , so it must occur as a block inside some word $\bx= a_0a_1a_2 \cdots \in X$, for if it did not this
would contradict  maximality of $\sB^{\perp} (X)$.
Say it is  positions $a_j a_{j+1} \cdots a_{j+k} = b_0 b_1 \cdots b_k$. Now $\by_k:= S^j \bx = .b_0b_1 \cdots b_k a_{k+1} \cdots \in S^k X \subseteq X$.
We now have a sequence $\{ \by_k: k \ge 0\}$ with $\by_k \in X$ that converges in the sequence topology to $\by \in Y$. Since $X$ is closed, we deduce  $\by \in X$
as required.
\end{proof}

We give examples of allowed  behavior  and of non-behavior of closed shift-stable sets.

\begin{exmp}\label{exmp:descending_chain} 
There exists a shift-stable closed set $X$ which yields an  infinite strictly descending chain 
of inclusions under application of the shift; i.e.:
$$
X \supsetneqq SX \supsetneqq S^2X \supsetneqq S^3X \supsetneqq \cdots.
$$
To construct  $X$, define for each  $k \ge 4$ the set  $X_k := (0^k 1)^k \{ 000, 111\}^{\NN}.$ 
That is,   $X_k$
has a  fixed finite prefix $(0^k1)^k$ of length $k(k+1)$ followed by a full $2$-block shift  
$$
Y = \{ 000, 111\}^{\NN}.
$$ 
Note that $S^3 Y=Y$. 
We now set 
$$
X :=  \bigcup_{k=4}^{\infty} \big(\bigcup_{n=0}^{\infty} S^n X_k \big). 
$$
The set  $X$ is shift-stable, since 
$$
SX= \bigcup_{k=4}^{\infty} \big(\bigcup_{n=1}^{\infty} S^n X_k \big)  \subseteq X.
$$
Every element of $X$ is an (eroded) finite prefix followed by a member of $Y$, $SY$, or $S^2Y$. 
The set $X$ is closed because the only limit point obtainable in $\sA^{\NN}$
from repeated shifts of blocks in the finite  prefixes alone is the vector $0^{\infty}$, which already belongs to $Y$.

To show all inclusions are strict, we note 
for $0 \le j \le 3$ the set $S^{j}X$ contains the word $0^{4-j}1 (0^4 1)^3 (000)^{\infty}$, which is
not contained in any $S^mX$ for $m \ge j+1$. 
For $j\ge 4$ each set $S^jX$ contains the word $1(0^{j}1)^{j-1}(000)^{\infty}$, which is not contained in
any $S^mX$ for $m \ge j+1$. 
\end{exmp}


\begin{exmp} \label{exmp:215} 
{\rm (Shift-stability  is not  always preserved under interleavings)} 
 The one-sided { Fibonacci shift}  $X_F$  having $11$ as a  forbidden block  and the one-sided { anti-Fibonacci shift}
  $X_{\text{AF}}$ 
 having  $00$ as a forbidden block  are both closed, shift-invariant sets.
We show  their $2$-fold interleaving  $Y = X_{\text{AF}} \pr X_{F}$  is not shift-stable. Indeed $X_{\text{AF}}$ allows the initial block $0110$, and $X_{F}$ allows the initial block $010$,
whence $X_{\text{AF}} \pr X_{F}$ allows the initial block $0011100$, so $SY$ contains the initial block $011100$.  If  $SY \subseteq Y$,
then there is a $\by=\by_1\pr\by_2\in Y$ with initial block $011100$. But this means $\by_2\in X_{F}$ has initial block $110$, which is a forbidden block of the Fibonacci shift, a contradiction showing that $SY \not\subseteq Y$.
(We do  have
$S^2 Y = Y$.) 
\end{exmp} 

\begin{exmp}\label{exmp:219a} 
{\rm (One-sided shifts)} 
The notion of   {\em one-sided shift} $X$ defined by Lind and Marcus  \cite[Sect. 12.8]{LM95} 
consists of those    sets $X \subseteq \sA^{\NN}$ that are the restriction to positions $k \ge 0$
of all  sequences in a two-sided shift $X_{\pm}$ described by forbidden blocks. 
One-sided shifts  $X$ are necessarily closed and  shift-invariant: $SX=X$, so  they form a strict subclass of
closed shift-stable $X$.

The difference between one-sided shifts and closed shift-stable sets  is visible at the level of {\em minimal  forbidden blocks},
which  are forbidden blocks that do not
contain any other forbidden block as a strict sub-block.  For a one-sided shift-stable set $X$ we let  
$\sB_{\min}^{\perp} (X)$ denote its minimal forbidden block set.
For a two-sided shift $X_{\pm}$ 
we let  $\sB_{\min, \pm}^{\perp}(X_{\pm})$  denote its
minimal forbidden block set . 
Now  consider  the closed set $Y= \{ 001^{\infty}, 01^{\infty}, 1^{\infty}\}$ which has $SY= \{ 01^{\infty}, 1^{\infty} \} \subset X$,
so is  shift-stable but not shift-invariant. Here $S^2Y = \{ 1^{\infty}\}$ is shift-invariant.  It is easy to check
that $\sB_{\min}^{\perp}(Y) = \{ 1001, 101, 000\}$. 
The  two-sided shift $Y^{\pm}$ determined 
by this set of forbidden blocks is $Y^{\pm} = \{ 1^{\ZZ}\} \in \sA^{\ZZ}$, because any bi-infinite word 
that contains  a $0$ must also contain one
of the patterns $101, 1001, 000$ and so is excluded. 
However  $Y_{\pm}$
has  minimal forbidden block set $\sB_{min, \pm}^{\perp} (Y_{\pm}) = \{ 0\}$ viewed as a two-sided shift. 
The  one-sided shift $\tilde{Y}$ determined from $Y_{\pm}$, using the Lind and Marcus prescription  has
 $\tilde{Y}= S^2Y= \{ 1^{\infty} \}$. The shift-stable sets  $Y$ and $SY$ cannot be obtained
by the Lind and Marcus prescription'  their minimal forbidden block sets are not  minimal forbidden block
sets of any two-sided shift. 
\end{exmp}

%
%
\subsection{Weakly shift-stable sets}\label{subsec:28}

The  notion  of {\em  weak shift-stability}  provides a large  class of sets $ X \subseteq \sA^{\NN}$  which respect the 
shift operator and  are closed under all decimation and interleaving operators. 
 This class of sets  includes all path sets studied in \cite{AL14a}, see \cite{ALS20}.

\begin{defn}\label{def:220} 
A general set $X \subseteq \sA^{\NN}$ is {\em weakly shift-stable} if there are $\ell > k \ge 0$ such
that $S^{\ell} X \subseteq S^kX$. We call $p = \ell-k$ an {\em eventual period} for this shift semi-stable set.
\end{defn}

The notion of   {\em eventual period} of $X$  reflects the inclusion 
$$
S^{ \ell +j} X = S^{(k+j) +p} X\subseteq S^{k+j } X  \quad \mbox{ for all} \quad j \ge 0.
$$

Theorem \ref{thm:operation-closure}  shows that the  class $\sW(\sA)$ of all weakly shift-stable sets is closed under all decimation and interleaving operations:
\begin{proof}[Proof of Theorem \ref{thm:operation-closure}] 
(1) Weak shift-stability $S^{\ell}X \subseteq S^kX$ gives $S^{\ell  +j} X \subseteq  S^{k+j}X$ for all $ j \ge 0$.
Setting $p= \ell-k$, we deduce for $m \ge k$ that 
\begin{equation}\label{eqn:inclusion3} 
S^{m + jp}(X) \subseteq S^{m}X \quad  \mbox{ whenever} \quad   j \ge 1.
\end{equation} 
By  Proposition \ref{prop:decimation_shift} 
we have, for $ j \ge 0$,  $n \ge 1$, 
$$
S^{\ell p} \psi_{j,n}(X)  = \psi_{j+ \ell p n,  n}(X) =  \psi_{j, n}(S^{\ell p n}X) \subseteq  \psi_{j, n} (S^{k p n} X) = S^{kp} \psi_{j, n}(X),
$$
the inclusion holding because $S^{\ell pn} (X) \subseteq S^{kpn}(X)$ by \eqref{eqn:inclusion3}, since the difference of iterations
is a multiple of $p$ and $kpn \ge k$.

(2) Let $X_j$ be  weakly  shift-stable with parameters $(\ell_j, k_j)$, for $ 0 \le j \le n-1$, and $p_j= \ell_j- k_j$.
We assert that $Y= (\pr_n)_{i=0}^{n-1} X_i$ is weakly shift-stable with an eventual period $p= p_0p_1 \cdots p_j$.
Indeed, setting $k=\max_j (k_j)$ and $\ell=k+1$, we have, using Proposition \ref{prop:interleaving-shift}:
\begin{eqnarray*}
S^{\ell p  n} Y & = & S^{\ell p n} (X_0 \pr X_1 \pr \cdots \pr X_{n-1}) \\
&= & (S^{\ell p} X_0)  \pr (S^{\ell p}  X_1) \pr \cdots \pr (S^{\ell p} X_{n-1}) \\
&\subseteq & (S^{kp} X_0) \pr (S^{kp} X_1) \pr \cdots \pr (S^{kp} X_{n-1}) 
=   S^{kpn} X^{[n]}.\\
\end{eqnarray*}
The  third line above used the inclusions $S^{\ell p} X_i \subseteq S^{k p}X_i$ for $0 \le i \le n-1$, which follow from \eqref{eqn:inclusion3},
since $k \ge k_i$, and $p_i$ divides $p$.

(3) We have by Proposition \ref{prop:interleaving-shift} and Proposition \ref{prop:decimation_shift}:  
\begin{eqnarray*}
S^{\ell n} X^{[n]} & = & S^{\ell n} \big(\psi_{0,n}(X) \pr \psi_{1,n}(X)  \pr \cdots \pr \psi_{n-1,n}(X)\big)  \\
&=& S^{\ell}\psi_{0,n}(X) \pr S^{\ell}\psi_{1,n}(X)  \pr \cdots \pr S^{\ell}\psi_{n-1,n}(X)
\\
&= & \psi_{0, n}(S^{\ell n} X)  \pr \psi_{1, n}(S^{\ell n} X) \pr \cdots \pr \psi_{ n-1,n} (S^{\ell n} X)
\end{eqnarray*}
Now, applying the hypothesis $S^{\ell}X \subseteq S^{k}X$:
\begin{eqnarray*}
S^{\ell n} X^{[n]} &= & \psi_{0, n}(S^{\ell n} X)  \pr \psi_{1, n}(S^{\ell n} X) \pr \cdots \pr \psi_{ n-1,n} (S^{\ell n} X)
\\
&\subseteq & \psi_{0, n}(S^{kn} X)  \pr \psi_{1, n}(S^{kn} X) \pr \cdots \pr \psi_{ n-1,n} (S^{kn} X) \\
&= & S^{kn} (\psi_{0,n}(X) \pr \psi_{1,n}(X)  \pr \cdots \pr \psi_{n-1,n} (X)) \\
&=  &S^{kn} X^{[n]}.
\end{eqnarray*}
Thus $X^{[n]}$ is weakly shift-stable. 
\end{proof}

\begin{rem}\label{rem:98}
Path sets, studied in \cite{AL14a}, are closed subsets of $\sA^{\NN}$ describable as
infinite paths in graphs of finite automata. Such sets are not always shift-stable.
In \cite{ALS20} it is shown they  are always weakly shift-invariant, so they are weakly shift-stable.
\end{rem}

%
%
\section{Entropy of interleavings for general sets} \label{sec:entropy}

We study studies two  notions of entropy for general sets $X \subseteq \sA^{\NN}$,
topological entropy $H(X)$ and prefix  entropy $H_p(X)$, defined for all sets $X$, and
we also study  a notion of stable
prefix topological entropy which only certain sets $X$ possess.

%
%
\subsection{Topological entropy and prefix 
topological entropy}\label{subsec:entropy}

We  recall  two   notions of topological entropy for  general sets $X \subseteq \sA^{\NN}$,
following the paper \cite{AL14a}, given  in Definition \ref{defn:217} and Definition \ref{defn:218}(1). 

(1) The {\em topological entropy} of $X$ is
\[ 
H_{\topp}(X) := \limsup_{k \to \infty} \frac{1}{k} \log N_k(X),
\]
where $N_k(X)$ counts the number of distinct blocks of length $k$ to be found across all words $\bx \in X$.
It is defined as a limsup, but the limit always exists.

(2) The {\em prefix entropy} (or {\em path topological entropy}) of $X$  is
\[ 
H_p(X) := \limsup_{k \to \infty} \frac{1}{k} \log N_k^I{(}X) 
\]
where $N_k^I(X)$ counts the number of distinct prefix blocks
 $b_0b_1\cdots b_{k-1}$ of length $k$ found 
across  all words $\bx \in X$.

As remarked  in Section \ref{subsubsec:134} for $H_{\topp}(X)$  the $\limsup$ is always a  limit.
However the limsup is  needed in the definition of  prefix entropy, as shown by the next example.

\begin{exmp}\label{exmp:prefix-top-limsup}
(The limit of $ \frac{1}{k} \log N_k^{I}(X)$ 
 may  not exist)  
Take $X_0 =\prod_{j=0}^{\infty} \sA_j$ where $\sA_j= \{0\}$ for $0 \le j \le 3$ and, for $m \ge 1$,
\begin{enumerate}
 \item[(i)]  
  $\sA_j=\{0\}$ for 
$2^{2m} \le j \le 2^{2m+1} -1$ 
\item[(ii)]  
$\sA_j=\{0, 1\}$ for
 $2^{2m+1}\le j \le 2^{2m+2}-1$. 
\end{enumerate} 
Then $X_0$ is a closed subset of $\sA^{\NN}$
having  values $\frac{1}{k} \log N_k^{I}(X_0)$ 
that oscillate between $\frac{1}{3} \log 2$ and $\frac{2}{3} \log 2$ infinitely often as $k \to \infty$,
with minima at $k= 2^{2m+1}$ and maxima at $k=2^{2m+2}$. 
Here the $\limsup$ gives  $H_p(X) = \frac{2}{3} \log 2.$ On the other hand, property (ii) implies 
$N_k(X_0) = 2^k$ so  $H_{\topp}(X_0) = \log 2.$
\end{exmp}

Example \ref{exmp:prefix-top-limsup} shows, first,  that  $H_p(X)$ cannot in general be defined as a limit, and second, that $H_p(X)$ and $H_{\topp}(X)$ need not be equal.

\begin{prop} \label{lem:61} 
For general sets $X \subseteq \sA^{\NN}$, the following hold.

(1)  Let $\overline{X}$ denote the closure of $X$ in the natural topology on $\sA^{\NN}$. 
One has $H_{\topp}(X) = H_{\topp}(\overline{X})$ and $H_p(X) = H_p(\overline{X})$.

(2) One has 
$$
H_p(X) \le H_{\topp}(X) \le \log |\sA|,
$$
\end{prop} 

\begin{proof} 
(1) The definitions of $H_{\topp}(X)$ and $H_p(X)$ depend only on finite symbol sequences
(resp. finite intial symbol sequences) that occur in $X$. However all infinite words in
$\overline{X} \smallsetminus X$ have all finite symbol sequences (resp. finite initial
symbol sequences) occurring for some word in $X$.

(2)  The bounds  follow from
$N_k^{I} (X) \le N_k(X)  \le |\sA|^k$. 
\end{proof}

\begin{exmp}\label{exmp:64}
(Strict inequality $H_p(X) < H_{\topp}(X)$ may occur for general $X$)
Let $\sA= \{0,1\}$, and let the closed set $X$ consist of all words which, for $m \ge 1$,
\begin{enumerate}
 \item[(i)]  
  have symbol $0$ in each position 
$2^m \le k \le 2^{m+1}-m$,
\item[(ii)]  
allow arbitrary symbols $\{0, 1\}$   in positions
$2^{m+1}-(m-1) \le k \le 2^{m+1}-1$.
\end{enumerate} 
Then $N_k = 2^k$ for all $k \ge 1$, because (ii) gives arbitrarily long blocks of the
full shift, whence $H_{\topp}(X)= \log 2$. 

On the other hand, for  a given symbol position  $k$ there are at most
$ (\log_2 k)^2 $ symbol positions of type (ii), so we obtain  
 $N_k^{I}(X) \le 2^{ (\log_2 k)^2}$. It follows that  $H_p(X)=0$.
\end{exmp}

%
\subsection{Entropy and the shift operator}\label{subsec:entropy2}

The shift operator preserves both entropies $H_{\topp}(X)$ and $H_p(X)$ separately.
\begin{prop} \label{thm:65}
For general sets $X \subseteq \sA^{\NN}$ on a finite alphabet $\sA$  the following hold.

(1) The  shift operator $S$ preserves  topological entropy:
$$
H_{\topp}(SX) = H_{\topp}(X).
$$.

(2) The shift operator $S$ preserves prefix  entropy:  
$$
H_p(SX) =  H_{p}(X).
$$
\end{prop}

\begin{proof}
(1) We have, for a finite alphabet,  
$$
N_k(X) \ge N_k(SX) \ge \frac{1}{|\sA|} N_{k+1}(X), 
$$
since there are at most $|\sA|$ choices for a letter that is dropped.
Using a  limsup definition for $H_{\topp}(X)$ (although the limit always exists) we have
$$
H_{\topp}(SX) = \limsup_{k \to \infty} \frac{1}{k} \log N_k(SX) 
\le \limsup_{k \to \infty} \frac{1}{k} \log
N_k(X)  = H_{\topp}(X).
$$
On the other hand,
\begin{eqnarray*}
H_{\topp}(SX) &=& \limsup_{k \to \infty} \frac{1}{k} \log
N_k(SX) \\
& \ge&  \limsup_{k \to \infty} \left(\frac{1}{k} \log 
N_{k+1}(X)  - \frac{1}{k} \log |\sA|\right)\\
&  = & \limsup_{k \to \infty} \frac{1}{k+1} \log
N_{k+1}(X)  = H_{\topp}(X).
\end{eqnarray*} 

(2) For a finite alphabet $\sA$ we  have
\begin{equation}\label{eqn:trapped} 
N_{k+1}^{I}(X) \ge N_k^{I}(SX) \ge \frac{1}{|\sA|} N_{k+1}^{I}(X).
\end{equation}
The result $H_p(SX) = H_p(X)$ is proved similarly to (1). 
\end{proof}

%
%
\subsection{Entropy and decimations}\label{subsec:entropy-decimation}

Entropies  may change under decimation, subject to the following inequalities.  
\begin{prop} \label{thm:9111}
For general sets $X \subseteq \sA^{\NN}$ on a finite alphabet $\sA$  the following hold.
Then for all $n \ge 1$ and all $i \ge 0$,
$$
0 \le  H_{\topp}(\psi_{i,n}(X)) \le \min (n H_{\topp}(X), \log |\sA|).
$$
and
$$
0 \le  H_{p}(\psi_{i,n}(X)) \le \min (n H_{p}(X), \log |\sA|).
$$
All equalities can be attained.
\end{prop}

\begin{proof}
The lower bounds are trivial, and the upper bounds $\log |\sA|$ are trivial.  
For the upper bounds, the symbols any block of size $k$ of $\psi_{i, n}(X)$ are contained
( in successive positions with index $i\, (\bmod\, n)$) inside  a block of
length $nk$ of $X$ with the first symbol aligned, hence  $N_k(\psi_{i,n}(X)) \le N_{nk}(X)$.
We have 
\begin{eqnarray*}
H_{\topp}( \psi_{i,n}(X)) &= & \limsup_{k \to \infty} \frac{1}{k} \log N_k(\psi_{i,n}(X))\\
& \le & \limsup_{k \to \infty} \frac{1}{k} \log N_{nk}(X) \le  n \left(\limsup_{k \to \infty} \frac{1}{k}  \log N_{k}(X)\right)= n H_{\topp}(X).
\end{eqnarray*} 
For the corresponding prefix entropy upper bound we use the bound  $N_{k} ^{I} (\psi_{i,n}(X)) \le |\sA|^i N_{nk}^{I}(X)$,
obtained by containment of a prefix of length $k$ in $\psi_{i,n}(X)$ inside a prefix of $X$ of length $nk+i$. 

To show the bounds are attained, take 
the interleaved set $X= (\pr_n)_{i=0}^{n-1} X_i$ where $X_0= \sA^{\NN}$ and each $X_i = \{ 0^{\infty}\}$ for $1\le i \le n-1$.
We have $H_{\topp}(X) = H_{p}(X) = \frac{1}{n} \log |\sA|$ (by counting blocks). 
For the upper bound we have
$H_{\topp}(\psi_{0,n} (X) ) = H_{p}(\psi_{0,n}( X) = \log |\sA|$.  
For the lower bound $H_{\topp}(\psi_{1,n} (X) ) = H_{p}(\psi_{1,n}( X)) = 0$. 
\end{proof}

%
%
\subsection{Prefix  entropy upper bound for interleaving}\label{subsec:prefix_entropy_bounds}

We prove  a general upper bound for the prefix entropy of an $n$-fold interleaving in terms of the
prefix entropies of its  factors, which is Theorem  \ref{thm:entthm1}.

\begin{proof}[Proof of Theorem \ref{thm:entthm1}] 
By definition
\begin{equation}
H_{p}(X_0 \pr \cdots \pr  X_{n-1}) =
\limsup_{k\to\infty}\frac{1}{k}\log\bigg{(}N_k^I(X_0 \pr \cdots \pr X_{n-1})\bigg{)},
\end{equation}
where $N_k^I(X)$ is the number of distinct initial blocks of length $k$ occurring in the symbol sequences of $X$.
Now we partition into subsequences $\{ nk+j: k \ge 0\}$ for $0 \le j \le n-1$ to obtain:
\begin{equation*}
H_p(X_0 \pr \cdots \pr X_{n-1})= \max_{0 \le j \le n-1}
\limsup_{k\to\infty}\frac{1}{nk+j}\log\bigg{(}N_{nk+j}^I(X_0 \pr \cdots \pr X_{n-1})\bigg{)}.
\end{equation*}
Call the terms on the right side 
$$
H_{p, j}(X) :=\limsup_{k\to\infty}\frac{1}{nk+j}\log\bigg{(}N_{nk+j}^I(X_0 \pr \cdots \pr X_{n-1})\bigg{)}
$$ 
for $0 \le j \le n-1$.
The number of distinct initial $(nk+j)$-blocks in 
$X_0 \pr \cdots \pr X_{n-1}$ is simply the product of the number of distinct initial $(k+1)$-blocks in each of $X_0, X_1, \ldots, X_{j-1}$ and of the
the  distinct initial $k$-blocks in $X_j, X_{j+1}, \cdots X_{n-1}$. Thus we obtain, for a fixed $j$, $0 \le j \le n-1$, 
\begin{eqnarray*}
H_{p, j}(X) &=& \limsup_{k\to\infty}\frac{1}{nk+j}\log\bigg{(}N_{nk+j}^I(X_0 \pr \cdots \pr X_{n-1})\bigg{)}\\
&= &\limsup_{k \to\infty} \frac{1}{nk+j}\log\bigg{(}\prod_{i=0}^{j-1} N_{k+1}^{I}(X_i) \cdot \prod_{i=j}^{n-1} N_k^{I}(X_i) \bigg{)} \\
&=&\limsup_{k\to\infty}\frac{1}{nk+j}\left(\sum_{i=1}^{j-1}\log N_{k+1}^I(X_i)+\sum_{i=j}^{n-1}\log N_{k}^I(X_i)\right)
\end{eqnarray*}

By \eqref{eqn:trapped}, which applies to general sets $X\subseteq\sA^{\NN}$, each $\log N_{k+1}^I(X_i)$ differs from $\log N_k^I(X_i)$ by no more than $\log|\sA|$. Since the entire sum is divided by $nk+j$, this difference does not affect the limsup, so:
\begin{eqnarray*}
H_{p,j}(X)&=&\limsup_{k\to \infty} \frac{1}{nk+j} \sum_{i=0}^{n-1} \log N_{k}^I(X_i)
=\frac1n\limsup_{k\to \infty} \frac{1}{k} \sum_{i=0}^{n-1} \log N_{k}^I(X_i) \nonumber
\\
&\le &\frac1n\sum_{i=0}^{n-1}  \limsup_{k\to\infty}\frac{1}{k}\log N_{k}^I(X_{i})
=\frac{1}{n} \sum_{i=0}^{n-1} H_p(X_i).
\end{eqnarray*}
Thus, all the $H_{p,j}(X)$ are bounded above by $\frac1n\sum_{i=0}^{n-1}H_p(X_i)$. It follows that $H_p=\max_{0\leq j\leq n-1}H_{p,j}(X)$ obeys the same bound.
\end{proof}

\begin{exmp}\label{exmp:910}
($X$ may have full topological entropy and  zero prefix entropy) 
We  start with the  closed set $X_0$  with  alphabet $\sA= \{0,1\}$ defined in Example \ref{exmp:prefix-top-limsup}.
Let a second closed set  $X_1$ consist of all words that allow $\{0\}$ in  index positions where $X_0$ allows $\{0, 1\}$,
and allow $\{0,1\}$ in all index positions where $X_0$ allows only $\{0\}$; i.e., 
 $X_1 =\prod_{j=0}^{\infty} \sA_j'$ where $\sA_j'= \{0,1\}$ for $0 \le j \le 3$ and, for $m \ge 1$,
\begin{enumerate}
 \item[(i)]  
  $\sA_j'=\{0,1\}$ for 
$2^{2m} \le j \le 2^{2m+1} -1$ 
\item[(ii)]  
$\sA_j'=\{0\}$ for
 $2^{2m+1}\le j \le 2^{2m+2}-1$. 
\end{enumerate} 
Then $\sB_k(X_1) = \{0,1\}^k$ for all $k \ge 1$, since (ii) has arbitrarily long blocks of the full shift, whence $H_{\topp}(X)= \log 2$. 
We  have 
$H_p(X_0) = H_p(X_1) = \frac{2}{3} \log 2$, by the same calculation as  in Example \ref{exmp:prefix-top-limsup}. 
We assert  that the interleaved set $X := X_0 \pr X_1$ has
$$
H_p(X) = \frac{1}{2} \log 2 < \frac{1}{2}\big(H_p(X_0) + H_p(X_1)\big) = \frac{2}{3} \log 2 .
$$ 
To compute $H_p(X)$, note that in each pair of consecutive symbol  positions $(2j, 2j+1)$ the words in $X$ have one symbol  frozen
to be $0$ and the other symbol free to be chosen in $\{0,1\}$, where the frozen symbol is the  symbol in position $2j$ for
$2^{2m}\le j < 2^{2m+1}$  and is the symbol in position $2j+1$  for $2^{2m+1} \le j < 2^{2m+2}$.
Thus  $2^{k/2 -1} \le N_k^{I}(X)  \le 2^{k/2 + 1}$ for all $k \ge 0$,  whence  $H_p(X) =\lim_{k\to\infty}\frac1k\log N_k^I(X)= \frac{1}{2} \log 2$. 
\end{exmp}

%
%
\subsection{Stable prefix entropy and  interleaving entropy equality}\label{subsec:stable_prefix_entropy_bounds}

We study  the concept of stable prefix entropy and show its consequences for the behavior of entropy under interleaving. 
Recall from Definition \ref{defn:218} (2)    that a  set $X \subseteq \sA^{\NN}$ has  {\em stable prefix  entropy},
if the prefix entropy can defined as a limit. That  is,   the following limit exists:  
\[ 
H_p(X) := \lim_{k \to \infty} \frac{1}{k} \log N_k^{I}(X).
\]

Recall that  Theorem \ref{thm:entthm-gen} asserts 
 that  stable prefix entropy
is preserved under interleaving, and that stable prefix entropy of all the interleaving factors 
implies equality in the prefix entropy formula of Theorem \ref{thm:entthm1}.


\begin{proof}[Proof of Theorem \ref{thm:entthm-gen}]
Let $X= (\pr_n)_{i=0}^{n-1} X_i$. The inequality $H_p(X) \le  \frac{1}{n} \sum_{i=0}^{n-1} H_p(X_i)$
in Theorem \ref{thm:entthm1} arose in  interchanging a finite sum with a $\limsup$.
Using the stable prefix hypothesis 
for each $X_i$, we obtain a matching lower bound. 

By definition $H_p(X):=\limsup_{k\to\infty}\frac{1}{k} \log N_k^{I}(X)$. Let $H_p^{'} (X):=\liminf_{k\to\infty}\frac{1}{k }N_k(X)$.
It suffices to show that $H_p^{'}(X) \ge \frac{1}{n} \sum_{i=0}^{n-1} H_p(X_i)$ to conclude that $H_p^{'}(X) = H_p(X)$
has a limit which is  the desired value $\frac{1}{n} \sum_{i=0}^{n-1} H_p(X_i)$.

Partitioning into subsequences $\{nk+j:k\geq0\}$ for $0\leq j\leq n-1$ as in the proof of Theorem \ref{thm:entthm1}, we get:
\begin{equation*}
H_p^{'}(X)= \min_{0 \le j \le n-1} \left( 
\liminf_{k\to\infty}\frac{1}{nk+j}\log\bigg{(}N_{nk+j}^I(X_0 \pr \cdots \pr X_{n-1})\bigg{)}\right).
\end{equation*}
Call the right side values $H_{p,j}^{'}(X)$. We have 
\begin{eqnarray*}
H_{p,j}^{'}(X)
&\ge & \frac{1}{n}   \liminf_{k\to\infty} \left( \sum_{i=0}^{n-1}\frac{1}{k}\log N_{k}^I(X_{i}) \right). \\
&\ge&  \frac{1}{n} \sum_{i=0}^{n-1}  \liminf_{k\to\infty}\frac{1}{k}\log N_{k}^I(X_{i})\\
&= & \frac{1}{n} \sum_{i=0}^{n-1}  \lim_{k\to\infty}\frac{1}{k}\log N_{k}^I(X_{i}) = \frac{1}{n} \sum_{i=0}^{n-1} H_p(X_i),
\end{eqnarray*}
where stable prefix entropy was used in the last line. We conclude $H_p^{'}(X) \ge \frac{1}{n} \sum_{i=0}^{n-1} H_p(X_i)$.
\end{proof}

\begin{exmp}\label{rem:99}
(Stable prefix entropy  is not always preserved under decimation) 
 The set $X= X_0\pr X_1$ of Example \ref{exmp:910} has stable prefix entropy, but $X_0= \psi_{0,2}(X)$ does not, as shown in Example \ref{exmp:prefix-top-limsup}.
The set $X_1 =\psi_{1,2}(X)$   does not have stable prefix entropy by a similar analysis.
\end{exmp}

Recall that  Theorem \ref{thm:semistable_prefix_entropy0}
 asserts weak shift-stability implies both stable prefix entropy and equality of the two notions of entropy, $H_p(X)$ and $H_{\topp}(X)$.

\begin{proof}[Proof of Theorem \ref{thm:semistable_prefix_entropy0}] 
 
For any set $X$ we have $N_m^{I}(X) \le N_m(X)$.
By  hypothesis,  $S^{\ell}X \subseteq S^kX$ for some $\ell  \ge k\ge 0.$
Since $X\subseteq Y$ implies $S(X)\subseteq S(Y)$, an easy induction argument shows that 
$S^{\ell+j}X \subseteq S^{k+j}X$ holds for all $j \ge 0$. Since any block of length $m$  in $X$, starting in any position $n$, is an initial block of $S^n(X)$, we may conclude that it is an initial block of $S^{\ell'}(X)$,
for some $\ell' \le \ell$.  Consequently all  such blocks are counted among the initial blocks of $X, SX, \cdots S^{\ell-1}(X)$ of length $m$.
To each such block one can associate an initial block of length $m+\ell$ of $X$ which contains the given block in positions $\ell'$ through $\ell'+m-1$.
Any initial block of length $m+\ell$ can   be counted this way at most $\ell+1$ times, one for each  prefix $\ell'\le \ell$, so we obtain the upper bound  
$$
N_m(X) \le (\ell+1) N_{m+ \ell}^{I}(X).
$$
We then obtain the bounds
$$
N_m^{I}(X) \le N_m(X) \le (\ell+1) |\sA|^{\ell} N_m^{I}(X), 
$$
since 
$N_{m + \ell}^{I}(X) \le |\sA|^{\ell} N_m^{I}(X)$. 
It follows that
$$
\log N_{m}^{I} (X) \le \log N_m(X) \le \log N_m^{I}(X) + C, 
$$
for an absolute constant $C$. Thus
$$
\lim_{m \to \infty} \frac{1}{m} 
\left(\log N_m^{I}(X) - \log N_m(X) \right) =0.
$$
Since the limit $\lim_{m \to \infty} \frac{1}{m} \log N_m(X)$
exists for topological entropy, it must also exist  for prefix entropy, showing stability. 
Moreover, since the limits are the same, $H_p(X)= H_{\topp}(X)$. 
Finally, since  weak shift-stability is preserved under $n$-fold interleaving,
 the entropy equation  \eqref{eqn:prefix-ent-eq3}  for topological entropy follows from Theorem \ref{thm:entthm-gen}.
\end{proof}

%
%
%
\section{Concluding remarks} \label{sec:concluding}

 %
%
\subsection{General interleaving operations} \label{sec:92}

Iterated interleaving factorizations are a special case of factorizations of closed sets  $X \subseteq \sA^{\NN}$ into a
product of closed sets 
obtained  by projections onto subsets of indices $I_j \subseteq \NN$, where the index sets 
$\{ I_j: 0 \le j \le n-1\}$ form a partition of $\NN$. 
Iterated interleaving  factorizations  project onto   a partition of $\NN$ in which each $I_j$  is a complete arithmetic
progression in $\NN$. 

{\em Exact covering systems} are
partitions of $\NN$ into a finite set of disjoint complete arithmetic progressions (of various moduli).
They have been extensively studied, see \cite{EG80}, \cite{Por81} and \cite{PorS02} for surveys. 
There are interesting  necessary and sufficient conditions for a finite set of complete arithmetic progressions to be an
exact cover of $\NN$, starting with Fraenkel \cite{Fr73}, see also Beebee \cite{Beebee92} and Porubsk\'{y} and Sch\'{o}nheim  \cite{PorS03}. 
The exact covers determined by iterated interleaving are the set of {\em natural exact covering systems} 
introduced by Porubsk\'{y} \cite{Por74}, who credits the construction to an unpublished paper of Znam. 
It is known that  not all exact covers can  be obtained by iterated interleaving constructions. 
An example  due to Znam (cf. Guy \cite[Problem F14]{Guy04}) is : 
$$
 \{0 \, (\bmod \, 6); 1\, (\bmod \,10); 2 \, (\bmod \, 15);\, 3, 4, 5, 7-10, 13-16, 19, 20, 22, 23, 25-29 \, (\bmod \, 30)\}.
$$
This set of arithmetic progressions  has $\gcd(6, 10, 15) =1$, while any iterated interleaving factorization 
with an initial $n$-fold interleaving necessarily has all arithmetic progressions in any refinement having periods divisible by $n$.
The natural exact covering systems play a special role in the reversion (inversion under composition) 
of the M\"{o}bius function power series, see   Goulden et al \cite{GGRS19}.

One can introduce more general interleaving operations, which might include arbitrary exact covering systems. 
 For a  set $X \subseteq \sA^{\NN}$,  one can ask which decimations 
 $\psi_{j,n}(X)$ have the property that  $X$ can be written as a topological product $\psi_{j,n}(X)  \times Y$,
 where $Y$ is the projection of $X$ onto the set $I$ of all indices having $i \ne j \, (\bmod \, n)$? 
 Call such a decimation $ \psi_{j,n}(X)$ with this property a {\em generalized factor} of $X$. Can one
  characterize the possible sets of all generalized factors of  $X$, as $X$ varies?

 %
%
\subsection{Iterated interleaving closure operations} \label{sec:93}

One may ask for a given set $X$, what are the set of all interleaving closures of it: $\{ X^{[n]}: n \ge 1\}$. 
We  can define a filtered limit as $n \to \infty$ as follows. Letting $p_k$ denote
the $k$th prime in increasing order, we can define 
$$
X^{[\infty]} := \lim_{n_k= (p_1p_2 \cdots p_k)^k \to \infty} X^{[n_k]},
$$
where the limit exists since $X^{[n_k]} \subseteq X^{[n_{k+1}]}$ by Proposition \ref{prop:cdc}(3),
and for each $n$ one has $n$ divides $n_k$ for all sufficiently large $k$.
The set  $X^{[\infty]}$ will be infinitely factorizable. What can one say about the possible forms of $X^{[\infty]}$?

%
%
\subsection{Characterizing  closed weakly shift-stable sets} \label{sec:96}

Is there a  characterization of closed weakly shift-stable sets $X \subseteq \sA^{\NN}$  having a  parallel with
the characterization by forbidden blocks of closed shift-stable sets given in Proposition \ref{prop:forbidden_block}? 

 %
%

\appendix
\section{Interleaving operad} \label{sec:operadsec}

Operads were systematically developed  by Boardman and Vogt \cite{BoardV73} and  May \cite{May72}  and 
as a vehicle to study iterated loop spaces in stable homotopy theory.  
More recently, operads have been used by researchers in homological algebra, category theory, algebraic geometry, and mathematical physics; see \cite{Stasheff04} for a brief introduction. Interleaving  operations determine a certain kind of operad, giving an application of the operad concept to symbolic dynamics. 
In this Appendix  we  only define operads over the category of sets, although they can be defined over any symmetric monoidal category.

Non-symmetric operads (as in \cite{LV12}, \cite{Giraudo18}) are a weak version of operads which do not require equivariance under actions of symmetric groups on factors.
They provide a convenient framework  to keep track of   properties of an infinite family of $n$-ary operations under iterated composition. 

\begin{defn} \label{operaddef} A \emph{non-symmetric operad} (or {\em plain operad}) $\underline{\mathcal{O}}$ consists of a set $\underline{\mathcal{O}}(n)$ for each natural number $n$ satisfying the following conditions:
\begin{enumerate}[(a)]
\item (composition) for all positive integers $n, k_1, \ldots, k_n$, there is a composition function 
\[
\circ: \underline{\mathcal{O}}(n) \times \underline{\mathcal{O}}(k_1) \times \cdots \times \underline{\mathcal{O}}(k_n) \rightarrow \underline{\mathcal{O}}(k_1+ \cdots + k_n),
\]
written as $(f,f_1,\ldots, f_n) \mapsto f \circ (f_1,\ldots, f_n)$ for elements $f \in \mathcal{O}(n)$ and $f_i \in \underline{\mathcal{O}}(k_i)$;
\item (identity) there is an element $1 \in \underline{\mathcal{O}}(1)$, called the \emph{identity}, such that
\[
f \circ (1,\ldots, 1) = f = 1 \circ f
\]
for all $f$;
\item (associativity) there holds
\begin{align*}
f \circ (f_1 \circ &(f_{1,1}, \ldots, f_{1,k_1}), f_n \circ (f_{n,1}, \ldots, f_{n,k_n})) = \\
&= (f \circ (f_1, \ldots, f_n)) \circ (f_{1,1}, \ldots, f_{1,k_1}, \ldots, f_{n,1}, \ldots, f_{n,k_n})
\end{align*}
for all $f \in \underline{\mathcal{O}}(n)$, $f_i \in \underline{\mathcal{O}}(k_i)$ and $f_{i,j}$. 
\end{enumerate}
\end{defn}
For a non-symmetric operad $\underline{\mathcal{O}}$, we think of the elements of $\underline{\mathcal{O}}(n)$ as $n$-ary operations. An \emph{operad} is a non-symmetric operad that also possesses a 
right-action of the symmetric group $\Sigma_n$ on the set of operations of arity $n$ for each $n$, satisfying an equivariance condition, as described in the definition below. 

Following \cite{MSS02}, we use an underline to denote non-symmetric operads $\underline{\mathcal{O}}$ and remove the underline for (symmetric) operads $\mathcal{O}$.

\begin{defn} \label{operaddef2} 
An \emph{ operad} (or {\em symmetric operad}) $\mathcal{O}$ is a non-symmetric operad together with a right action of the symmetric group $\Sigma_n$ on each $\mathcal{O}(n)$ 
satisfying the following equivariance conditions for each $\sigma \in \Sigma_n$, $\tau_i \in \Sigma_{k_i}$, $f \in \mathcal{O}(n)$, and $f_i \in \mathcal{O}(k_i)$ for $1 \leq i \leq n$:
\begin{enumerate}[(A)]
\item $(f \cdot \sigma) \circ  (f_1, \ldots, f_n) = (f \circ (f_1, \ldots, f_n)) \cdot \sigma$;
\item $f \circ (f_1 \cdot \tau_1, \ldots , f_n \cdot \tau_n) = (f \circ (f_1, \ldots, f_n) \cdot (\tau_1, \ldots, \tau_n)$.
\end{enumerate}
Here the action of $\sigma$ on the right-half of (A) is defined as the action of the permutation $\widetilde{\sigma} \in \sigma_{k_1 + \cdots + k_n}$ that
 permutes consecutive blocks of length $k_1, \ldots, k_n$, respectively, according to the permutation $\sigma$.
\end{defn} 

We let $\sS(\sA)$ denote any  class of subsets of $\sA^{\NN}$ that is closed under all decimation and interleaving operations, combining $n$ sets in $\sS(\sA)$ in any
 order in any $n$-fold interleaving.
Examples of such classes include the collection $\sW(\sA)$ of all weakly shift-stable sets (Theorem \ref{thm:operation-closure}), the sub-collection $\overline{\sW}(\sA)$ of all closed
 weakly shift-stable sets (since the property of being  closed is
preserved under all decimation and interleaving operations), and the class $\sC(\sA)$ of path sets studied in \cite{AL14a}, which is shown to satisfy weak shift-stability in  \cite{ALS20}.

We first construct a non-symmetric operad $\underline{\mathcal{I}}$ such that each element of $\underline{\mathcal{I}}(n)$ is an $n$-ary operation acting on $\sS(\sA) \times \sS(\sA) \times \cdots \times \sS(\sA)$ ($n$ times). 
Although the non-symmetric operad $\underline{\mathcal{I}}$ will be built up from the $n$-fold interleaving operations, the 
resulting set $\underline{\mathcal{I}}(n)$ of operations at level $n$ will contain many more operations. 
For notational convenience, let $\pr_n$ denote the $n$-fold interleaving operation on $\mathcal{S}(\mathcal{A})$. We let $\underline{\mathcal{I}}(1) = \{\pr_1\}$, 
where of course $\pr_1 = id_{\mathcal{S}(\mathcal{A})}$ is the trivial ``$1$-fold interleaving''. Also let $\underline{\mathcal{I}}(2) = \{\pr_2\}$. However, 
it will not be sufficient for $\underline{\mathcal{I}}(3)$ to be a singleton set. Rather,
\[
\underline{\mathcal{I}}(3) = \{\pr_3, \pr_2 \circ (\pr_1, \pr_2), \pr_2 \circ (\pr_2, \pr_1)\},
\]
where, for instance,
\[
[\pr_2 \circ (\pr_1, \pr_2)](\sX_1, \sX_2, \sX_3) = \sX_1 \pr (\sX_2 \pr \sX_3)
\]
for general sets $\sX_1, \sX_2, \sX_3 \in \mathcal{S}(\mathcal{A})$. $\underline{\mathcal{I}}(n)$ for $n > 3$ is defined analogously,
 so as to satisfy the composition condition of Definition ~\ref{operaddef}. It is easy to see that $\pr_1$ serves as an identity for $\underline{\mathcal{I}}$ with respect to the various compositions, as in (b).
  Since the compositions of $\underline{\mathcal{I}}$ are genuine function composition, associativity in $\underline{\mathcal{I}}$ follows from the associativity of function composition. 
  Therefore, $\underline{\mathcal{I}}$ is a non-symmetric operad. We call $\underline{\mathcal{I}}$ the
   \emph{interleaving non-symmetric operad}, and
   refer to operations from $\underline{\mathcal{I}}$ as \emph{compound interleaving operations}. 

The non-symmetric operad $\underline{\mathcal{I}}$ can be upgraded to a symmetric operad by adding a right action of the symmetric
group permuting the interleaving factors. This requires adding additional $n$-ary operations for each $n$. 
In particular, for $\sigma \in \Sigma_n$ and an operation 
$f \in \mathcal{I}(n)$, we need to admit the operation $f \cdot \sigma$ where
 $(f \cdot \sigma)(\sX_1, \ldots, \sX_n) = f(\sX_{\sigma(1)}, \ldots, \sX_{\sigma(n)})$. 
Note that, like the interleaving operations themselves, this is also a function 
$\mathcal{S}(\mathcal{A}) \times \cdots \times \mathcal{S}(\mathcal{A}) \rightarrow \mathcal{S}(\mathcal{A})$, 
given by a (possibly compound) interleaving of some permutation of the input sets. 
Denote by $\mathcal{I}(n)$ the set of $n$-ary operations 
expanded to include the operations $f \cdot \sigma$ defined above, which permute the inputs prior to any (compound) interleaving. 
Note that we can think of an element $f \in \underline{\mathcal{I}}(n)$ as 
corresponding to $f \cdot \epsilon \in  \mathcal{I}(n)$, where $\epsilon \in \Sigma_n$ is the identity element. 
We can then extend the compositions for the $\underline{\mathcal{I}}(k)$ to
\[
\circ: \mathcal{I}(n) \times \mathcal{I}(k_1) \times \cdots \times \mathcal{I}(k_n) \rightarrow \mathcal{I}(k_1+ \cdots + k_n),
\]
by genuine function composition. 
 Then it is natural to define a right action of $\Sigma_n$ on $\mathcal{I}(n)$ by
  $(f \cdot \sigma) \cdot \tau = f \cdot (\sigma \tau)$ for $f \cdot \sigma \in \mathcal{I}(n)$ and $\tau \in \Sigma_n$.  
  Note that the equivariance conditions (A) and (B) of Definition ~\ref{operaddef2} apply generally to an action 
  permuting the inputs of genuine functions with respect to genuine composition. 
  Thus, since the $n$-ary operations in $\mathcal{I}(n)$ are genuine functions on sets and
   the compositions are function composition, these conditions hold. 
   We call the resulting (symmetric) operad the \emph{ interleaving symmetric operad} 
     and denote it by $\mathcal{I}$.


\begin{prop} 
\label{operadprop}
 Let $\mathcal{I}$ be the  interleaving symmetric operad acting on a collection of sets $\sS(\sA)$ closed under 
 all decimation and interleaving operations. 
 Then for any $f \in \mathcal{I}(n)$ and any sets $X_0, \ldots, X_{n-1} \in \mathcal{S}(\mathcal{A})$, 
 we have also $f(X_0, \ldots, X_{n-1}) \in \mathcal{S}(\mathcal{A})$. 
\end{prop}

\begin{proof} Every $f \in \mathcal{I}(n)$ is just a composition of  interleavings of various $n$-arities, 
where possibly the input sets have their order permuted. Since $\sS(\mathcal{A})$ is closed under the 
interleaving operations, it  follows that it is closed under all composition operations from $\mathcal{I}$. 
\end{proof}

Generally, we recall below the notion of an algebra over an operad. We will see that the descriptions given above
for the nonsymmetric operad $\underline{\mathcal{I}}$ and the (symmetric) operad
 $\mathcal{I}$ were really given in terms of certain algebras over those operads. 
 This approach has helped to keep the exposition concretely rooted in the examples of interest, 
 but differs from the more typical, categorical exposition.

The following definition matches  \cite[Definition 1.20]{MSS02}, restricted to operads in the category of sets. 
For a set $X$, let $\mathcal{E}nd_X(n)$ denote the set 
of all functions $X^n \rightarrow X$, and let $\mathcal{E}nd_X = \bigcup_{n=1}^\infty \mathcal{E}nd_X(n)$. 
Then $\mathcal{E}nd_X$ has the structure of an operad, 
and is called the {\em Endomorphism Operad} (of sets), see \cite[Definition 1.7]{MSS02}.


%
%

\begin{defn} \label{operadalg} Let $\mathcal{O}$ be an operad in the category of sets, 
and let $X$ be a set. An {\em $\mathcal{O}$-algebra structure} on $X$ is a morphism of operads 
$\alpha_X: \mathcal{O} \rightarrow \mathcal{E}nd_X$, that is, a family of $\Sigma_n$-equivariant morphisms 
$\alpha_X(n): \mathcal{O}(n) \rightarrow \mathcal{E}nd_X(n)$, $n \geq 1$, compatible with 
the identity, composition, and equivariance structures of $\mathcal{O}$ and $\mathcal{E}nd_X$.
\end{defn}

If we omit the equivariance structure from the above definition, then we get the notion of 
an {\em algebra over a nonsymmetric operad}.



\begin{exmp} \label{exmp:operad_alg}   (Algebras over interleaving  nonsymmetric operad $\underline{\mathcal{I}}$) 
 
The sets $\mathcal{W}(\mathcal{A})$ of all weakly shift-stable sets on the fiinite alphabet $\mathcal{A}$,
 $\overline{\mathcal{W}}(\mathcal{A})$ of all closed weakly shift-stable sets on $\mathcal{A}$, and $\mathcal{C}(A)$ of 
 path sets on $\mathcal{A}$ are all algebras over the interleaving nonsymmetric operad 
 $\underline{\mathcal{I}}$. If the set $\mathcal{S}(\mathcal{A})$
  is any of these sets, and for any $n \in \mathbb{N}$, the maps $\alpha_{\mathcal{S}(\mathcal{A})}$ of 
  Definition ~\ref{operadalg} are built up from
\[
\alpha_{\mathcal{S}(\mathcal{A})}(n)(\prn)[(X_0,\ldots,X_{n-1})]  := (\prn)_{j=0}^{n-1} X_j = X_0 \pr X_1 \pr \cdots \pr X_{n-1}
\]
by function composition, where $(X_0,\ldots, X_{n-1}) \in \mathcal{S}(\mathcal{A})^n$. 
\end{exmp}



\begin{thebibliography}{9}
\bibitem{ABL17} 
W.C. Abram, A. Bolshakov, J.C. Lagarias,
Intersection of multiplicative translates of 3-adic Cantor sets II: Two infinite families,
Exp. Math. {26} (2017) 468--489.

\bibitem{AL14a} 
W.C. Abram, J.C. Lagarias,
Path sets in one-sided symbolic dynamics,
Adv. Appl. Math. {56} (2014) 109--134.

\bibitem{AL14b} 
W.C. Abram, J.C. Lagarias,
$p$-Adic path set fractals and arithmetic,
J. Fractal Geom. {1} (2014) 45--81.

\bibitem{AL14c} 
W.C.Abram, J.C. Lagarias,
Intersections of multiplicative translates of 3-adic Cantor sets,
J. Fractal Geom. {1}  (2014) 349--390.

\bibitem{ALS20}
W.C. Abram,  J.C. Lagarias, D.J. Slonim,
Interleaving of path sets, 
 paper in  preparation. 

\bibitem{AKM65}
R. Adler, A.G. Konheim, M.H. McAndrew,
Topological entropy,
Trans. Amer. Math. Soc. {114} (1965) 309--319. 






\bibitem{Beebee92}
J.Beebee, 
Bernoulli numbers and exact covering systems,
Amer. Math. Monthly {99} (1992) 946--948.



\bibitem{BoardV73}
J.M. Boardman, R.M. Vogt,
{ Homotopy invariant algebraic structures on topological spaces,}
Lecture Notes in Math. 347, Springer, Berlin, 1973.

\bibitem{BremnerD:14}
M.R.  Bremner, V. Dotsenko,
{Algebraic Operads: An Algorithmic Companion},
CRC Press, Boca Raton, FL, 2014.

\bibitem{CFR18}
S.D. Cardell, A. Fuster-Sabatier, A.H. Ranea,
Linearity in decimation-based generators: An improved cryptanalysis of the shrinking generator,
Open Math. {16} (2018), 646--655.


\bibitem{Carlsen08}
T.M. Carlsen, 
Cuntz-Pimsner $C^{\ast}$-algebras associated with subshifts,
Internat. J. Math. {9} (2008) 47--70.

\bibitem{Carlsen13}
T.M. Carlsen,
An introduction to the $C^{\ast}$-algebra of a one-sided shift space, {Operator algebras and dynamics}, 63--88,
Springer Proc. Math. Stat. {58},
Springer, Heidelberg 2013.

\bibitem{Cottrell18}
T. Cottrell,
A study of Penon weak $n$-categories, Part 1. Monad interleaving,
Cah. Topol. G\'{e}om Diff\'{e}r. Cat\'{e}g. {59} (2018) 197--259.

\bibitem{Cuntz81}
J. Cuntz,
A class of $C^{\ast}$- algebras and topological Markov chains II. Reducible chains and
the Ext-function for $C^{\ast}$-algebras, Invent. Math. {63} (1981) 25--40.

\bibitem{CuntzK80}
J. Cuntz, W. Krieger,
A class of $C^{\ast}$- algebras and topological Markov chains,
 Invent. Math. {56} (1980) 251--268. 
 
 
 \bibitem{Da92}
 I. Daubechies,
 {Ten Lectures on Wavelets},
CBMS-NSF Regional Conf. Ser. in Appl. Math. 61,
 SIAM, Philadelphia, PA, 1992.
 
 



\bibitem{DE17}
M. Dokuchaev, R. Exel,
Partial actions and subshifts,
J. Funct. Anal. {272} (2017) 5038--5106.


\bibitem{DFLL:01}
G. Duchamp, M. Flouret, \'{E}. Laugerotte, J.G. Luque,
Direct and dual laws for automata with multiplicities,
Theoret. Comput. Sci. {267} (2001), 105--120.


\bibitem{Eilenberg74}
S. Eilenberg,
{Automata, languages and machines, Volume A},
Pure and Applied Mathematics, Vl. 59A,
Academic Press, New York, 1974.


\bibitem{Erdos79} 
P. Erd\H{o}s, 
Some unconventional problems in number theory,
Math. Mag. {52} (1979), 67--70.

\bibitem{EG80}
P. Erd\H{o}s, R.L. Graham,
Old and new problems and results in combinatorial number theory,
Monographic No. 28 de L'Enseignements Math\'{e}matique Univ. Gen\'{e}ve, 1980. 


\bibitem{Exel17}
R. Exel,
{Partial dynamical systems, Fell bundles and applications,}
Math. Surveys Monogr. vol. 224, Amer. Math. Soc. Providence, RI, 2017.

\bibitem{Fr73}
A. Fraenkel, 
A characterization of exactly covering congruences,
Discrete Math. {4} (1973) 359--366.




\bibitem{Giraudo17}
S. Giraudo,
Operads in algebraic combinatorics,
Habilitation 2017, Univsit\'{e} Paris-Est Marne-la-Vall\'{e}e , 388 pages. 
eprint: {\tt arXiv:1712.03782}. 


\bibitem{Giraudo18}
S. Giraudo,
Nonsymmetric operads in combinatorics,
Springer, Cham, 2018. 


\bibitem{Gratzer:11}
G. Gr\"{a}tzer,
{Lattice theory: Foundation},
Birkh\"{a}user, Berlin, 2011.

\bibitem{GGRS19}
I.P. Goulden, A. Granville, L.B. Richmond, J. Shallit,
Natural exact covering systems and the reversion of the M\"{o}bius series,
Ramanujan J. {50} (2019) 211--235.

\bibitem{Guy04}
R.K. Guy,
Unsolved Problems in Number Theory: Third Edition,
Springer-Verlag, New York, 2004. 



\bibitem{HBL17}
Zhiqing Hui, J. Baccou, J. Liandrat,
On the coupling of decimation operator with subdivision schemes for
multi-scale analysis,
Mathematical methods for curves and surfaces, 162--185,
Lecture Notes in Computer Science 10521, Springer, Cham 2017.

\bibitem{KMRRS:09}
D. Krieger, A. Miller, N. Rampersad, B. Ravikumar, J.O.  Shallit,
Decimations of languages and state complexity,
Theoret. Comput. Sci. {410} (2009) 2401--2409. 

\bibitem{Lagarias09} 
J.C. Lagarias, 
Ternary expansions of powers of $2$.
 J. London Math. Soc. {79} (2009) 562--588.
 


\bibitem{Leinster04}
T. Leinster,
{Higher operads, higher categories,}
London Math. Soc. Lecture Note Series, 298.
Cambridge University Press, Cambridge, 2004.

\bibitem{LM95}
D. Lind, B. Marcus,
{An Introduction to Symbolic Dynamics and Coding},
Cambridge University Press, New York, 1995. (Reprinted 1999 with corrections.)

\bibitem{LV12}
Jean-Louis  Loday, Bruno Vallette,
{Algebraic operads},
Grundlehren der mathematischen Wissenschaften, Vol. 346,
Springer-Verlag, New York 2012. 


\bibitem{MSS02}
M. Markl, S. Shnider, J. Stasheff,
{Operads in algebra, topology and physics.}
Math. Surveys Monogr. 96. American Mathematical Society,
Providence, R,I 2002. 

\bibitem{Markl08}
M. Markl,
Operads and PROPS, pp. 87--140 in: Handbook of algebra, Vol. 5,
Elsevier/North Holland, Amsterdam, 2008.

\bibitem{May72}
J.P. May,
{The Geometry of Iterated Loop Spaces}. 
{Springer Lecture Notes in Mathematics}, Vol. 271, Springer, 1972.

\bibitem{May97} 
J.P. May,
Operads, algebras and modules, 
in  {Operads: Proceedings of Renaissance Conferences (Hartford, CT/Luminy, 1995)}, 15--31, 
Contemp. Math., 202, Amer. Math. Soc., Providence, RI, 1997.





\bibitem{MH:38}
M. Morse, G. A.Hedlund,
Symbolic dynamics
Amer. J. Math. {60} (1938) 815--866.





\bibitem{PP04}
D. Perrin, J-E. Pin,
{Infinite Words: Automata, Semigroups, Logic and Games},
Elesevier, Dordrecht, 2004.

\bibitem{Por74}
S. Porubsk\'{y},
Natural exactly covering systems of congruences, 
Czechoslovak Math. J. {24} (1974) 598--606.

\bibitem{Por81}
S. Porubsk\'{y},
Results and problem on covering systems of residue classes,
Mitteilungen Math. Sem. Giessen {150} (1981) 1--85.

\bibitem{PorS02}
S. Porubsk\'{y}, J. Sch\"{o}nheim,
Covering systems of Paul Erd\H{o}s. Past, present and future, 
pp. 581--627 in: Paul Erd\H{o}s and his mathematics I. (Budapest, 1999), 
J. Bolyai Math. Soc., Budapest 2002.

\bibitem{PorS03}
S. Porubsk\'{y}, J. Sch\"{o}nheim,
Old and new necessary and sufficient conditions on $(a_i, m_i)$ in order that
$n \equiv a_i \, (\bmod \, m_i)$ be a covering system,
Math. Slovaca {53} (2003) 341--349. 

\bibitem{Rueppel86}
R.A. Rueppel,
{Analysis and Design of Stream Ciphers,}
Springer-Verlag, Berlin, 1986.


%
\bibitem{Schechter96}
E. Schechter,
{Handbook of Analysis and Its Foundations},
Academic Press, New York,  1996.
(Elsevier, Dordrecht, 1997.)


\bibitem{Stasheff04}
J. Stasheff,
What is an operad?, 
Notices Amer. Math. Soc. {51} (2004) 630--631.


\bibitem{VO89}
S.A. Vanstone, P. van Oorschot,
{An Introduction to Error Correcting Codes with Applications},
Kluwer Academic Publishers, Boston, 1989.


\end{thebibliography}
\end{document}